\documentclass[11pt]{amsart}
\usepackage{wrapfig}
\usepackage{graphicx} 
\usepackage{caption}
\usepackage{subcaption}

\usepackage{diagbox}
\usepackage{booktabs}
\usepackage{amssymb}
\usepackage{amsthm}
\usepackage{amsmath}
\usepackage{amsxtra}
\usepackage{hyperref}
\usepackage{epsfig}
\usepackage{comment}
\usepackage{enumerate}
\usepackage{enumitem}
\setlist{itemsep=4pt}
\usepackage{fullpage}
\usepackage{hyperref}
\hypersetup{colorlinks=true,urlcolor=blue,citecolor=blue,linkcolor=blue}
\usepackage{colonequals}
\usepackage{booktabs}
\usepackage{multirow}
\usepackage{numprint}
\npdecimalsign{.}
\usepackage{booktabs}

\numberwithin{equation}{subsection}

\theoremstyle{plain}
\newtheorem{thm}[equation]{Theorem}
\newtheorem{prop}[equation]{Proposition}
\newtheorem{lem}[equation]{Lemma}
\newtheorem{cor}[equation]{Corollary}
\newtheorem{conj}[equation]{Conjecture}

\theoremstyle{definition}
\newtheorem{defn}[equation]{Definition}

\newtheorem{rmk}[equation]{Remark}
\newtheorem{remark}[equation]{Remark}

\newtheorem{exm}[equation]{Example}
\newtheorem{example}[equation]{Example}

\newtheorem{ques}[equation]{Question}

\newcommand{\arb}{{\textsf{Arb}}}
\newcommand{\magma}{{\textsf{Magma}}}
\newcommand{\pari}{{\textsf{Pari/GP}}}
\newcommand{\sage}{{\textsf{SageMath}}}

\newcommand{\Q}{\mathbb{Q}}
\newcommand{\Z}{\mathbb{Z}}
\newcommand{\R}{\mathbb{R}}
\newcommand{\C}{\mathbb{C}}
\newcommand{\F}{\mathbb{F}}
\newcommand{\PP}{\mathbb{P}}
\newcommand{\TT}{\mathcal{T}}
\newcommand{\Qbar}{\Q^{\textup{al}}}

\newcommand{\frakl}{\mathfrak{l}}
\newcommand{\eps}{\varepsilon}

\newcommand{\order}{\mathcal{O}}
\newcommand{\calO}{\order}

\newcommand{\lambdanew}{\lambda_{\mathrm{new}}}

\newcommand{\defi}[1]{\textsf{#1}} 	
\newcommand{\legen}[2]{\left(\frac{#1}{#2}\right)}
\newcommand{\newformlink}[1]{\href{https://www.lmfdb.org/ModularForm/GL2/Q/holomorphic/#1}{\textsf{#1}}}
\newcommand{\dircharlink}[1]{\href{https://www.lmfdb.org/Character/Dirichlet/#1}{\textsf{#1}}}
\newcommand{\psmod}[1]{~(\textup{\text{mod}}~{#1})}

\DeclareMathOperator{\Aut}{Aut}
\DeclareMathOperator{\cond}{cond}
\DeclareMathOperator{\End}{End}
\DeclareMathOperator{\Gal}{Gal}
\DeclareMathOperator{\lcm}{lcm}
\DeclareMathOperator{\disc}{disc}

\DeclareMathOperator{\GL}{GL}
\DeclareMathOperator{\PGL}{PGL}
\DeclareMathOperator{\id}{id}
\DeclareMathOperator{\SL}{SL}

\DeclareMathOperator{\Hom}{Hom}

\DeclareMathOperator{\Frob}{Frob}
\DeclareMathOperator{\Cls}{Cls}

\DeclareMathOperator{\nrd}{nrd}

\DeclareMathOperator{\Tr}{Tr}
\DeclareMathOperator{\Nm}{Nm}

\DeclareMathOperator{\Span}{span}
\DeclareMathOperator{\Sturm}{Sturm}
\DeclareMathOperator{\ord}{ord}

\DeclareMathOperator{\impart}{Im}
\DeclareMathOperator{\ModSym}{ModSym}
\DeclareMathOperator{\Div}{Div}

\newcommand{\tx}{\tilde{x}}
\newcommand{\tl}{\tilde{\lambda}}

\newcommand{\tP}{\tilde{P}}
\newcommand{\Sk}{S_k}
\newcommand{\Eknew}[1]{E_{#1}^\mathrm{new}}
\newcommand{\Sknew}[1]{S_{#1}^\mathrm{new}}
\newcommand{\calH}{\mathcal{H}}

\newcommand{\trnew}[1]{\Tr(T_{#1}\,|\,\Sknew{k}(N,\chi))}
\newcommand{\softO}{\widetilde{O}}
\newcommand{\secref}[1]{\S\ref{#1}}

\newcommand{\weightone}{weight one}

\newenvironment{enumalph}
{\begin{enumerate}}
{\end{enumerate}}

\newenvironment{enumroman}
{\begin{enumerate}}
{\end{enumerate}}

\usepackage{xcolor}
\definecolor{pastelred}{rgb}{0.741176, 0., 0.14902}
\definecolor{darkspringgreen}{rgb}{0.09, 0.45, 0.27}
\definecolor{pastelyellow}{rgb}{0.61, 0.61, 0.0}
\definecolor{darkmidnightblue}{rgb}{0.0, 0.2, 0.4}

\DeclareMathOperator{\InnTw}{InnTw}
\DeclareMathOperator{\SelfTw}{SelfTw}

\begin{document}

\title{Computing classical modular forms}

\author{Alex J.\ Best}
\address{Department of Mathematics \& Statistics, Boston University, 111 Cummington Mall, Boston, MA 02215, USA}
\email{alexjbest@gmail.com}
\urladdr{\url{https://alexjbest.github.io/}}

\author{Jonathan Bober}
\address{School of Mathematics, University of Bristol, Bristol, BS8 1TW, UK, and the Heilbronn Institute for Mathematical Research, Bristol, UK}
\email{j.bober@bristol.ac.uk}
\urladdr{\url{https://people.maths.bris.ac.uk/~jb12407/}}

\author{Andrew R. Booker}
\address{School of Mathematics, University of Bristol, Woodland Road, Bristol, BS8 1UG, UK}
\email{andrew.booker@bristol.ac.uk}
\urladdr{\url{http://people.maths.bris.ac.uk/~maarb/}}

\author{Edgar Costa}
\address{Department of Mathematics, Massachusetts Institute of Technology, 77 Massachusetts Avenue, Cambridge, MA 02139, USA}
\email{edgarc@mit.edu}
\urladdr{\url{https://edgarcosta.org}}

\author{John Cremona}
\address{Mathematics Institute,
         University of Warwick,
         Coventry CV4 7AL,
         United Kingdom}
\email{j.e.cremona@warwick.ac.uk}

\author{Maarten Derickx}
\address{Department of Mathematics, Massachusetts Institute of Technology, 77 Massachusetts Avenue, Cambridge, MA 02139, USA}
\email{maarten@mderickx.nl}
\urladdr{\url{http://www.maartenderickx.nl/}}

\author{Min Lee}
\address{School of Mathematics, University of Bristol, Woodland Road, Bristol, BS8 1UG, UK}
\email{min.lee@bristol.ac.uk}
\urladdr{\url{https://people.maths.bris.ac.uk/~ml14850/}}

\author{David Lowry-Duda}
\address{Institute of Computational and Experimental Research in Mathematics, 121 South Main Street, Box E, 11th Floor, Providence, RI 02903, USA}
\email{david@lowryduda.com}
\urladdr{\url{https://davidlowryduda.com}}

\author{David Roe}
\address{Department of Mathematics, Massachusetts Institute of Technology, 77 Massachusetts Avenue, Cambridge, MA 02139, USA}
\email{roed@mit.edu}
\urladdr{\url{http://math.mit.edu/~roed/}}

\author{Andrew V. Sutherland}
\address{Department of Mathematics, Massachusetts Institute of Technology, 77 Massachusetts Avenue, Cambridge, MA \ 02139, USA}
\email{drew@math.mit.edu}
\urladdr{\url{http://math.mit.edu/~drew/}}

\author{John Voight}
\address{Department of Mathematics, Dartmouth College, 6188 Kemeny Hall, Hanover, NH 03755, USA}
\email{jvoight@gmail.com}
\urladdr{\url{http://www.math.dartmouth.edu/~jvoight/}}

\date{\today}

\hypersetup{pdftitle={Computing classical modular forms},
pdfauthor={}}

\begin{abstract}
We discuss practical and some theoretical aspects of computing a database of classical modular forms in the $L$-functions and Modular Forms Database (LMFDB).  
\end{abstract}

\maketitle
\tableofcontents

\section{Introduction}

\subsection{Motivation}

Databases of classical modular forms have been used for a variety of mathematical purposes and have almost a 50 year history (see \secref{sec:history}).  In this article, we report on a recent effort in this direction in the $L$-functions and Modular Forms Database (LMFDB \cite{LMFDB}, \url{https://lmfdb.org}); for more on the LMFDB, see the overview by Cremona \cite{LMFDB:Cremona}.

\subsection{Organization}

The paper is organized as follows.  In \secref{sec:history}, we begin with a short history, and we follow this in \secref{sec:chars} with a preliminary discussion of Dirichlet characters.  Next, in \secref{sec:compmf} we make more explicit what we mean by computing (spaces of) modular forms, and then in section \secref{sec:algs} we give a short overview of the many existing algorithmic approaches to computing modular forms.  We pause in \secref{sec:technical} to prove two technical results.  In \secref{sec:implementations}, we sample the available implementations and make some comparisons.  Next, in \secref{sec:issues} we discuss some computational, theoretical, and practical issues that arose in our efforts and in \secref{sec:Lfunctions} we explain how we (rigorously) computed the $L$-functions attached to modular newforms.  Turning to our main effort, in \secref{sec:overview} we provide an overview of the computations we performed, make some remarks on the data obtained, and explain some of the features of our database.  Finally, in \secref{sec:twists} and \secref{sec:weight1} we treat twists and issues specific to modular forms of weight $1$.  

As is clear from this organization, we consider the algorithmic problem of computing modular forms from a variety of perspectives, so this paper need not be read linearly.  For the convenience of readers, we draw attention  here to a number of highlights:
\begin{itemize}
    \item In  \secref{sec:history}, we survey the rather interesting history of computing databases of modular forms.
    \item In \secref{subsec:conrey}, we exhibit a labeling scheme for Dirichlet characters, due to Conrey. 
    \item In Theorem \ref{thm:thmeisen}, we record formulas for the new, old, and total dimensions of spaces of Eisenstein series of arbitrary integer weight $k \geq 2$, level, and character, obtained from work of Cohen--Oesterl\'e and Buzzard.  (Such formulas are not available for weight $k=1$.)
    \item In Corollary \ref{cor:tnonnew}, we compute an Eichler--Selberg trace formula restricted to the space of newforms; this was used by Belabas--Cohen \cite{BelabasCohen} in their implementation in \pari{}.
    \item In Tables \ref{tab:timings1} and \ref{tab:timings2}, we compare the implementations of \magma{} and \pari{}; in Table \ref{tab:hardspaces} we note some computationally challenging newspaces.
    \item In \secref{sec:LLLbasis}, we show that by writing Hecke eigenvalues in terms of an LLL-reduced basis of the Hecke order, we can drastically reduce their total size.
    \item In \secref{sec:analyticrank}, we certify analytic ranks of $L$-functions of modular forms and remark on the ranks occurring in our dataset.
    \item In \secref{sec:chowla}, we numerically verify a generalization of Chowla's conjecture for central values of non-self-dual modular form $L$-functions.  
    \item In \secref{sec:statistics}, we present statistics on our data, and in \secref{sec:interestingandextreme} we note some interesting and extreme behavior that we observed in our dataset.  
    \item In Theorems \ref{thm:innertwist} and \ref{thm:justcheckcoprime}, we exhibit simple and effectively computable criteria for rigorously certifying that a modular form has an inner twist. 
    \item In section \ref{sec:wt1cool}, we highlight some interesting and extreme behavior found among weight $1$ modular forms in our database.
\end{itemize}

\subsection{Acknowledgments}

The authors would like to thank Eran Assaf, Karim Belabas, Henri Cohen, Alan Lauder, David Loeffler, David Platt, Mark Watkins, and the anonymous referees for their comments.  This research was undertaken as part of the \emph{Simons Collaboration on Arithmetic Geometry, Number Theory, and Computation}, with the support of Simons Collaboration Grants: 546235, to Brendan Hassett, supporting Lowry-Duda; 550023, to Jennifer Balakrishnan, supporting Best; 550029, to John Voight; and 550033, to Bjorn Poonen and Andrew V.~Sutherland, supporting Costa, Derickx, and Roe.  Additional support was provided by a Programme Grant from the UK Engineering and Physical Sciences Research Council (EPSRC) \emph{LMF: L-functions and modular forms}, EPSRC reference EP/K034383/1.

\section{History} \label{sec:history}

In this section, we survey the history of computing tables of modular forms; for a broader but still computationally-oriented history, see Kilford \cite[Section 7.1]{Kilford}.

\begin{itemize}
    \item Perhaps the first systematic tabulation of modular forms was performed by Wada \cite{Wada1,Wada2}.  As early as 1971, he used the Eichler--Selberg trace formula to compute a factorization of the characteristic polynomial of the Hecke operator $T_p$ on $S_2(\Gamma_0(q),\chi)$ for $q \equiv 1 \pmod{4}$ prime where $\chi$ was either trivial or the quadratic character of conductor $q$.  The total computation time was reported to be about 300 hours on a TOSBAC-3000.  

\item The next major step was made in the famous \emph{Antwerp IV} tables \cite{AntwerpIV} (published in 1975), motivated by the study of modularity of elliptic curves.  V\'elu and Stephens--V\'elu computed all newforms in $S_2(\Gamma_0(N))$ with $N \leq 200$ using modular symbols \cite[Table 3]{AntwerpIV} and these forms were matched with isogeny classes of elliptic curves over $\Q$ found by Swinnerton-Dyer. Tingley \cite{Tingley:thesis} computed the complete splitting into Hecke eigenspaces of  $S_2(\Gamma_0(N))$ for $N \leq 300$, extending an earlier table due to Atkin.  In particular he found the dimensions of the Atkin-Lehner eigenspaces, and computed the actual eigenvalues as floating point numbers, numerically matching conjugate newforms.  By integrating differentials, he also computed elliptic curves from the newforms with integer eigenvalues.   In some cases, this computation revealed the existence of elliptic curves not previously found by search.  (According to Birch, this was the case for the elliptic curve with Antwerp label 78A and Cremona label \href{http://www.lmfdb.org/EllipticCurve/Q/78a1/}{\textsf{78a1}}; the curves in its isogeny class have rather large coefficients.) 

\item Extending the Antwerp IV tables, Cremona \cite{Cremona:mec} (first edition published in 1992) computed a database of newforms in $S_2(\Gamma_0(N))$ with rational coefficients for $N \leq 1000$, providing also a wealth of data on the corresponding (modular) elliptic curves.  In the second edition and in later computations, this data was considerably extended.  A more recent report \cite{Cremona:database} was made on the elliptic curve tables to conductor $\numprint{130000}$, later extended to conductor \numprint{500000} and rank at most $3$.  By 2016 this database had reached conductor $\numprint{400000}$, and in July 2019 Cremona and Sutherland extended it to conductor $\numprint{500000}$.  In this range there are $\numprint{2164260}$ rational newforms, and the same number of isogeny classes of elliptic curves.

\item Miyake \cite{Miyake} published some numerical tables of modular forms as appendices in his book on modular forms; these were computed using the trace formula.  These tables included dimensions of $S_k(\Gamma_0(N))$ for $k \geq 2$ even and small values of $N$, eigenvalues and characteristic polynomials of Hecke operators on $S_2(\Gamma_0(N))$ for small prime values of $N$, and Fourier coefficients of a primitive form in $S_2(\Gamma_0(N),\chi_N)$ for $N=29,37$.

\item In the 1990s, Cohen, Skoruppa, and Zagier compiled tables of eigenforms in weights $2$ through $12$, levels up to $1000$ in weight $2$ and with a smaller range in higher weight; also some tables of eigenforms with non-trivial character.  Their method followed a paper by Skoruppa and Zagier on the trace formula \cite{SkoruppaZagier}, but these tables were not published.  

\item In the early 2000s, Stein created an online modular forms database \cite{Stein:Tables}, computed primarily using a modular symbols package \cite{Stein} he implemented in \magma{} \cite{Magma} starting in the late 1990s.  The data was computed using a rack of six custom-built machines and a Sun V480; it was stored in a PostgreSQL database (more than 10 GB), and a (Python-based) web interface to the data was provided.  These tables included dimensions, characteristic polynomials, and $q$-expansions in a variety of weights and levels.  

\item Using this \magma{} implementation, Meyer \cite{Meyer1,Meyer2} computed a table of newforms for $\Gamma_0(N)$ with rational coefficients: in weight $k=2$ he went to $N \leq 3000$ and for $k=4$ to $N \leq 2000$.

\item Prior to our work, the LMFDB had a database of classical modular forms computed by  Ehlen and Str\"omberg~\cite{LMFDB10}, which used the \sage{} \cite{Sage} implementation of modular symbols.  This dataset included partial information on $S_k(\Gamma_0(N))$ for $(k,N)$ in the ranges $[2,12] \times [1,100]$ and $[2,40] \times [1,25]$, and on $S_k(\Gamma_1(N))$ in the ranges $[2,10] \times [1,50]$ and $[2,20] \times [1,16]$.  
\end{itemize}

The scope of our modular forms database includes all of the ranges mentioned above (and more), with the exception of Cremona's tables of elliptic curves; see \secref{sec:dataextent} for details.

\section{Characters} \label{sec:chars}

Our database of modular forms is organized into subspaces identified by a \defi{level} $N\in \Z_{\ge 1}$, a \defi{weight} $k\in \Z_{\ge 1}$, and a \defi{character} $\chi\colon \Z\to \C$ taking values in the cyclotomic field $\Q(\zeta_N)$.  In order to identify these subspaces and the modular forms they contain, we adopt a standard convention for identifying Dirichlet characters that is well suited to computation, the \defi{Conrey labels} recalled in \S\ref{subsec:conrey} below.  We also introduce a convention for identifying Galois orbits of Dirichlet characters that will be used to identify the \defi{newform subspaces} and \defi{newform orbits} defined in \S\ref{sec:compmf}.

\subsection{Definitions}
For $N \in \Z_{\geq 1}$, a \defi{Dirichlet character of modulus $N$} is a pair $(\chi,N)$ where $\chi\colon \Z \to \C$ is a periodic function modulo $N$ that is the extension of a group homomorphism $(\Z/N\Z)^\times \to \C^\times$ by zero (defining $\chi(n)=0$ whenever $\gcd(n,N) \neq 1$)---in particular, $\chi$ is totally multiplicative.  The \defi{degree} of a Dirichlet character $\chi$ is the degree of the cyclotomic subfield $\Q(\chi) \subseteq \C$ generated by the values of $\chi$.  

Given two Dirichlet characters $\chi,\chi'$ of moduli $N,N'$, we define their product $\chi\chi'$ to be the Dirichlet character of modulus $\lcm(N,N')$ defined by $(\chi\chi')(n)=\chi(n)\chi'(n)$.  Under this definition, the set of Dirichlet characters of a fixed modulus $N$ has the structure of a finite abelian group, with identity the \defi{principal} (or \defi{trivial}) character with $\chi(n)=1$ if $\gcd(n,N)=1$ and $\chi(n)=0$ otherwise.  The \defi{order} $\ord(\chi)$ of a Dirichlet character $\chi$ is its order in this group,  i.e., the smallest $m \in \Z_{\geq 1}$ such that $\chi^m$ is the principal character. 

Let $\chi$ be a Dirichlet character of modulus $N$.  Given a multiple $N'$ of $N$, we may \defi{induce} $\chi$ to a Dirichlet character $\chi'$ of modulus $N'$ by $\chi'(n) \colonequals \chi(n \bmod N)$ whenever $\gcd(n,N') = 1$ and $\chi'(n) = 0$ otherwise.  Consequently, there is a well-defined \emph{minimal} modulus $M \colonequals \cond(\chi) \mid N$, called the \defi{conductor} of $\chi$, such that  $\chi$ is induced from a Dirichlet character of modulus $M$.
If $\cond(\chi)=N$, i.e., the conductor of $\chi$ is equal to its modulus, then we say that $\chi$ is a \defi{primitive} character.  

It is sometimes convenient to think about Dirichlet characters \emph{without} a modulus, remembering only a periodic, totally multiplicative arithmetic function $\chi$.  In our context, Dirichlet characters arise from modular forms with level structure, so there should be little chance for confusion.

\subsection{Conrey labels}\label{subsec:conrey}

We briefly describe a scheme, due to Brian Conrey, for labeling and computing with Dirichlet characters. 
Our labeling scheme can be thought of as a choice of an explicit isomorphism between two finite abelian groups: the multiplicative group $(\Z/N\Z)^\times$ and the group of Dirichlet characters modulo $N$.  In particular, our Dirichlet characters by definition take values in the complex numbers, so implicit in our choice of labels is a choice of embedding $\Q^\mathrm{ab} \hookrightarrow \C$.  

For each $N \in \Z_{\ge1}$, we will construct a function 
\begin{equation}
\chi_N \colon (\Z/N\Z)^\times \times (\Z/N\Z)^\times \to \C^\times
\end{equation}
satisfying the following three properties:
\begin{itemize}
\item $\chi_N$ is multiplicative in each variable (separately);
\item $\chi_N$ is symmetric (i.e., $\chi_N(m,n)=\chi_N(n, m)$ for all $m,n \in (\Z/N\Z)^\times$); and
\item $\chi_N$ is nondegenerate (i.e., if $\chi_N(m,n)=1$ for all $m \in (\Z/N\Z)^\times$, then $n \equiv 1 \pmod{N}$).
\end{itemize}
Moreover, $\chi_N$ will be multiplicative in $N$, and hence it is sufficient to define it for prime powers $p^e$ and then extend $\chi_N(m, n)$ to general $N$ by multiplicativity:
\[\chi_N(m, n)=\prod_{\substack{p^e \| N}} \chi_{p^e} (m, n).\]
We use the notation $p^e \| N$ to mean that $p^e \mid N$ but $p^{e+1} \nmid N$.
On the left side, $m$ and $n$ denote elements of $(\Z/N\Z)^\times$, while on the right they denote the images of these in $(\Z/p^e\Z)^\times$.
We then extend $\chi_N$ to a multiplicative, periodic function on $\Z\times \Z$ by setting $\chi_N(m, n)=0$ whenever $\gcd(mn,N)>1$.

Under these conditions, fixing one input to $\chi_N$ defines a Dirichlet character modulo $N$ and conversely every Dirichlet character arises in this way.  Thus each Dirichlet character is given a unique name of the form $\chi_N(m, \cdot)$ for $m\in (\Z/N\Z)^\times$. 
In particular, by symmetry, we see that $\chi_N(1, \cdot)$ is the trivial character modulo~$N$, and $\chi_N(m, \cdot)$ is a quadratic character when $m\not\equiv1\pmod{N}$ but $m^2\equiv 1\pmod{N}$.  (More generally, the order of the character~$\chi_N(m,\cdot)$ is the multiplicative order of $m$ modulo~$N$.)

We now describe the construction of $\chi_N$. 
\medskip

\noindent
\textbf{Odd prime powers}: Let $p$ be an odd prime.  Let $g$ be the smallest positive integer that is a primitive root mod $p^e$ for all $e \geq 1$. (This is almost always the same as the smallest primitive root mod $p$, but may not be; the only odd prime under one million for which these differ is 40487.)
For $m \in (\Z/p^e\Z)^\times$, we define $\log_g(m) \in \Z/\phi(p^e)\Z$ by the condition
\begin{equation} 
m\equiv g^{\log_g(m)} \pmod{p^e}, 
\end{equation}
so that $\log_g \colon (\Z/p^e\Z)^\times \to \Z/\phi(p^e)\Z$ is an isomorphism of groups.

For $m,n \in (\Z/p^e\Z)^\times$, we then define
\begin{equation}
\chi_{p^e}(m, n) \colonequals \exp\left(2\pi i \frac{\log_g(m)\log_g(n)}{\varphi(p^e)}\right).
\end{equation}
Then $\chi_{p^e}$ clearly satisfies the three required conditions (multiplicative, symmetric, and nondegenerate).

\smallskip

\noindent
\textbf{Powers of 2}:  We define $\chi_2$ to be the trivial map (so $\chi_2(1,1)=1$), and define
\begin{equation}
\chi_4(m,n) = (-1)^{(m-1)(n-1)/2}
\end{equation}
for $m,n \in (\Z/4\Z)^\times$.  Let $e \geq 3$.  The group $(\Z/2^e\Z)^\times$ is generated by $5$ and $-1$.
For $m \in (\Z/2^e\Z)^\times$, we define $\epsilon(m)\in \{0, 1\}$ and $\log_5(m)\in \Z/2^{e-2}\Z$ by 
\begin{equation}
m\equiv (-1)^{\epsilon(m)} 5^{\log_5(m)}\pmod{2^e}
\end{equation}
so that now $(\epsilon,\log_5) \colon (\Z/2^e\Z)^\times \to \Z/2\Z \times \Z/2^{e-2}\Z$ is an isomorphism.  For $m,n \in (\Z/2^e\Z)^\times$, we then define
\begin{equation}
\chi_{2^e}(m, n) \colonequals \exp\left(2\pi i \frac{\epsilon(m)\epsilon(n)}{2} + 2\pi i \frac{\log_5(m)\log_5(n)}{2^{e-2}}\right).
\end{equation}
As for the case of odd prime power modulus, this function satisfies the required properties.

In this article, as in the LMFDB, the Conrey label of the character $\chi_N(m,\cdot)$ has the form \textsf{N.m}. For example, the Conrey label of $\chi_7(6,\cdot)$, the unique quadratic character of modulus 7, is \href{http://www.lmfdb.org/Character/Dirichlet/7/6}{\textsf{7.6}}.

\subsection{Orbit labels}

There is an action of the absolute Galois group $\Gal_\Q \colonequals \Gal(\Qbar\,|\,\Q)$ of $\Q$ on the set of Dirichlet characters of modulus $N$, defined by
\begin{equation}
    (\sigma \chi)(n) \colonequals \sigma(\chi(n)) 
\end{equation}
for $\sigma \in \Gal_\Q$ and $n \in \Z$.  

It is natural to organize characters by Galois orbits, and indeed we will also want to work with modular forms defined without an embedding into the complex numbers, specified up to the action of Galois (see \secref{sec:galoisdigress}).  So we also assign an \defi{orbit label} to each Galois orbit of Dirichlet characters, as follows.  To choose this label we lexicographically order the sequences
\[
    \ord(\chi), \Tr{\chi(1)}, \Tr{\chi(2)}, \Tr{\chi(3)}, \Tr{\chi(4)}, \ldots
\]
of integers, where $\Tr \colon \Q(\chi) \to \Q$ is the absolute trace; we then
assign the label written in base $26$ using the letters of the alphabet, so
\begin{center}
\textsf{a}, \textsf{b}, \dots, \textsf{z}, \textsf{ba}, \textsf{bb}, \dots, \textsf{bz}, \textsf{ca}, \dots, \textsf{zz}, \textsf{baa}, \dots.
\end{center}

For every modulus $N\ge 1$, the Dirichlet character orbit \textsf{N.a} is the trivial character, since it is the unique character with (smallest) order $1$.

\begin{exm}
The table below lists the Conrey labels of the eight Dirichlet characters of modulus 20, their values on the generators $11$ and $17$ of $(\Z/20\Z)^\times$, their orders, the absolute traces of their values the first five positive integers coprime to 20 (note $\Tr(\chi(n))=0$ if $\gcd(20,n)\ne 1$), and the labels of the six Galois orbits in which they lie.
\smallskip

\begin{center}
\small
\setlength{\tabcolsep}{4.5pt}
\begin{tabular}{lccccccccl}
Conrey label  & $\chi(11)$ & $\chi(17)$ & $\ord(\chi)$ & $\Tr(\chi(1))$ & $\Tr(\chi(3))$ & $\Tr(\chi(7))$ & $\Tr(\chi(11))$ & $\Tr(\chi(13))$& orbit label\\\midrule
\href{https://www.lmfdb.org/Character/Dirichlet/20/1}{\textsf{20.1\phantom{1}}} & 1 & 1 & 1 & 1 & 1 & 1 & 1 & 1 & \href{https://www.lmfdb.org/Character/Dirichlet/20/a}{\textsf{20.a}} \\
\href{https://www.lmfdb.org/Character/Dirichlet/20/11}{\textsf{20.11}} & -1 & 1 & 2 & 1 & -1 & -1 & 1 & -1 & \href{https://www.lmfdb.org/Character/Dirichlet/20/b}{\texttt{20.b}} \\
\href{https://www.lmfdb.org/Character/Dirichlet/20/9}{\textsf{20.9\phantom{1}}} & 1 & -1 & 2 & 1 & -1 & -1 & 1 & 1 & \href{https://www.lmfdb.org/Character/Dirichlet/20/c}{\textsf{20.c}} \\
\href{https://www.lmfdb.org/Character/Dirichlet/20/19}{\textsf{20.19}} & -1 & -1 & 2 & 1 & 1 & 1 & 1 & -1 & \href{https://www.lmfdb.org/Character/Dirichlet/20/d}{\texttt{20.d}} \\
\href{https://www.lmfdb.org/Character/Dirichlet/20/3}{\textsf{20.3\phantom{1}}} & -1 & $-i$ & 4 & 2 & 0 & 0 & -2 & -2 & \href{https://www.lmfdb.org/Character/Dirichlet/20/e}{\textsf{20.e}} \\
\href{https://www.lmfdb.org/Character/Dirichlet/20/7}{\textsf{20.7\phantom{1}}} & -1 & $i$ & 4 & 2 & 0 & 0 & -2 & -2 & \href{https://www.lmfdb.org/Character/Dirichlet/20/e}{\textsf{20.e}} \\
\href{https://www.lmfdb.org/Character/Dirichlet/20/13}{\textsf{20.13}} & 1 & $-i$ & 4 & 2 & 0 & 0 & -2 & 2 & \href{https://www.lmfdb.org/Character/Dirichlet/20/f}{\textsf{20.f}} \\
\href{https://www.lmfdb.org/Character/Dirichlet/20/17}{\textsf{20.17}} & 1 & $i$ & 4 & 2 & 0 & 0 & -2 & 2 & \href{https://www.lmfdb.org/Character/Dirichlet/20/f}{\textsf{20.f}} \\\bottomrule
\end{tabular}
\end{center}
\smallskip

\end{exm}

\begin{remark}
The field $\Q(\chi)$ is contained in the coefficient field $\Q(f)$ of a newform $f$ with character $\chi$.  When the dimension of $\Q(f)$ is large it may be difficult to compute a complex embedding $\Q(f)\to\C$, and we often need to distinguish embeddings that are compatible with the Hecke action, which means we must know the image of $\Q(\chi)$ under embeddings of $\Q(f)$.  Matching up roots of unity of large order can be surprisingly nontrivial!  So when computing the coefficient field (as an abstract field, not necessarily embedded in the complex numbers), we compute the values of $\chi$ on generators for $(\Z/N\Z)^\times$ as elements of the coefficient field.  In this way, we may organize embeddings of the coefficient field according to a desired embedding of $\Q(\chi)$.

We could instead keep track of the coefficient field as an extension of $\Q(\chi)$, but that approach creates headaches when comparing results across implementations, it shifts the problem to a different place when working with forms in a Galois orbit, and it does not allow us to represent eigenvalues in terms of a nice LLL-reduced basis (see \secref{sec:LLLbasis}). 
\end{remark}

\section{Computing modular forms} \label{sec:compmf}

In this section, we make precise what it means to \emph{compute modular forms}.  For background, we refer to the wealth of references available, for example Cohen--Str\"omberg \cite{CohenStroemberg}, Diamond--Shurman \cite{DiamondShurman}, Serre \cite[Chapter VII]{Serre:course}, and Stein \cite{Stein}.

\subsection{Setup}

The group $\SL_2(\R)$ acts (on the left) by linear fractional transformations on the upper half-plane $\calH \colonequals \{z \in \C : \impart z > 0\}$.  For $N \in \Z_{\geq 1}$, define the congruence subgroups
\begin{equation}
    \begin{aligned}
        \Gamma_0(N) &\colonequals \left\{ \gamma \in \SL_2(\Z) : \gamma \equiv \begin{pmatrix} * & * \\ 0 & * \end{pmatrix} \psmod{N} \right\}\text, \\
    \Gamma_1(N) &\colonequals \left\{ \gamma \in \SL_2(\Z) : \gamma \equiv \begin{pmatrix} 1 & * \\ 0 & 1 \end{pmatrix} \psmod{N} \right\}.
    \end{aligned}
\end{equation}
For $\Gamma \leq \SL_2(\Z)$ a congruence subgroup, the quotient $Y(\Gamma) \colonequals \Gamma \backslash \calH$ can be compactified to $X(\Gamma)$ by adding finitely many cusps, identified with the orbits of $\Gamma$ on $\PP^1(\Q)$.  As usual, we write $X_0(N),X_1(N)$ for the quotients $X(\Gamma)$ with $\Gamma=\Gamma_0(N),\Gamma_1(N)$.

For $k,N \in \Z_{\geq 1}$, a \defi{modular form of weight $k$ and level $N$} is a holomorphic function $f \colon \mathcal{H} \to \C$ that is bounded in vertical strips and satisfies
\begin{equation} 
f\left(\frac{az+b}{cz+d}\right)=(cz+d)^k f(z) 
\end{equation}
for all $\gamma=\begin{pmatrix} a & b \\ c & d \end{pmatrix} \in \Gamma_1(N)$; the $\C$-vector space of such forms is denoted $M_k(\Gamma_1(N))$.  

Modular forms are organized by character, as follows.  The space $M_k(\Gamma_1(N))$ decomposes according to the action of diamond operators as 
\begin{equation} \label{eqn:Mk1chi}
M_k(\Gamma_1(N)) = \bigoplus_{\chi} M_k(\Gamma_0(N),\chi),
\end{equation}
the sum being over all Dirichlet characters $\chi \colon \Z/N\Z \to \C$ of modulus $N$, where $M_k(\Gamma_0(N),\chi)$ is the subspace of modular forms \defi{with (Nebentypus) character $\chi$} consisting of those forms $f$ satisfying
\begin{equation} 
f\left(\frac{az+b}{cz+d}\right)=\chi(d) (cz+d)^k f(z) 
\end{equation}
for all $\gamma \in \Gamma_0(N)$.  Throughout, we will abbreviate $M_k(\Gamma_0(N),\chi)$ to $M_k(N,\chi)$ and when $\chi$ is trivial, write simply $M_k(N)$.  

In order to handle character values with some finesse (as explained above in \secref{sec:chars} and below in \secref{sec:galoisdigress}), we work in the absolute situation (relative to $\Q$) and consider the entire Galois orbit $[\chi]$ of $\chi$, and so we write
\begin{equation} \label{eqn:Mkdecompchi}
    M_k(\Gamma_0(N),[\chi])\colonequals \bigoplus_{\chi' \in [\chi]} M_k(\Gamma_0(N),\chi'),
\end{equation}
so that from \eqref{eqn:Mk1chi} we have
\[ M_k(\Gamma_1(N)) = \bigoplus_{[\chi]} M_k(\Gamma_0(N),[\chi]), \]
where the direct sum is over Galois orbits of characters $[\chi]$.  We similarly abbreviate $M_k(\Gamma_0(N),[\chi])$ to just $M_k(N,[\chi])$.

Every such modular form $f$ has a $q$-expansion (i.e., Fourier expansion at $\infty$)
\begin{equation} \label{eqn:qexp}
    f(z) = \sum_{n=0}^{\infty} a_n q^n \in \C[[q]],
\end{equation}
where $q=\exp(2\pi i z)$ and $z \in \calH$.  We call $a_n \in \C$ the \defi{coefficients} of $f$, and we write $\Z[\{a_n\}_n]$ for the \defi{coefficient ring} and $\Q(\{a_n\}_n)$ for the \defi{coefficient field} of $f$, the subring and subfield of $\C$ generated by its coefficients, respectively. 

A modular form $f$ is a \defi{cusp form} if $f$ vanishes at the cusps of $X_1(N)$.  The subspace of cusp forms is denoted $S_k(\Gamma_1(N)) \subseteq M_k(\Gamma_1(N))$, and similarly $S_k(\Gamma_0(N),\chi) \subseteq M_k(\Gamma_0(N),\chi)$.  In particular, a cusp form vanishes at the cusp $\infty$, so that the coefficient $a_0$ of its $q$-expansion is zero.

The Petersson inner product provides an orthogonal decomposition
\begin{equation} 
M_k(\Gamma_0(N),\chi) = S_k(\Gamma_0(N),\chi) \oplus E_k(\Gamma_0(N),\chi)
\end{equation}
where $E_k(\Gamma_0(N),\chi)$ is the space spanned by Eisenstein series, obtained in an explicit way using characters (see \secref{sub:eisenstein}).  Each of the spaces above can further be decomposed into old and new subspaces, and we denote the new subspace by $\Sknew{k}(\Gamma_1(N))$, etc. 

The above spaces can be equipped with an action of \emph{Hecke operators} $T_n$ indexed by $n \in \Z_{\geq 1}$.  The operators $T_n$ are normal and pairwise commute for $\gcd(n,N)=1$, so there is a common normalized ($a_1=1$) basis for the action of the Hecke operators, called \defi{eigenforms}; for such forms, $T_n f = a_n f$ for $f$ as in \eqref{eqn:qexp}.  A new cuspidal eigenform is called an \defi{(embedded) newform}.  The coefficients of a newform are algebraic integers and the coefficient field is a number field.  When $\chi$ is trivial, this coefficient field is totally real. When $\chi$ is trivial, we also have Atkin--Lehner involutions $W_p$ for $p \mid N$, and the Fricke involution $W_N \colonequals \prod_{p \mid N} W_p$.  (See subsection~\ref{SS:AtkinLehner} below.)  

For a subring $A \subseteq \C$, we write $M_k(\Gamma_1(N);A) \subseteq M_k(\Gamma_1(N))$ for the $A$-submodule of modular forms whose $q$-expansions have coefficients in $A$, and similarly with the other decorated spaces.  

From now on, we suppose we are given the input of a weight $k \in \Z_{\geq 1}$, a level $N \in \Z_{\geq 1}$, and an orbit of Dirichlet characters $\chi$ of modulus $N$ and orbit label $\textsf{N.s}$; we encode this data of a space of modular forms in the label \textsf{N.k.s}.  

\begin{example}
For $N=280$, $k=2$, and trivial character $\chi$ having label \textsf{280.a}, the space $M_2(280) =M_2(\Gamma_0(280))$ has label
\href{http://www.lmfdb.org/ModularForm/GL2/Q/holomorphic/280/2/a/}{\textsf{280.2.a}}.
\end{example}

\begin{rmk}
We restrict ourselves to integral weight forms in this article.  For forms of half-integral weight, the algorithms, applications, and issues that arise are quite different.  
\end{rmk}

\subsection{Galois digression} \label{sec:galoisdigress}

As is usual in Galois theory, it is convenient to work both with abstract objects as well as embedded objects.  To this end, we call the $\Aut(\C)$-orbit of an embedded newform $f$ a \defi{newform orbit}, and write $[f]$ for this orbit.  We call a $\Q$-subspace of $\Sknew{k}(\Gamma_0(N),[\chi];\Q)$ that is irreducible under the action of the Hecke operators a \defi{newform subspace}.  

For an eigenform $f$ in a newform subspace, we obtain an embedded newform by a choice of embedding of its coefficient field into $\C$, and all such embeddings are conjugate under $\Aut(\C)$.  Conversely, given an embedded newform $f \in \Sknew{k}(\Gamma_0(N),\chi)$, the $\C$-subspace of $S_k(\Gamma_0(N),[\chi])$ spanned by $\sigma(f) \colonequals \sum_{n} \sigma(a_n(f)) q^n$ for $\sigma \in \Aut(\C)$ descends to a newform subspace $V_f \subseteq S_k(\Gamma_0(N),[\chi];\Q)$, visibly depending only on the $\Aut(\C)$-orbit of $f$.  In other words, there is a bijection between newform subspaces $V \subseteq \Sknew{k}(\Gamma_1(N))$ and newform orbits $[f]$ of embedded newforms $f$ of weight $k$ and level $N$.  

The coefficient field $K$ of a newform subspace, defined to be the coefficient field of any eigenform in the subspace, is well-defined as an abstract number field.  The expansion \eqref{eqn:qexp} considered in $K$, is similarly well-defined.  

\subsection{Dimensions}

The first thing one may ask to compute for a space of modular forms is just dimensions of the subspaces as defined above: the total dimension $\dim_\C M_k(\Gamma_0(N),[\chi])$, the dimension of the Eisenstein subspace $\dim_\C E_k(\Gamma_0(N),[\chi])$, and the dimension of the cuspidal subspace $\dim_\C S_k(\Gamma_0(N),[\chi])$, as well as the old and new subspaces of each of these.  Since these subspaces are naturally vector spaces over $\Q$, we have
\[ \dim_\C M_k(\Gamma_0(N),[\chi]) = \dim_\Q M_k(\Gamma_0(N),[\chi];\Q); \]
moreover, an individual space $M_k(\Gamma_0(N),\chi)$ is a vector space over $\Q(\chi)$ and each summand in \eqref{eqn:Mkdecompchi} has the same dimension, so these absolute dimensions are the product of their relative dimension by the degree $d=[\Q(\chi):\Q]$ of $\chi$, i.e., we also have
\[ \dim_\C M_k(\Gamma_0(N),[\chi]) = \dim_{\Q(\chi)} M_k(\Gamma_0(N),\chi;\Q). \]

\begin{rmk}
To avoid errors, to compare across packages, and to store data conveniently, we found it essential to compute in the absolute setting (over $\Q$) rather than the relative setting (over $\Q(\chi)$).
\end{rmk}

For weight $k \geq 2$, these dimensions can be computed using the valence formula, the Riemann--Roch theorem, or the trace formula---they are given explicitly e.g.\ by Cohen--Str\"omberg \cite[Theorem 7.4.1]{CohenStroemberg}.  Unfortunately, no formula is known for these dimensions when $k=1$.

Because they can be understood explicitly in terms of Dirichlet characters, there are separately given formulas for the Eisenstein dimension as well as the dimension of the new and old subspaces in all weights $k \geq 1$: see Cohen--Str\"omberg \cite[Propositions 8.5.15 and 8.5.21]{CohenStroemberg} for the full dimension, with the new dimension worked out by Buzzard \cite{Buzzard} using a formula of Cohen--Oesterl\'e \cite[Theorem 1]{CohenOesterle} as follows.  Lacking a reference for these formulas, we record them here.  

For $r,s,p \in \Z$ with $p$ prime and $r > 0$ and $s \leq r$, define
\begin{equation}
    \lambda(r,s,p) \colonequals 
    \begin{cases}
    p^{r'} + p^{r'-1}, & \text{ if $2s \leq r=2r'$;} \\
    2p^{r'}, & \text{ if $2s \leq r=2r'+1$;} \\
    2p^{r-s}, & \text{ if $2s>r$;}
    \end{cases}
\end{equation}
and
\begin{equation}
    \lambdanew(r,s,p) \colonequals 
    \left\{
    \begin{aligned}&
    \left.
\begin{aligned}
   & 2                 \\
   & 2p-4              \\
   & 2(p-1)^2p^{r-s-2} \\
\end{aligned}\right\} && \text{ if } 2s > r \text{ and } \begin{cases}
 r=s\text;      \\
 r=s+1\text;      \\
 r \geq s+2\text; \\
\end{cases}\\ &
    \left.
\begin{aligned}
   & 0                 \\
   & p-3               \\
   & (p-2)(p-1)p^{s-2} \\
\end{aligned}\right\} && \text{ if } 2s = r \text{ and }\begin{cases}
p=2\text;        \\
r=2\text{ and }p \geq 3\text; \\
r \geq 4\text;           \\
\end{cases}\\ &
    \left.
\begin{aligned}
   & 0                \\
   & p-2              \\
   & (p-1)^2p^{r/2-2} \\
\end{aligned}\right\} && \text{ if } 2s < r \text{ and } \begin{cases}
2 \nmid r\text;              \\
r=2\text;                        \\
r \geq 4\text{ and }2 \mid r\text.\\
\end{cases}
    \end{aligned}\right.
\end{equation}

\begin{thm}[Cohen--Oesterl\'e, Buzzard] \label{thm:thmeisen}
Let $N,k \in \Z_{\geq 1}$ and let $\chi$ be a character of modulus $N$ and conductor $M \mid N$.  Then the following statements hold:
\begin{enumalph}
\item If $\chi(-1) \neq (-1)^k$, then $\dim_\C E_k(N,\chi)=\dim_\C \Eknew{k}(N,\chi)=0$.
\item For $N=1$, we have 
\begin{equation}
\dim_\C E_k(1)=\dim_\C \Eknew{k}(1)=\begin{cases} 
1, & \text{ if $k \geq 4$ and $2 \mid k$;} \\
0, & \text{ otherwise.}
\end{cases}
\end{equation}
\end{enumalph}

Suppose further that $N \geq 2$ and $\chi(-1)=(-1)^k$, and let
\begin{equation} \label{eqn:eek}
    \begin{aligned}
e &\colonequals \prod_{p \mid N} \lambda(\ord_p(N),\ord_p(M),p) \\
e_{\textup{new}} &\colonequals \prod_{p \mid N} \lambdanew(\ord_p(N),\ord_p(M),p). \\
\end{aligned}
\end{equation} 
Then the following hold:
\begin{enumalph} \addtocounter{enumi}{2}
\item We have
\begin{equation}
     \dim_\C E_k(N,\chi) = 
     \begin{cases}
     e-1 & \text{ if $k=2$ and $\chi$ is trivial;} \\
     e/2 & \text{ if $k=1$;} \\
     e & \text{ otherwise.}
     \end{cases}
\end{equation}
\item We have
\begin{equation}
     \dim_\C \Eknew{k}(N,\chi) = 
     \begin{cases}
     e_{\textup{new}}+1 & \text{ if $k=2$ and $\chi$ is trivial and $N$ is prime;} \\
     e_{\textup{new}}/2 & \text{ if $k=1$;} \\
     e_{\textup{new}} & \text{ otherwise.}
     \end{cases}
\end{equation}
\item We have 
\[ \dim_\C E_k(N,[\chi]) = d \dim_\C E_k(N,\chi) \] 
where $d=[\Q(\chi):\Q]$ is the degree of $\chi$, and similarly with $\dim_\C \Eknew{k}(N,[\chi])$.
\end{enumalph}
\end{thm}

\begin{proof}
The proof is an elaborate and rather tedious exercise in counting characters using the trace formula.
\end{proof}

We organize this dimension data in a table, as follows.

\begin{example}
We consider the space $M_3(560,[\chi])$ with label \href{http://www.lmfdb.org/ModularForm/GL2/Q/holomorphic/560/3/bt/}{\textsf{560.3.bt}}; a character $\chi$ in this orbit has label \textsf{560.bt}, order $6$, and degree $2$.  We then compute dimensions as in Table \ref{tab:dims}.
\begin{equation} \label{tab:dims}\addtocounter{equation}{1} \notag
\begin{gathered}
{\renewcommand{\arraystretch}{1.2}
\begin{tabular}{c|ccc}
& Total & New & Old \\
\hline
Modular forms & 408 & 96 & 312 \\
Cusp forms & 360 & 96 & 264 \\
Eisenstein series & 48 & 0 & 48
\end{tabular}
} \\[4pt]
\text{Table \ref{tab:dims}: Dimensions for subspaces of $M_3(560,[\chi])$}\\[4pt]
\end{gathered}
\end{equation}
\end{example}

One can also ask for the \defi{full trace form}
    \begin{equation} \label{eqn:fulltraceform}
    \sum_{n=1}^{\infty} \Tr(T_n\,|\,S_k(N,[\chi])) q^n \in S_k(N,[\chi];\Z) 
    \end{equation}
    on $S_k(N,[\chi])$ to some ($q$-adic) precision, with analogous definitions for the other subspaces considered above; see also \eqref{eqn:traceformjustone} below.

\subsection{Eisenstein series} \label{sub:eisenstein}

Beyond dimensions, we may next ask for further information about the decomposition of the space $M_k(N,\chi)$.  Of course the first step is the decomposition of the Eisenstein subspaces $E_k(N,\chi)$---for this purpose, explicit bases are given by Cohen--Str\"omberg \cite[Theorems 8.5.17, 8.5.22, and 8.5.23]{CohenStroemberg}.

\begin{rmk}
We do not currently display an Eisenstein basis in the LMFDB\@.
\end{rmk}

\subsection{Decomposition of newspaces into Hecke orbits} \label{sub:decomposition}

With the Eisenstein subspace described explicitly above, we now turn to the cuspidal subspace.  By the newform theory of Atkin--Lehner \cite{AtkinLehner} and Li \cite{Li}, the multiplicity of the space $\Sknew{k}(M,\chi_M)$ in $S_k(N,\chi)$, is equal to the number of divisors of $N/M$ (so depends only on the conductor and level).  While it suffices to study the new subspace, it may be computationally expensive to determine $\Sknew{k}(N,\chi)$ as a subspace of $S_k(N,\chi)$; one way to do this is via projection operators called \emph{degeneracy maps}, one for each prime divisor of~$N$.  

At this stage, for each \defi{newspace} $\Sknew{k}(N,[\chi])$ we may first ask for just the dimensions of its newform subspaces $V$ or \defi{Hecke orbits}---see \secref{sec:irred} below for a discussion of decomposition and irreducibility.  When $\chi$ is trivial, we may also ask for the decomposition of the space under the Atkin--Lehner involutions and the Fricke involution.  

\begin{example}
The space $\Sknew{2}(3111)$, with trivial character, has dimension $159$; it decomposes into newspaces of dimensions $1+2+3+3+7+13+14+14+21+24+28+29=159$, and we have the following decomposition into subspaces under Atkin--Lehner operators:

\begin{equation} \label{tab:ALdimex}\addtocounter{equation}{1} \notag
\begin{gathered}
{\renewcommand{\arraystretch}{1.2}
\begin{tabular}{ccc|c|cc}
$3$ & $17$ & $61$ & Fricke & dimension & decomposition \\
\hline
$+$ & $+$ & $+$ & $+$ & $13$ & $1+2+3+7$ \\
$+$ & $+$ & $-$ & $-$ & $29$ & $29$ \\
$+$ & $-$ & $+$ & $-$ & $24$ & $3+21$ \\
$+$ & $-$ & $-$ & $+$ & $14$ & $14$ \\
$-$ & $+$ & $+$ & $-$ & $24$ & $24$ \\
$-$ & $+$ & $-$ & $+$ & $14$ & $14$ \\
$-$ & $-$ & $+$ & $+$ & $13$ & $13$ \\
$-$ & $-$ & $-$ & $-$ & $28$ & $28$ 
\end{tabular}
} \\[4pt]
\text{Table \ref{tab:ALdimex}: Dimensions for subspaces of $\Sknew{2}(3111)$}\\[4pt]
\end{gathered}
\end{equation}
\end{example}

In practice, one computes this decomposition as follows.  We first compute a $\Q(\chi)$-basis for $\Sknew{k}(N,\chi)$ in some manner, and then we compute the matrix of $T_p$ on this basis for $p \nmid N$ for a few primes $p$ in such a way that a small (finite) $\Z$-linear combination $\sum_p c_p T_p$ has squarefree characteristic polynomial.  Therefore, the $\Q$-dimension decomposition is simply the degrees of the irreducible factors each multiplied by $[\Q(\chi):\Q]$.  There seems to be no problem in practice finding such a small linear combination, but the best thing that we can say rigorously involves the Sturm bound and appears to be far from optimal.  Already at this point engineering concerns enter: for example, the time to compute such a characteristic polynomial may be faster in certain implementations if done over $\Q$ instead.  

With this basic decomposition data in hand, we may continue.  For each newform orbit $[f] \leftrightarrow V$ (cf.\ \secref{sec:galoisdigress}) we wish to compute the following:
\begin{enumerate}
    \item The \defi{trace form} 
    \begin{equation} \label{eqn:traceformjustone}
    \Tr(f)(q) \colonequals \sum_{n=1}^{\infty} \Tr_{K|\Q}(a_n(f)) q^n \in S_k(N,[\chi];\Z) 
    \end{equation}
    (well-defined on the Galois orbit $[f]$), where $K$ is the coefficient field of $f$, to precision $n$ up to the \emph{Sturm bound} (see \secref{sec:sturm}).  Equivalently, writing $\Tr(f)(q) = \sum_n t_n q^n \in \Z[[q]]$, we have $t_n=\Tr(T_n\,|\,V)$ as the trace of the Hecke operator $T_n$ restricted to $V$---see \secref{sec:traceform} for further discussion.
    \item A minimal polynomial for the coefficient field $K$ of $[f]$.
    \item A finite set of generators for the \defi{Hecke kernel} for $V$, the ideal in the Hecke algebra on $\Sknew{k}(N,\chi)$ that vanishes on $V$; i.e., a finite set of polynomials in $T_n$ such that the ideal generated by these polynomials cuts out exactly $V$.  (We use the Hecke kernel when computing inner twists: see \secref{sec:twists}.)  
\end{enumerate}

Although it is possible to compute coefficients of the trace form $\Tr(f)$ by computing coefficients of $f$ and taking traces, this is more expensive than other techniques and is not computationally feasible in many cases where it is feasible to compute the trace form (e.g., using the trace formula: see section \secref{sec:algorithmTF}).  The trace form conveniently records interesting information about the newform orbit, e.g., the coefficient $t_1$ of the trace form is equal to the dimension of the newform subspace.  

\begin{example}
Consider the space $S_2(1166,[\chi])$ with label \href{http://www.lmfdb.org/ModularForm/GL2/Q/holomorphic/1166/2/c/}{\textsf{1166.2.c}}, the character having order $2$ and conductor $53 \mid 1166$.  The old subspace decomposes as
\[ S_2^{\textup{old}}(1166,[\chi]) \simeq \Sknew{2}(53,[\chi])^{\oplus 4} \oplus 
\Sknew{2}(106,[\chi])^{\oplus 2} \oplus \Sknew{2}(583,[\chi])^{\oplus 2}. \]
The decomposition of the new space $\Sknew{2}(1166,[\chi])$ into irreducibles by $\Q$-dimension is $46 = 2 + 22 + 22$, giving rise to three newform orbits \newformlink{1166.2.c.a}, \newformlink{1166.2.c.b}, \newformlink{1166.2.c.c} with respective trace forms
\begin{equation}
\begin{aligned}
    \Tr(f_{\textsf{a}})(q) &= 2q-2q^4 + 2q^6 - 8q^7 + 4q^9 + O(q^{10}) \\
    \Tr(f_{\textsf{b}})(q) &= 22q-22q^4-6q^6-24q^9 + O(q^{10}) \\
    \Tr(f_{\textsf{c}})(q) &= 22q-22q^4+4q^6+8q^7-34q^9 + O(q^{10}).
\end{aligned}
\end{equation}
We computed the last two trace forms \emph{without} computing  coefficients of a constituent newform (belonging to a number field of degree $22$, or even determining what this number field is), which would have been much more time consuming.  For the newform orbit $[f_{\textsf{a}}]$, we determined that its coefficient field is $\Q(\sqrt{-1})$, that it can be constructed as the kernel of the linear operator ${T_3^2+1}$ acting on $\Sknew{2}(1166,[\chi])$, and then computed the first 1000 coefficients $a_n$ of its $q$-expansion $\sum a_nq^n$ as elements of $\Q(\sqrt{-1})$. 
\end{example}

\subsection{Hecke eigenvalues}

Finally, for a newform $f$, we can ask for the coefficients of $f$ up to (at least) the Sturm bound.  These coefficients can be represented either exactly or as complex numbers (approximately, e.g.\ using interval arithmetic).  

\begin{itemize}
    \item For exact coefficients, there are issues in representing them compactly: see \secref{sec:LLLbasis} for our approaches. 
    \item For the numerical (complex) coefficients $a_n$, the most useful for computing $L$-functions (see the next section), we ask for these coefficients for each embedded form in the newspace.  These coefficients are of size $O(n^{(k-1)/2+\epsilon})$ for all $\epsilon>0$, so in large weight we prefer to compute the \defi{normalized coefficients} $a_n/n^{(k-1)/2}$, which by the Ramanujan--Petersson bounds have absolute value of size $O(n^{\epsilon})$.  
\end{itemize}
For large degree coefficient fields, it is often practical to compute numerical coefficients even when storing exact coefficients would be impractical.   

Finally, when the character is trivial, for the signs of the Atkin--Lehner involutions.

\begin{example}
Consider the newform orbit $\newformlink{5355.2.a.bf}$ of dimension $3$, with coefficient field $\Q(\nu)$ (LMFDB label \href{http://www.lmfdb.org/NumberField/3.3.169.1}{\textsf{3.3.169.1}}) where $\nu$ is a root of the polynomial $x^3-x^2-4x-1$.  The $q$-expansion of a newform $f$ in this orbit, with coefficients in $\Q(\nu)$, is
\[ f(q)=q+(1-\beta_1)q^2+(2-\beta_1+\beta_2)q^4+q^5+q^7+(2-\beta_1+2\beta_2)q^8+O(q^{10}) \]
where $\beta_1=\nu$ and $\beta_2=\nu^2-\nu-3$.

The $3$ embedded newforms are labeled $\textsf{5355.2.a.bf.1.m}$ for $m=1,2,3$ encoding the three embeddings $\iota_m \colon \Q(\nu) \hookrightarrow \C$; the embedded coefficients to $6$ decimal digits are as follows:

\begin{equation} \label{tab:embedcoeffs}\addtocounter{equation}{1} \notag
\begin{gathered}
{\renewcommand{\arraystretch}{1.2}
\begin{tabular}{cc|ccccccc}
Label & $\iota_m(\nu)$ & $a_2$ & $a_3$ & $a_4$ & $a_5$ & $a_6$ & $a_7$ & $a_8$\\
\hline
\href{http://www.lmfdb.org/ModularForm/GL2/Q/holomorphic/5355/2/a/bf/1/1/}{\textsf{1.1}} & $2.65109$ & $-1.65109$ & 0 & $0.726109$ & $1.00000$ & 0 & $1.00000$ & $2.10331$ \\
\href{http://www.lmfdb.org/ModularForm/GL2/Q/holomorphic/5355/2/a/bf/1/2/}{\textsf{1.2}} & $-0.273891$ & $1.27389$ & 0 & $-0.377203$ & $1.00000$ & 0 & $1.00000$ & $-3.02830$ \\ 
\href{http://www.lmfdb.org/ModularForm/GL2/Q/holomorphic/5355/2/a/bf/1/3/}{\textsf{1.3}} & $-1.37720$ & $2.37720$ & 0 & $3.65109$ & $1.00000$ & 0 & $1.00000$ & $3.92498$
\end{tabular}
} \\[4pt]
\text{Table \ref{tab:embedcoeffs}: Embedded newforms for \newformlink{5355.2.a.bf}. }\\[4pt]
\end{gathered}
\end{equation}
\end{example}

\subsection{\texorpdfstring{$L$}{L}-functions}

We can also ask for computations related to (invariants of) $L$-functions of modular forms, including the sign of the functional equation, the first few zeros, and special values to some precision: see \secref{sec:Lfunctions} for more detail.

\section{Algorithms} \label{sec:algs}

In this section, we give a brief overview of different algorithmic methods to compute modular forms and indicate where they are currently implemented.  In our computations for the LMFDB, we only used the first two (modular symbols and the trace formula), but here we also survey the others.  Our goal is to give a flavor of what each method entails, referring to the references provided for details.  Throughout, we keep notation from the previous section.  

\subsection{Modular symbols}

The most well-known method to compute modular forms is the method of \emph{modular symbols}, introduced by Birch \cite{Birch} and developed by Manin \cite{Manin}, Merel \cite{Merel}, Stein \cite{Stein}, and many others.  For an extensive history, see Stein \cite[8.10.2]{Stein}, and for a gentle overview see Stein \cite{Stein2}.  This method was implemented in \magma{} \cite{Magma} by William Stein, with contributions by Steve Donnelly and Mark Watkins, and in \sage{} \cite{Sage} by William Stein, with contributions by David Loeffler, Craig Citro, Peter Bruin, Fr\'ed\'eric Chapoton, Alex Ghitza, and {many} others.

We now briefly introduce modular symbols.  Assume $k \geq 2$.  Integration gives a perfect pairing
\begin{equation} \label{eqn:SkH1}
\begin{aligned}
S_k(\Gamma_1(N)) \times H_1(X_1(N),\R[x,y]_{k-2}) &\to \C \\
(f, \upsilon \otimes P) &\mapsto \int_\upsilon f(z)\,P(z,1)\,\mathrm{d}z
\end{aligned}
\end{equation}
where $\R[x,y]_{k-2}$ denotes the $\R$-vector space of homogeneous polynomials of degree $k-2$.  
In a slogan, \eqref{eqn:SkH1} indicates that \emph{the homology of a modular curve is dual to its cusp forms}, and this is formalized as follows.  Let $\Div (\PP^1(\Q))$ be the free abelian group on symbols $[\alpha]$ for $\alpha \in \PP^1(\Q)$, and let $\Div^0(\PP^1(\Q)) \leq \Div (\PP^1(\Q))$ be the subgroup of degree zero elements under the natural degree map.  Then $\Div^0(\PP^1(\Q))$ is generated by elements $\{\alpha,\beta\} \colonequals [\alpha]-[\beta]$ for $\alpha,\beta \in \PP^1(\Q)$, written this way to suggest a path from $\alpha$ to $\beta$ in $\C$.  We define the space of \defi{modular symbols} of weight $k$ and level $N$ (with $\Q$-coefficients) to be the quotient
\[
\ModSym_k(\Gamma_1(N);\Q) \colonequals \frac{\Q[x,y]_{k-2} \otimes \Div^0(\PP^1(\Q))}{\langle P \otimes \{\alpha,\beta\}-\gamma (P \otimes \{\alpha,\beta\}) \rangle_{\substack{\alpha,\beta \in \PP^1(\Q), \gamma \in \Gamma_1(N)}}}
\] 
under the natural action of $\Gamma_1(N) \leq \SL_2(\Q)$.  The space $\ModSym_k(\Gamma_1(N);\Q)$ of modular symbols has moreover a natural action of Hecke operators and Atkin-Lehner operators.  

\begin{thm} \label{thm:modsym}
There is a Hecke-equivariant isomorphism
\[ \ModSym_k(\Gamma_1(N);\Q) \xrightarrow{\sim} M_k(\Gamma_1(N);\Q) \oplus \overline{S}_k(\Gamma_1(N);\Q) \]
where $\overline{S}_k(\Gamma_1(N);\Q)$ denotes the space of anti-holomorphic cusp forms, the image of $S_k(\Gamma_1(N);\Q)$ under complex conjugation.
\end{thm}

\begin{proof}
See Manin \cite{Manin}, Merel \cite{Merel}, or Stein \cite[\S 8.5]{Stein}.
\end{proof}

Theorem \ref{thm:modsym} has many variants: one may restrict to $\Gamma_0(N)$, work with the (appropriately defined) space of \emph{cuspidal} modular symbols as the kernel of a certain boundary map, carve out just $M_k(\Gamma_1(N);\Q)$ as the $+$-space for a natural action of complex conjugation, and so on.  

\begin{example}
For $\Gamma_0(N)$, the space of modular symbols has a convenient description in terms of \emph{Manin symbols} as follows: $\ModSym_k(\Gamma_0(N);\Q)$ is the $\Q$-vector space generated by the set $\Delta$ of elements $\delta=(x^i y^{k-2-i}, (c:d))$ for $i=0,\dots,k-2$ and $(c:d) \in \PP^1(\Z/N\Z)$, modulo the subspace
\[ \langle \delta  + \delta S, \delta + \delta R + \delta R^2 \rangle_{\delta \in \Delta} \]
where $S=\begin{pmatrix} 0 & 1 \\ -1 & 0 \end{pmatrix}$ and $R=\begin{pmatrix} 0 & 1 \\ -1 & 1 \end{pmatrix}$.  The Hecke operators do not preserve Manin symbols, but there is an efficient procedure (arising from the Euclidean algorithm) for reducing an arbitrary element of $\ModSym_k(\Gamma_0(N);\Q)$ to a linear combination of Manin symbols.  
\end{example}

One feature of modular symbols is that they are especially well-suited for certain applications, including arithmetic invariants of elliptic curve quotients \cite{Cremona:database} (and more generally modular abelian varieties) as well as $L$-values of modular forms (see e.g.\ \secref{sec:analyticrank} below for an application).  Moreover, modular symbols can be employed for arbitrary congruence subgroups (see \cite{BanwaitCremona} for an example).  

In practice, it is quite efficient to compute the space of modular symbols with its Hecke action.  It is a matter of sparse linear algebra to compute a basis of modular symbols, a negligible contribution.  The number of field operations to compute the action of the Hecke operator $T_n$ on this basis is $\widetilde{O}(nd)$, where $d$ is the $\Q(\chi)$-dimension of the space under consideration: for each of the $d$ basis elements, we sum the action of $\sigma_1(n) \colonequals \sum_{d\mid n} d = \widetilde{O}(n)$ cosets and reduce to the basis in time polynomial in $\log n$ using continued fractions.  In this way, we may compute the $q$-expansions of a basis to precision $O(q^r)$ using $\softO(dr^2)$ field operations, and thereby also the trace form.

The most difficult engineering effort that goes into a working implementation of modular symbols is the careful handling of linear algebra aspects: we apply degeneracy operators to obtain precisely the subspace $\Sknew{k}(N,\chi)$, and once the matrices $[T_n]$ representing the Hecke operators are computed on this space, we compute its decomposition into newform subspaces, etc.  Indeed, in the preceding paragraphs, the actual \emph{time} complexity of this method may depend on the output desired and the meaning of ``arithmetic operation''. If we wish for exact results, which is the approach taken by \magma{} and \pari, then we need to do exact arithmetic with elements of cyclotomic fields, and the larger the order of the corresponding Dirichlet character, the more expensive the computation. Similarly, the coefficients of the newforms themselves may live in a large extension of the field of character values, and the larger this extension is, the harder the computation.  

\begin{remark} \label{rmk:coefffield}
As alternatives, we may do all of the computations described using floating point approximations to complex numbers, for example using \emph{complex ball arithmetic} to compute rigorous error bounds for all of the output.  In this case, the degree of the field of coefficients of the modular form is irrelevant, and the time complexity matches the estimates above; this is  particularly attractive if our application is to the computation of Dirichlet coefficients for input into $L$-function computations.  Similar comments apply by doing computations over a finite field, for example working with coefficients over a finite field with prime cardinality congruent to $1$ modulo the order of $\chi$---in this case, we can do all computations over $\F_p$.  In both cases, we must do some reconstruction to obtain exact results in characteristic zero.
\end{remark}

The above description requires weight $k \geq 2$.  For weight $1$, there are two approaches that reduce the problem to higher weight.  In the approach originated by Buhler \cite{BuhlerMono}, further developed by Buzzard \cite{BuzzardWtOne}, and carried out to scale by Buzzard--Lauder \cite{BuzzardLauder}, we choose nonzero $f \in M_k(\Gamma_1(N))$ and consider $S_1(\Gamma_1(N)) \subseteq f^{-1} M_{k+1}(\Gamma_1(N))$. Intersecting the spaces obtained for many choices of~$f$, we quickly obtain an upper bound for the space $S_1(\Gamma_1(N))$ that can then be matched with a lower bound.  Using Buzzard's code, this method was  implemented in \magma{} by Steve Donnelly.  (Currently, \magma{} can provide a basis for the cuspidal subspace, but it does not decompose the space into the old and new subspace and does not provide the action of the Hecke operators; this was implemented by Buzzard--Lauder, but has not yet been incorporated into \magma{}.)  A second related approach is to use the \emph{Hecke stability} method of Schaeffer \cite{Schaeffer}, instead computing the largest subspace of $f^{-1} M_{k+1}(\Gamma_1(N))$ that is stable under the Hecke operators; this has been implemented in \sage{} by Schaeffer and Loeffler, and in \pari{} by Belabas and Cohen \cite[\S 4]{BelabasCohen}.

\subsection{Trace formula} \label{sec:algorithmTF}

Perhaps the earliest method to compute modular forms used the \emph{trace formula}.  The trace formula is an explicit formula for the trace of a Hecke operator acting on a space of modular forms, and it was pioneered by Selberg \cite{Selberg} and later developed by Eichler \cite{Eichler}, Hijikata \cite{Hijikata}, and Cohen--Oesterl\'e \cite{CohenOesterle}.  A comprehensive treatment with references is the book of Knightly--Li \cite{KnightlyLi}, and a tidy presentation is given by Schoof--van der Vlugt \cite[Theorem 2.2]{SV91}. Proofs of the trace formula from different vantage points continue to be developed, see e.g.\ Popa \cite{Popa}.  This method has been implemented in \pari{} \cite{Pari} by Belabas--Cohen \cite{BelabasCohen} and in a standalone implementation by Bober, described in \secref{sec:bober}.  

We again assume $k \geq 2$.  An explicit version of the trace formula for $\Tr(T_n\,|\,S_k(\Gamma_0(N),\chi)) \in \Q(\chi)$ is too complicated to give here.  Aside from easily computed terms, it can be naively understood as a weighted sum of (Hurwitz) class numbers of imaginary quadratic fields: for a precise statement, see e.g.\ Belabas--Cohen \cite[Theorem 4]{BelabasCohen}.  We obtain  $\Tr(T_n\,|\,\Sknew{k}(N,\chi))$ from a nontrivial application of the M\"obius inversion formula, proven in Corollary \ref{cor:tnonnew} below.  

Let $d \colonequals \dim_{\Q(\chi)} \Sknew{k}(N, \chi) = O(kN)$.  The computation of $\trnew{n}$ requires computing class numbers of imaginary quadratic fields with absolute discriminant up to $O(n)$, and one can compute all of these at once in time complexity $\softO(n^{3/2})$.  For the purposes of a large-scale computation, these class numbers are cached and may be assumed to be precomputed (their cost amortized over many computations, thereby negligible).  Under this assumption, and given factorizations of $n$ and $N$, to compute $\trnew{n}$ we sum $O(\sqrt{n})$ terms giving a complexity of $O(\sqrt{n}N^\epsilon)$ field operations for any $\epsilon>0$; computing all traces up to $n>d$ then takes $\softO(n^{3/2})$ field operations.

In this manner, we compute the relative trace form on the new cuspidal subspace
    \begin{equation} 
    t(q) \colonequals \sum_{n=1}^{\infty} \trnew{n} q^n \in \Sknew{k}(N,\chi;\Z[\chi]),
    \end{equation}
    and from this we quickly compute the full trace form \eqref{eqn:fulltraceform} in $\Sknew{k}(\Gamma_1(N);\Z)$.   
    In particular, using the trace formula method we can compute either trace form to precision $O(q^r)$ using $\softO(r^{3/2}N^\epsilon)$ field operations, which for $r>d$ becomes $\softO(r^{3/2})$ as in the previous paragraph.
    
By multiplicity one theorems, and since the Hecke operators act semisimply on the new\-space, the images of $t$ under the Hecke operators span $\Sknew{k}(N,\chi)$.  Explicitly, applying $T_m$ to $t$, we obtain
\begin{equation}
   (T_m t)(q) = \sum_{n=1}^\infty \Tr(T_mT_n\,|\,\Sknew{k}(N,\chi)) q^n,
\end{equation}
and the forms $T_1 t, T_2 t, \dots$ span $\Sknew{k}(N, \chi)$. (We recall that $T_mT_n=T_{mn}$ when $\gcd(m,n)=1$, and more generally a recursion for the Hecke operators applies.  Therefore, these coefficients can again be expressed in terms of traces of Hecke operators.)  Once we have a spanning set, we can extract a basis and apply Hecke operators to that basis.

Typically (in practice) we need $O(d)$ forms to span and $O(d)$ coefficients of each form to get a full rank matrix.  Thus writing down a basis typically requires the first $O(d^2)$ values of $\trnew{n}$, which can be computed using $O(d^{3})$ field operations.  Finding this basis---and the $q$-expansion to precision $\widetilde{O}(d)$ for each form---is standard linear algebra, accomplished using $\widetilde{O}(d^3)$ field operations.  To compute the matrix of the Hecke operator $T_n$ on this basis requires traces up to $O(nd)$ and so $\softO(n^{3/2}d^{3/2})$ operations.  Finally and similarly, to compute a basis of $q$-expansions to precision $O(q^r)$ with $r > d$, we compute traces up to $O(rd)$ and apply a change of basis, for a total of $\softO(d^{3/2} r^{3/2})$ arithmetic operations.

We summarize the estimated complexity of these two approaches in Table \ref{tab:heuristicO}, where again $d$ is the $\Q(\chi)$-dimension of the space under consideration and we suppose precision $r>d$.

\begin{equation} \label{tab:heuristicO}\addtocounter{equation}{1} \notag
\begin{gathered}
{\renewcommand{\arraystretch}{1.2}
\begin{tabular}{l|cc}
    Task & Modular symbols & Trace formula \\ \hline
    Full trace form to precision $O(q^r)$, $d=O(r)$\phantom{\rule[0pt]{0pt}{15pt}} & $\softO(dr^2)$ & $\softO(r^{3/2})$ \\ 
    $[T_n]$ on a basis & $\softO(dn)$ & $\softO(d^{3/2}n^{3/2} + d^3)$ \\
    Characteristic polynomial of $T_n$ on a basis & $\softO(dn + d^3)$ & $\softO(d^{3/2}n^{3/2} + d^3)$ \\
    Basis of $q$-expansions to precision $O(q^r)$, $d=O(r)$ & $\softO(dr^2)$ & $\softO(d^{3/2}r^{3/2})$ \\
    Hecke decomposition & $\softO(d^3)$ & $\softO(d^3)$ \\
    Minimal polynomials for newspace coefficient fields & $\softO(d^3)$ & $\softO(d^3)$ \\
    \end{tabular}
} \\[6pt]
\text{Table \ref{tab:heuristicO}: Heuristic complexity of modular form computations}\\[4pt]
\end{gathered}
\end{equation}

So although linear algebra eventually dominates both approaches, neither modular symbols nor the trace formula seems to be a winner for all tasks: it seems to be much better to use modular symbols to get information about a small number of Hecke operators, while it is much better to use the trace formula to get a large number of coefficients of a basis of newforms.  This heuristic analysis matches our practical experience in the course of our computations.  

Similar comments with reference to weight $1$ forms apply as in the previous section.  The same is true for the issue of time complexity and the coefficient field (see e.g.\ Remark \ref{rmk:coefffield}), with the caveat that the matrices representing Hecke operators using modular symbols tend to be much sparser in comparison to those using the trace formula.  In particular, one expects that taking advantage of sparsity will allow a more efficient implementation of the linear algebra aspects for modular symbols.

\begin{remark}
In some circumstances, it can be more convenient to work with a basis that is in echelon form with respect to $q$-expansions (sometimes called a \emph{Victor Miller basis}) in the trace formula method.  With such a basis, going back and forth between an action on $q$-expansions and the matrix form for various linear operators one can see some gains in efficiency.
\end{remark}

\subsection{Definite methods}

In both of the previous algorithms, we work (either explicitly or implicitly) on the modular curve.  In this section, we indicate another class of algorithms that compute systems of Hecke eigenvalues using a different underlying object.  

Going back at least to Jacobi, surely the first modular forms studied were theta series.  Let $Q(x)=Q(x_1,\dots,x_d) \in \Z[x_1,\dots,x_d]$ be a positive definite integral quadratic form in $d=2k \in 2\Z_{\ge1}$ variables with discriminant $N$.  Let $P(x)$ be a (nonzero) spherical polynomial of degree $m \geq 0$ with respect to $Q$, for example $P(x)=1$.  We form the generating series for representations of $n \in \Z_{\geq 0}$ by $Q$ weighted by $P$, a \defi{theta series} of $Q$, by
\begin{equation} 
\theta_{Q,P}(q) \colonequals \sum_{x \in \Z^d} P(x) q^{Q(x)}.
\end{equation}
For example, if $P(x)=1$, then 
\begin{equation} 
\theta_{Q,1}(q) = \sum_{n=0}^{\infty} r_Q(n) q^n \in \Z[[q]] 
\end{equation}
where $r_Q(n)=\#\{x \in \Z^d : Q(x)=n\}$ counts the number of representations of $n$ by $Q$.  By letting $q=e^{2\pi iz}$ for $z \in \calH$ as usual, we obtain a holomorphic function $\theta_Q\colon \calH \to \C$.  Further, by an application of the Poisson summation formula (see e.g.\ Miyake \cite[Corollary 4.9.5]{Miyake}), we find that $\theta_{P,Q} \in M_{k+m}(\Gamma_0(2N),\chi_N)$ is a classical modular form of weight $k+m$, level $2N$, and character $\chi_N(a) \colonequals \displaystyle{\legen{N}{a}}$ of order at most 2.  

Turning this around, we can use theta series to compute spaces of classical modular forms.  Perhaps the most convenient source of such theta series is to work with quaternary ($d=4$) quadratic forms of square discriminant coming from quaternion algebras---this method goes by the name \emph{Brandt matrices} as it came about from early work of Brandt.  Building on work of Eichler \cite{Eichler:basis}, Hijikata--Pizer--Shemanske \cite{HPS} proved that linear combinations of such theta series span the space of cusp forms, up to twists.  (See also Martin \cite{Martin} for a more recent development.)  The coefficients of theta series can then be reformulated in terms of classes of right ideals of specified reduced norm in a quaternion order.  This method was first developed in an algorithmic context by Pizer \cite{Pizer}; it has been implemented in \magma{} by David Kohel and in \sage{} by Bober, Alia Hamieh, Victoria de Quehen, William Stein, and Gonzalo Tornar\'{i}a.  

In a little more detail, the method of Brandt matrices runs as follows.  Let $B$ be a definite quaternion algebra of discriminant $D \colonequals \disc B$, a squarefree product of the primes that ramify in $B$.  Let $\calO \subseteq B$ be an Eichler order of level $M$ with $\gcd(D,M)=1$, and let $N \colonequals DM$.  Let $\Cls \calO$ be the set of locally principal (equivalently, invertible) fractional right $\calO$-ideals up to isomorphism (given by left multiplication by an element of $B^\times$).  Then $\Cls \calO$ is a finite set, so let $\Cls \calO=\{[I_1],[I_2],\dots,[I_h]\}$ with $h \colonequals \#\Cls \calO$.  Let $\calO_{\textup{L}}(I_i)$ be the left order of $I_i$, and let $\Gamma_i \colonequals \calO_L(I_i)^\times$ be its unit group with $\#\Gamma_i < \infty$.  Let $q_i \colonequals \nrd(I_i)$.  For $n \in \Z_{\geq 1}$, define
\[ \Theta(n)_{i,j} \colonequals \Gamma_i \backslash \{ \alpha \in I_j I_i^{-1} : \nrd(\alpha)q_iq_j^{-1} = n\}. \]
We have $\alpha \in \Theta(n)_{i,j}$ if and only if $\alpha I_i \subseteq I_j$ with index $n^2$.  To connect this with the previous paragraph, we have
\begin{equation} 
\begin{aligned}
 Q_{ij}&\colon I_j I_i^{-1} \to \Z \\
 Q_{ij}&(\alpha) = \nrd(\alpha)\frac{q_i}{q_j} 
 \end{aligned}
 \end{equation}
 is a positive definite integral quaternary quadratic form of discriminant $N^2$ whose theta series descends to a modular form of level $N$---in the notation above, we have $r_{Q_{ij}}(n) = \#\Theta(n)_{i,j}$.  In this way, we can compute a matrix for the Hecke operator $[T_n]$ acting on $S_k(\Gamma_0(N),\chi)$ by quaternionic arithmetic: for weight $k=2$, the matrix $[T_n]$ is the adjacency matrix of the directed graph with vertex set $\Cls \calO$ and directed edges between $[I_i]$ and $[I_j]$ with multiplicity $\#\Theta(n)_{i,j}$.  
 
The method of Brandt matrices has several advantages.  First, the forms computed this way are necessarily new at all primes $p \mid D$, so linear algebra with degeneracy operators can be minimized.  Second, the matrices $[T_n]$ of Hecke operators are sparse: for example, in weight $k=2$ they have nonnegative integer coefficients whose columns sum to $\sigma(n)$.  Accordingly, linear algebra steps have an improved complexity both in theory and in practice.  Third, Brandt matrices also carry useful arithmetic information about the reduction of modular curves at primes of bad reduction.  Fourth, the set $\Theta(n)_{i,j}$ is independent of the weight $k$ and so may be reused.  Despite these advantages, the main limitation of Brandt matrices seems to be that it works most simply when there exists a prime $p$ that exactly divides the level $N$ (so that an Eichler order of reduced discriminant $N$ exists); otherwise, we must work with non-Eichler orders.  Hence current implementations focus on this case.  

The Brandt graph is an expander graph by the Ramanujan--Peterson bound, so with short vector computations one can compute a set of representatives for $\Cls \calO$ and a spanning set for $S_k(\Gamma_0(N),\chi)$ using $\softO(h^2)$ operations; computing a basis from this is a matter of sparse linear algebra and can be considered to be negligible.  To compute a single matrix $[T_n]$, in principle we could use Minkowski reduction (together with some awkward corner cases) on $h\sigma(n)$ right ideals using $\softO(hn) = \softO(dn)$ operations.  To compute a basis of $q$-expansions to precision $O(q^r)$ with $d=O(r)$, for each of the $h$ classes we can enumerate elements of small reduced norm using the Fincke--Pohst algorithm in time proportional to the volume so $\softO(dr^2)$, performing reduction with the same complexity.  These heuristics match the running time of modular symbols with linear algebra again eventually dominating--- however, it is here where sparse linear algebra may ultimately in practice give the Brandt matrix method an edge.  

A method that shares much in common with Brandt matrices is the \emph{method of graphs} due to Mestre \cite{Mestre} and Oesterl\'e.  We suppose that $p \parallel N$ and work with the quaternion algebra $B$ of discriminant $D=p$.  We recall that there is an equivalence of categories between supersingular elliptic curves over an algebraic closure of $\F_p$ under isogenies and invertible right (or left) $\calO$-modules under homomorphisms.  So to compute a matrix for the Hecke operator, in place of $\Cls \calO$ we can compute the set of isomorphism classes of pairs $(E,C)$ where $E$ is a supersingular elliptic curve in characteristic $p$ and $C$ is a cyclic subgroup of order $M = N/p$, and in place of the sets $\Theta(n)_{ij}$ we can enumerate cyclic isogenies between these points up to a natural equivalence.  

Finally, a related method of Birch \cite{Birch:ternary} (who sought to generalize the method of graphs beyond discriminant $D=p$) uses \emph{ternary} quadratic forms instead.  This method captures all of the advantages above, with an additional feature: work in progress by Hein--Tornar\'ia--Voight shows that one can carve out not just a new subspace but moreover one can specify the Atkin--Lehner eigenvalue, reducing the total dimension and thereby the complexity of linear algebra operations.  

\subsection{Other methods}

We conclude by briefly indicating two other methods in addition to the above.  
\begin{itemize}
\item \emph{Multiplying forms of lower weight}.  We compute a presentation for the graded ring of modular forms of level $N$
\[ M(\Gamma_1(N)) \colonequals \bigoplus_{k=0}^{\infty} M_k(\Gamma_1(N)) \]
(or the same for $\Gamma_0(N)$) in terms of a finite set of generators and a Gr\"obner basis for the ideal of relations among them; see work of Voight--Zureick-Brown \cite{VZB} for an explicit description of this graded ring in terms of the genus and number of cusps for $\Gamma_1(N)$ (and more generally in terms of the signature of the uniformizing Fuchsian group) as well as further references and discussion.  From this, one can compute for each weight $k$ a set of (leading) monomials in the generators that are a $\Q$-basis for $M_k(\Gamma_1(N))$.  Using fast Fourier techniques, the multiplication of these $q$-expansions allows the computation of a basis for large weights $k$ (and fixed level $N$) quite efficiently in comparison to any of the approaches above.  
\item \emph{Polynomial-time algorithms}.  By work of Edixhoven--Couveignes \cite{EdixhovenCouveignes}, Bruin \cite{Bruin}, and Mascot~\cite{Mascot}, one can compute coefficients of modular forms of level $1$ in polynomial time: for example, for the modular discriminant $\Delta(q) = \sum_n \tau(n)q^n \in S_{12}(1)$, the value $\tau(p)$ for a prime $p$ can be computed in time bounded by a fixed power of $\log p$.  
\end{itemize}

\section{Two technical ingredients} \label{sec:technical}

In this section, we consider two technical results that are needed in the above algorithmic description.

\subsection{Eichler--Selberg trace formula for newforms} \label{sec:eichsel}

We first prove a technical result that is used by Belabas--Cohen \cite{BelabasCohen} in the computation of modular forms in \pari{} \cite{Pari}, as explained above: we describe the trace of Hecke operators on the new subspace in terms of the trace on the total space.  

Let $\chi$ be a primitive character of conductor $Q\mid N$ and $k$ a positive
integer satisfying $\chi(-1)=(-1)^k$; we take these to be fixed and
suppress their dependence from the notation.

For any positive integer $n$, the $n$th Hecke
operator $T_n\colon \Sk(N,\chi)\to\Sk(N,\chi)$ may
be defined by
\[(T_nf)(z)
= \frac1n\sum_{\substack{ad=n\\\gcd(a,N)=1}}\chi(a)a^k\sum_{b\bmod{d}} 
f\!\left(\frac{az+b}{d}\right).\]
Then 
\[(T_nf)(z)=\sum_{m=1}^\infty \left(\sum_{\substack{d\mid (m, n) \\ (d, N)=1}}
\chi(d) d^{k-1} a_{\frac{mn}{d^2}}\right)e(mz),\]
where $f(z)=\sum_{m=1}^\infty a_me(mz)$. 
This operator stabilizes the subspace $\Sknew{k}(N,\chi)$.

Let $\{f_{N, j}\}_{j=1}^{s_N}$ be a basis of normalized newforms
for $\Sknew{k}(N,\chi)$ and write
\[f_{N,j}(z)=\sum_{m=1}^\infty a_{N,j}(m)e(mz).\]
We assume that each $f_{N, j}$ is an eigenfunction of $T_n$ of eigenvalue $a_{N, j}(n)$ and define 
\[g_n = \sum_{j=1}^{s_N} a_{N, j}(n) f_{N, j}
= \sum_{m=1}^\infty e(mz)\Tr \big(T_n T_m|_{\Sknew{k}(N, \chi)}\big).\]

We parameterize the basis of $\Sk(N, \chi)$: 
for $M_1,M_2\in\Z_{\ge1}$ with $Q\mid M_1$ and $M_1M_2\mid N$, let
\[f_{M_1,j}^{M_2}(z) := f_{M_1,j}(M_2z).\]
Then
\[\bigl\{f_{M_1,j}^{M_2}\;:\; M_1,M_2\in\Z_{\ge1},\,Q\mid M_1,\,M_1M_2\mid N\bigr\}\]
is a basis for $\Sk(N,\chi)$.
Let us extend the definition of $a_{N,j}(n)$ to $\Q_{>0}$ by
writing $a_{N,j}(x)=0$ if $x\notin\Z_{\ge1}$.

If $\gcd(n,N)=1$, then
\begin{align*}
	T_nf_{M_1,j}^{M_2}
	&=
	\sum_{m=1}^\infty\sum_{\substack{d\mid(m,n) \\ (d,N)=1}}
	\chi(d)d^{k-1}a_{M_1,j}\!\left(\frac{mn}{d^2M_2}\right)e(mz)
	\\ &=
	\sum_{m=1}^\infty
	a_{M_1,j}\!\left(\frac{m}{M_2}\right)
	a_{M_1,j}(n)e(mz)
	=
	a_{M_1,j}(n)f_{M_1,j}^{M_2},
\end{align*}
so each $f_{M_1,j}^{M_2}$ is an eigenfunction of $T_n$.
To compute the action of $T_n$ when $\gcd(n,N)>1$, we need the following
theorem.  

\begin{thm}
Let $p \mid N$ be prime, let $\alpha\in \Z_{\geq 0}$, and let $r \colonequals \ord_pM_2$.  Let $\chi_{M_1}$ be the character modulo $M_1$ induced from $\chi$.  
Then
\begin{equation}\label{e:Tpalpha_fM1M2}
T_{p^{\alpha}} f_{M_1, j}^{M_2} 
= \begin{cases}
f_{M_1, j}^{M_2p^{-\alpha}},
& \text{ if } \alpha -r \leq 0; \\
a_{M_1, j}(p^{\alpha-r}) f_{M_1, j}^{M_2p^{-r}}
- \chi_{M_1} (p) p^{k-1} a_{M_1, j}(p^{\alpha-r-1})
f_{M_1, j}^{M_2p^{-r+1}},
& \text{ if } \alpha -r > 0.
\end{cases}
\end{equation}
\end{thm}

\begin{proof}
By the definition of the Hecke operator, we get
\[T_{p^{\alpha}} f_{M_1, j}^{M_2} 
= \sum_{m=1}^\infty 
\sum_{\substack{d\mid (m, p^{\alpha}), \\ (d, N)=1}}
\chi(d) d^{k-1} a_{M_1, j}\left(\frac{m p^{\alpha}}{M_2 d^2}\right) e(mz)
= \sum_{m=1}^\infty 
a_{M_1, j}\left(\frac{m p^{\alpha}}{M_2}\right) e(mz).\]
If $\alpha-r\leq 0$, then 
\[T_{p^{\alpha}} f_{M_1, j}^{M_2} 
= \sum_{m=1}^\infty 
a_{M_1, j}\left(\frac{m}{M_2p^{-r} \cdot p^{-\alpha+r}}\right) e(mz)
= f_{M_1, j}^{M_2p^{-\alpha}}. \]

Assume that $\alpha-r > 0$. 
Since $f_{M_1, j}$ is a normalized newform for $\Gamma_0(M_1)$, we get
\begin{multline}\label{eqn:p_Hecke}
	a_{M_1, j}\left(\frac{m}{M_2p^{-r}}\right)
	\cdot 
	a_{M_1, j}\left(p^{\alpha-r}\right)
	\\ = \begin{cases}
	a_{M_1, j}\left(\frac{mp^{\alpha}}{M_2}\right), & \text{ if } p\mid M_1, \\
	\sum_{e=0}^{\min\{ \ord_p(m), \alpha-r\}} 
	(\chi(p) p^{(k-1)})^e a_{M_1, j} \left(\frac{mp^{\alpha-2e-r}}{M_2p^{-r}}\right), & \text{ if } p\nmid M_1. 
	\end{cases}
\end{multline}
Then, if $p\mid M_1$, we have
$$
	T_{p^{\alpha}} f_{M_1, j}^{M_2} 
	=
	a_{M_1, j}(p)^{\alpha-r} \cdot f_{M_1, j}^{M_2p^{-r}}. 
$$

We now assume that $\alpha-r>0$ and $p\nmid M_1$. 
If $\alpha-r = 1$, we have
$$
	a_{M_1, j}\left(\frac{mp^\alpha}{M_2}\right)
	=
	a_{M_1, j}\left(\frac{m}{M_2p^{-r}}\right)
	\cdot 
	a_{M_1, j}\left(p^{\alpha-r}\right)
	-
	\delta_{\ord_p(m)\geq 1}
	\chi(p) p^{(k-1)} a_{M_1, j} \left(\frac{mp^{\alpha-2}}{M_2}\right).	
$$
By taking the summation over $m\in \Z_{\ge1}$, we get:
\begin{align*}
	T_{p^{\alpha}} f_{M_1, j}^{M_2} 
	&=
	a_{M_1, j}(p^{\alpha-r}) f_{M_1, j}^{M_2p^{-r}}
	-
	\chi(p)p^{k-1}
	\sum_{m=1}^\infty 
	a_{M_1, j}\left(\frac{mp^{\alpha-1}}{M_2}\right) e(mpz)
	\\
	&=
	a_{M_1, j}(p^{\alpha-r}) f_{M_1, j}^{M_2p^{-r}}
	-
	\chi(p)p^{k-1}
	f_{M_1, j}^{M_2p^{-r+1}}
	.
\end{align*}
Note that when $r=0$ we have $M_1M_2p\mid N$, since $p\mid N$. 

If $\alpha-r-2\geq 0$, by changing $\alpha$ to $\alpha-2$, we get
\[\chi(p) p^{k-1}
a_{M_1, j}\left(\frac{m}{M_2p^{-r}}\right)
\cdot a_{M_1, j}\left(p^{\alpha-2-r}\right)
= \sum_{e=1}^{\min\{ \ord_p(m), \alpha-2-r\}+1} 
(\chi(p)p^{(k-1)})^e a_{M_1, j} \left(\frac{mp^{\alpha-2e-r}}{M_2p^{-r}}\right).\]
By subtracting from \eqref{eqn:p_Hecke}, we get
\begin{multline*}
	\left\{a_{M_1, j}(p^{\alpha-r})
	- \chi(p)p^{k-1} a_{M_1, j}(p^{\alpha-r-2})\right\}
	a_{M_1, j}\left(\frac{m}{M_2p^{-r}}\right)
	=
	a_{M_1, j}\left(\frac{mp^{\alpha}}{M_2}\right)
	\\+
	\begin{cases}
	- \left(\chi(p)p^{k-1}\right)^{\ord_p(m)+1}
	a_{M_1, j}\left(\frac{mp^{\alpha-2(\ord_p(m)+1)-r}}{M_2p^{-r}}\right), 
	& \text{ if } 0\leq \ord_p(m) \leq \alpha-2-r, \\
	\left(\chi(p)p^{k-1}\right)^{\alpha-r}
	a_{M_1, j}\left(\frac{mp^{-(\alpha-r)}}{M_2p^{-r}}\right), 
	& \text{ if } \ord_p(m) \geq \alpha-r, \\
	0, & \text{ otherwise.} 
	\end{cases}
\end{multline*}

After taking the summation over $m\in \Z_{\ge1}$ on both sides, we get
\begin{align*}
	\sum_{m=1}^\infty a_{M_1, j}\left(\frac{mp^{\alpha}}{M_2}\right) e(mz)
	&=
	T_\alpha f_{M_1, j}^{M_2}(z)
	\\
	&=
	\left\{
	a_{M_1, j}(p^{\alpha-r}) - \chi(p)p^{k-1} a_{M_1, j}(p^{\alpha-2-r})\right\}
	f_{M_1, j}^{M_2p^{-r}}
	\\
	&\qquad+
	\sum_{\ell=0}^{\alpha-2-r} 
	(\chi(p)p^{k-1})^{\ell+1}
	\sum_{\substack{m=1, \\ p\nmid m}}^\infty 
	a_{M_1, j}\left(\frac{mp^{\alpha-2-r-\ell}}{M_2p^{-r}}\right) e(mp^\ell z)
	\\
	&\qquad -
	(\chi(p)p^{k-1})^{\alpha-r} 
	\sum_{m=1}^\infty a_{M_1,j}\left(\frac{m}{M_2p^{-r}}\right) e(mp^{\alpha-r} z)
	.
\end{align*}

For the last piece, we have 
\[\sum_{m=1}^\infty a_{M_1,j}\left(\frac{m}{M_2p^{-r}}\right) e(mp^{\alpha-r} z)
= f_{M_1, j}^{M_2p^{\alpha-2r}}(z). \]

Now consider 
\[\sum_{\substack{m=1, \\ p\nmid m}}^\infty 
a_{M_1, j}\left(\frac{m}{M_2p^{-r}}\right)e(mz)
= f_{M_1, j}^{M_2p^{-r}}(z)
- \sum_{m=1}^\infty a_{M_1, j}\left(\frac{mp}{M_2p^{-r}}\right) e(mpz).\]
Since
\[a_{M_1, j}(p) \cdot f_{M_1, j}^{M_2p^{-r}}(z)
= \sum_{m=1}^\infty a_{M_1, j}\left(\frac{mp}{M_2p^{-r}}\right) e(mz)
+\chi(p)p^{k-1} f_{M_1, j}^{M_2p^{-r+1}}(z), \]
we get
\[\sum_{\substack{m=1, \\ p\nmid m}}^\infty 
a_{M_1, j}\left(\frac{m}{M_2p^{-r}}\right) e(mz)
= f_{M_1, j}^{M_2p^{-r}}(z)
- a_{M_1, j}(p) \cdot f_{M_1, j}^{M_2p^{-r+1}}(z)
+\chi(p)p^{k-1} f_{M_1, j}^{M_2p^{-r+2}}(z).\]
For each $0\leq \ell \leq \alpha-2-r$, we get
\begin{multline*}
	\sum_{\substack{m=1, \\ p\nmid m}}^\infty 
	a_{M_1, j}\left(\frac{mp^{\alpha-2-r-\ell}}{M_2p^{-r}}\right) e(mp^\ell z)
	\\
	=
	a_{M_1, j}\left(p^{\alpha-2-r-\ell}\right)
	\cdot 
	\left\{
	f_{M_1, j}^{M_2p^{-r+\ell}}(z)
	- 
	a_{M_1, j}(p) \cdot f_{M_1, j}^{M_2p^{-r+1+\ell}}(z)
	+
	\chi(p)p^{k-1} f_{M_1, j}^{M_2p^{-r+2+\ell}}(z)
	\right\}.
\end{multline*}
So we finally get
\begin{align*}
	T_\alpha f_{M_1, j}^{M_2}
	&=
	\left\{
	a_{M_1, j}(p^{\alpha-r}) - \chi(p)p^{k-1} a_{M_1, j}(p^{\alpha-2-r})\right\}
	f_{M_1, j}^{M_2p^{-r}}
	\\
	&\qquad +
	\sum_{\ell=0}^{\alpha-2-r} 
	(\chi(p)p^{k-1})^{\ell+1}
	a_{M_1, j}\left(p^{\alpha-2-r-\ell}\right) \\
	&\qquad\qquad\qquad\qquad \cdot 
	\left\{
	f_{M_1, j}^{M_2p^{-r+\ell}} 
	- 
	a_{M_1, j}(p) \cdot f_{M_1, j}^{M_2p^{-r+1+\ell}}
	+
	\chi(p)p^{k-1} f_{M_1, j}^{M_2p^{-r+2+\ell}}
	\right\}
	\\
	&\qquad -
	(\chi(p)p^{k-1})^{\alpha-r} 
	f_{M_1, j}^{M_2p^{\alpha-2r}}
	.
\end{align*}

For $s\in \Z_{\geq 0}$, we have
\[a_{M_1, j}(p^s)\cdot a_{M_1, j}(p)
= a_{M_1, j}(p^{s+1})
+ \chi(p)p^{k-1} a_{M_1, j}(p^{s-1}), \]
so we get
\begin{align*}
	&\sum_{\ell=0}^{\alpha-2-r} 
	(\chi(p)p^{k-1})^{\ell+1}
	a_{M_1, j}\left(p^{\alpha-2-r-\ell}\right)
	a_{M_1, j}(p) \cdot f_{M_1, j}^{M_2p^{-r+1+\ell}}
	\\
	&\qquad =
	\sum_{\ell=0}^{\alpha-2-r} 
	(\chi(p)p^{k-1})^{\ell+1}
	a_{M_1, j}(p^{\alpha-1-r-\ell})
	\cdot f_{M_1, j}^{M_2p^{-r+1+\ell}} \\
	&\qquad\qquad
	+
	\sum_{\ell=0}^{\alpha-2-r} 
	(\chi(p)p^{k-1})^{\ell+2}
	a_{M_1, j}\left(p^{\alpha-3-r-\ell}\right)
	\cdot f_{M_1, j}^{M_2p^{-r+1+\ell}}
	\\
	&\qquad =
	\sum_{\ell=1}^{\alpha-1-r} 
	(\chi(p)p^{k-1})^{\ell}
	a_{M_1, j}(p^{\alpha-r-\ell})
	\cdot f_{M_1, j}^{M_2p^{-r+\ell}} \\
	&\qquad\qquad +
	\sum_{\ell=1}^{\alpha-1-r} 
	(\chi(p)p^{k-1})^{\ell+2}
	a_{M_1, j}\left(p^{\alpha-2-r-\ell}\right)
	\cdot f_{M_1, j}^{M_2p^{-r+\ell}}
	.
\end{align*}
Then we have
\begin{align*}
	T_\alpha f_{M_1, j}^{M_2}&=
	\left\{
	a_{M_1, j}(p^{\alpha-r}) - \chi(p)p^{k-1} a_{M_1, j}(p^{\alpha-2-r})\right\}
	f_{M_1, j}^{M_2p^{-r}}
	\\
	&\qquad+
	\sum_{\ell=0}^{\alpha-2-r} 
	(\chi(p)p^{k-1})^{\ell+1}
	a_{M_1, j}\left(p^{\alpha-2-r-\ell}\right)
	f_{M_1, j}^{M_2p^{-r+\ell}} \\
	&\qquad+
	\sum_{\ell=2}^{\alpha-r} 
	(\chi(p)p^{k-1})^{\ell}
	a_{M_1, j}\left(p^{\alpha-r-\ell}\right)
	f_{M_1, j}^{M_2p^{-r+\ell}}
	\\
	&\qquad-
	\sum_{\ell=1}^{\alpha-1-r} 
	(\chi(p)p^{k-1})^{\ell}
	a_{M_1, j}(p^{\alpha-r-\ell})
	\cdot f_{M_1, j}^{M_2p^{-r+\ell}} \\
	&\qquad-
	\sum_{\ell=1}^{\alpha-1-r} 
	(\chi(p)p^{k-1})^{\ell+2}
	a_{M_1, j}\left(p^{\alpha-2-r-\ell}\right)
	\cdot f_{M_1, j}^{M_2p^{-r+\ell}}
	\\
	&\qquad-
	(\chi(p)p^{k-1})^{\alpha-r} 
	f_{M_1, j}^{M_2p^{\alpha-2r}}
	\\
	&=
	a_{M_1, j}(p^{\alpha-r}) 
	f_{M_1, j}^{M_2p^{-r}}
	-
	(\chi(p)p^{k-1})
	a_{M_1, j}(p^{\alpha-r-1})
	f_{M_1, j}^{M_2p^{-r+1}}
	.
\end{align*}
Combining, we obtain \eqref{e:Tpalpha_fM1M2}.
\end{proof}

For $n,N \in \Z_{>0}$, we write $\gcd(n,N^\infty)$ for the largest positive integer $d$ such that $d \mid n$ and $d \mid N^k$ for some $k \in \Z_{\geq 1}$, i.e.,
\begin{equation}
\gcd(n, N^\infty)=\prod_{p \mid \gcd(n, N)} p^{\ord_p(n)}.
\end{equation}
The following corollary is then immediate.

\begin{cor}	\label{cor:tnonnew}
With notation as above, we have
$$
	\Tr(T_n\,|\, \Sk (N, \chi))
	=
	\sum_{\substack{ M\in \Z_{\ge1} \\ M\mid N \\ \cond(\chi) \mid M}}
	d\left(\frac{N/M}{\gcd(N/M, n^\infty)}\right)
	\sum_{\substack{b^2\mid\gcd(n,N^\infty)\\\gcd(b,M)=1}}
	\mu(b)
	\chi(b) b^{k-1}
	\Tr(T_{\frac{n}{b^2}}\,|\,\Sknew{k}(M, \chi))
	.
$$
\end{cor}

\subsection{Certifying generalized eigenvalues} \label{sec:certifyingeigenvalues}

Second, we show how to certify generalized eigenvalues.  Consider the generalized eigensystem
\begin{equation}
Ax=\lambda{}Bx,
\end{equation}
where $A$ and $B$ are real symmetric $n\times n$ matrices, with $B$ positive
definite. Choosing $R$ such that $B=R^TR$ and making the change of
variables $x=R^{-1}y$, this becomes
\begin{equation}
A'y=\lambda{y},
\end{equation}
where $A'=(R^{-1})^TAR^{-1}$.
Note that $A'$ is again symmetric, so there is an orthonormal basis
$\{y_1,\ldots,y_n\}$ with $A'y_j=\lambda_jy_j$. We set $x_j=R^{-1}y_j$,
so that the $x_j$ are orthonormal with respect to the inner product
defined by $B$.

Suppose that we have found approximate eigenvalues $\tl_j$ and
eigenvectors $\tx_j$, i.e.\ so that $e_j=(A-\tl_jB)\tx_j$
is small. Let
$$
\tx_j=\sum_{k=1}^n c_{jk}x_k
$$
be the expansion of $\tx_j$ in terms of the eigenbasis.
For any $\varepsilon>0$, define
\begin{equation}
V_{j,\varepsilon}=\Span\{x_k:|\lambda_k-\tl_j|<\varepsilon\},
\end{equation}
and let
\begin{equation}
v_{j,\varepsilon}=\sum_{\substack{k \\ |\lambda_k-\tl_j|<\varepsilon}}c_{jk}x_k
\end{equation}
be the orthogonal projection (with respect to the inner product defined
by $B$) of $\tx_j$ onto $V_{j,\varepsilon}$.  Then we have
\begin{equation}
\begin{aligned}
v_{j,\varepsilon}^TBv_{j,\varepsilon}
&=\tx_j^TB\tx_j-\sum_{\{k:|\lambda_k-\tl_j|\ge\varepsilon\}}c_{jk}^2
\ge
\tx_j^TB\tx_j-\varepsilon^{-2}\sum_{k=1}^n c_{jk}^2(\lambda_k-\tl_j)^2\\
&=\tx_j^TB\tx_j
-\varepsilon^{-2}[(B^{-1}A-\tl_j)\tx_j]^TB[(B^{-1}A-\tl_j)\tx_j]\\
&=\tx_j^TB\tx_j
-\varepsilon^{-2}e_j^TB^{-1}e_j
\ge \tx_j^TB\tx_j-\varepsilon^{-2}b^{-1}|e_j|^2,
\end{aligned}
\end{equation}
where $b>0$ is the smallest eigenvalue of $B$.
Note that this is positive if
$$
\varepsilon>\varepsilon_j:=\frac{|e_j|}{\sqrt{b\tx_j^TB\tx_j}},
$$
and thus $V_{j,\varepsilon}$ is non-zero. Hence, there is an eigenvalue
$\lambda_k$ in the interval $I_j=[\tl_j-\varepsilon_j,\tl_j+\varepsilon_j]$.

Suppose that we are in the favorable situation that the $I_j$ are
pairwise disjoint. Then our system has distinct eigenvalues, and we may
assume without loss of generality that $\lambda_j\in I_j$.  Next, let
$\delta_j=\min\bigl\{|\lambda-\tl_j|:\lambda\in\bigcup_{k\ne j}I_k\bigr\}$,
so that $(\tl_j-\delta_j,\tl_j+\delta_j)$ contains $\lambda_j$ and no
other eigenvalues. Put $\Delta_j=\tx_j-v_{j,\delta_j}$, where
$V_{j,\varepsilon}$ and $v_{j,\varepsilon}$ are as above, so that
$\tx_j-\Delta_j$ is an eigenvector with eigenvalue $\lambda_j$.
To use this in practice, we bound the coordinates of $\Delta_j$ from above
and add them as small error intervals onto the coordinates of $\tx_j$.
(The resulting vector must then be renormalized in interval arithmetic,
according to whatever convention we use, e.g.\ first Fourier coefficient
$1$.) To that end, we have
\begin{equation}
\begin{aligned}
|R\Delta_j|^2=\Delta_j^TB\Delta_j&=\sum_{\{k:|\lambda_k-\tl_j|\ge\delta_j\}}c_{jk}^2
\le\delta_j^{-2}\sum_{k=1}^n c_{jk}^2(\lambda_k-\tl_j)^2\\
&=\delta_j^{-2}[(B^{-1}A-\tl_j)\tx_j]^TB[(B^{-1}A-\tl_j)\tx_j]
=\delta_j^{-2}e_j^TB^{-1}e_j,
\end{aligned}
\end{equation}
and thus
\begin{equation}
|\Delta_j|\le\delta_j^{-1}\sqrt{b^{-1}e_j^TB^{-1}e_j}
\le\frac{|e_j|}{b\delta_j}.
\end{equation}

Finally, to estimate $b$ we first compute a double-precision approximation
$\tP$ to the orthogonal matrix which diagonalizes $B$. We then compute
in interval arithmetic the matrices
$$
S=(s_{jk})=\tP^TB\tP\quad\mbox{and}\quad
T=(t_{jk})=\tP^T\tP.
$$
By Sylvester's law of inertia, we have $b>\lambda$ for any $\lambda$
such that $S-\lambda{T}$ is positive definite. In turn, by the
Gershgorin circle theorem, this holds if the diagonal entries $s_{jj}$
and $t_{jj}$ are strictly positive and
\begin{equation}
\lambda > \lambda^* \colonequals
\min_j\frac{2s_{jj}-\sum_k|s_{jk}|}{\sum_k|t_{jk}|}.
\end{equation}
Hence $b\ge\lambda^*$.

\section{A sample of the implementations} \label{sec:implementations}

\subsection{Comparison of methods}

In the course of our computations we made extensive use of the modular forms functionality included in both \pari{} \cite{Pari} and \magma{} \cite{Magma}.  In this section we compare the performance of the two implementations on a small but representative subset of the modular forms we computed: all newforms of weight $k$ and level $N$ with $Nk^2 \le 1000$ and $k>1$.  We exclude the case $k=1$ from this comparison because it is not fully supported in \magma{} and the algorithms used to compute \weightone{} forms are substantially different.  For modular forms of weight $k>1$ the \magma{} implementation is based on the modular symbols approach, while the \pari{} implementation uses the explicit trace formula.

For each level $N$ in our chosen range we fix a representative Dirichlet character $\chi$ for each Galois orbit $[\chi]$ of modulus $N$, and for each newspace $\Sknew{k}(N,\chi)$ with $k>1$ and $Nk^2\le 2000$ we carried out the following computations in both \pari{} and \magma{}:

\begin{enumerate}
\item Determine the dimensions of the irreducible subspaces of $\Sknew{k}(N,[\chi])$ (the newform orbits).
\item For each newform orbit $[f]$, compute the first 1000 integer coefficients $t_n$ of the trace form $\Tr(f)=\sum_{n\ge 1} t_nq^n$.
\item For each newform orbit $[f]$ of (absolute) dimension $d\le 20$, compute a (reasonably nice) defining polynomial for its coefficient field $K$ and the first 1000 algebraic integer coefficients $a_n(f) \in K$ for a constituent newform $f$.
\item For each newform orbit $[f]$ of dimension $d\le 20$, compute an LLL-optimized basis for its coefficient ring and express the first 1000 coefficients $a_n(f)$ in this basis.
\end{enumerate}

\begin{equation} \label{tab:timings1}\addtocounter{equation}{1} \notag
\begin{gathered}
{\renewcommand{\arraystretch}{1.2}
    \begin{tabular}{l|rrr|rr|rr}
    & & & & \multicolumn{2}{c}{split time (s)} & \multicolumn{2}{c}{total time (s)}\\
    $Nk^2$ & num $S$ & num $f$ & $\sum \dim(S)$ & \magma{} & \pari{}  & \magma{} & \pari{} \\\hline
    $[1,200]$ &\numprint{183} & \numprint{214} & \numprint{897} & \numprint{0.4} & \numprint{1.1} & \numprint{73.8} & \numprint{18.2}\\
    $[201,400]$ &\numprint{453} & \numprint{709} & \numprint{7560} & \numprint{3.5} & \numprint{17.2} & \numprint{302.4} & \numprint{116.6}\\
    $[401,600]$ &\numprint{574} & \numprint{1050} & \numprint{21452} & \numprint{22.2} & \numprint{50.2} & \numprint{643.4} & \numprint{220.1}\\
    $[601,800]$ &\numprint{677} & \numprint{1326} & \numprint{43515} & \numprint{132.1} & \numprint{70.8} & \numprint{2444.8} & \numprint{300.6}\\
    $[801,1000]$ &\numprint{764} & \numprint{1542} & \numprint{71358} & \numprint{751.3} & \numprint{322.3} & \numprint{9216.4} & \numprint{728.2}\\
    $[1001,1200]$ &\numprint{879} & \numprint{1805} & \numprint{109570} & \numprint{2653.1} & \numprint{1253.3} & \numprint{36940.0} & \numprint{2347.6}\\
    $[1201,1400]$ &\numprint{905} & \numprint{2001} & \numprint{152344} & \numprint{8889.0} & \numprint{5517.0} & \numprint{161327.7} & \numprint{11855.8}\\
    $[1401,1600]$ &\numprint{995} & \numprint{2284} & \numprint{203492} & \numprint{24841.1} & \numprint{21256.5} & \numprint{349656.8} & \numprint{67233.4}\\
    $[1601,1800]$ &\numprint{1032} & \numprint{2420} & \numprint{264506} & \numprint{63476.2} & \numprint{59392.6} & \numprint{952669.0} & \numprint{194405.8}\\
    $[1801,2000]$ &\numprint{1157} & \numprint{2378} & \numprint{331348} & \numprint{79307.2} & \numprint{175890.1} & \numprint{1752685.4} & \numprint{596779.2}\\
    \hline
    & \numprint{7621} & \numprint{15731} & \numprint{1206658} & \numprint{180089.5} & \numprint{263771.8} & \numprint{3266135.9} & \numprint{874006.0}\\
    \end{tabular}
    }
\\[6pt]
\text{\parbox{\textwidth}{\centering Table \ref{tab:timings1}: \magma{} 2.24-7 vs.\ \pari{} 2.12.1 (Intel Xeon W-2155 3.3GHz);\\
timings for newspaces $S$ of level $N\ge 1$, weight $k>1$ by $Nk^2$ range}}\\[4pt]
\end{gathered}
\end{equation}

\begin{equation} \label{tab:timings2}\addtocounter{equation}{1} \notag
\begin{gathered}
{\renewcommand{\arraystretch}{1.2}
    \begin{tabular}{rl|rrr|rr|rr}
    & & & & & \multicolumn{2}{c}{split time (s)} & \multicolumn{2}{c}{total time (s)}\\
    $\#S$ & $\max\dim(f)$ & num $S$ & num $f$ & $\sum \dim(S)$ & \magma{} & \pari{}  & \magma{} & \pari{} \\\hline
    $ 1$ & $[1,200]$ & \numprint{2859} & \numprint{2859} & \numprint{161375} & \numprint{423.8} & \numprint{529.9} & \numprint{11967.2} & \numprint{818.0}\\
    $ 1$ & $[201,2000]$ & \numprint{1027} & \numprint{1027} & \numprint{544092} & \numprint{26060.6} & \numprint{55272.8} & \numprint{701094.2} & \numprint{53727.6}\\
    $ 1$ & $[2001,\infty]$ & \numprint{65} & \numprint{65} & \numprint{215016} & \numprint{146044.1} & \numprint{170751.3} & \numprint{2226371.4} & \numprint{163789.9}\\
    $ 2$ & $[1,200]$ & \numprint{1703} & \numprint{3406} & \numprint{100080} & \numprint{278.7} & \numprint{660.9} & \numprint{4233.8} & \numprint{30837.0}\\
    $ 2$ & $[201,2000]$ & \numprint{145} & \numprint{290} & \numprint{95704} & \numprint{4192.1} & \numprint{8188.7} & \numprint{188745.7} & \numprint{576764.6}\\
    $ 2$ & $[2001,\infty]$ & \numprint{4} & \numprint{8} & \numprint{10870} & \numprint{2636.5} & \numprint{26821.1} & \numprint{97329.8} & \numprint{24785.1}\\
    $\ge 3$ & $[1,20]$ & \numprint{873} & \numprint{4785} & \numprint{19282} & \numprint{46.2} & \numprint{64.1} & \numprint{1596.2} & \numprint{1197.9}\\
    $\ge 3$ & $[21,200]$ & \numprint{235} & \numprint{1155} & \numprint{23135} & \numprint{160.8} & \numprint{275.7} & \numprint{1228.8} & \numprint{5255.3}\\
    $\ge 3$ & $[201,\infty]$ & \numprint{3} & \numprint{15} & \numprint{1024} & \numprint{12.0} & \numprint{347.7} & \numprint{1364.5} & \numprint{357.4}\\
    \hline
    & & \numprint{7621} & \numprint{15731} & \numprint{1206658} & \numprint{180089.5} & \numprint{263771.8} & \numprint{3266135.9} & \numprint{874006.0}\\
    \end{tabular}
}\\[6pt]
\text{\parbox{\textwidth}{\centering Table \ref{tab:timings2}: \magma{} 2.24-7 vs.\ \pari{} 2.12.1 (Intel Xeon W-2155 3.3GHz); \\ timings for newspaces $S$ of level $N\ge 1$, weight $k>1$, $Nk^2\le 2000$ by $\#S\colonequals \#\{f\in S\}$.}}\\[4pt]
\end{gathered}
\end{equation}

As can be seen in Tables~\ref{tab:timings1} and ~\ref{tab:timings2}, the explicit trace formula approach used by \pari{} is faster overall than the modular symbol method implemented in \magma{}, especially for spaces that consists of a single Galois orbit, but for newspaces that split into multiple Galois orbits it is typically slower, and in general \magma{} is able to decompose newspaces into Galois orbits more quickly than \pari{}.  The large advantage \pari{} has on irreducible spaces is due to the fact that in this situation we can use \texttt{mfsplit} to determine that the space is irreducible without actually computing any eigenforms, and then use
\texttt{mftraceform} to compute the trace form for the entire space.

In Table~\ref{tab:hardspaces} we list the 10 newspaces in our chosen range that were the computationally most difficult for either \magma{} or \pari{}.
In each case, the 10 most time consuming newspaces accounted for approximately half of the total time to process the 7621 nonzero newspaces in our test range.

Notably, only two newspaces (\newformlink{467.2.c} and \newformlink{497.2.c}) were among the computationally most difficult for both methods (these are the two newspaces of largest dimension in our chosen range).  Most of the newspaces listed in Table~\ref{tab:hardspaces} were computationally much more difficult for one of the two methods: on the largest irreducible spaces in our test range \pari{} is typically at least ten times as fast as \magma{}, but for newspaces that split into two large Galois orbits  \magma{} is faster than \pari{} by a similar (or even larger) factor.  This suggests that the optimal approach is to use the explicit trace formula and modular symbol methods in combination.  Indeed, a hybrid approach that uses \magma{} to decompose the space, and then delegates the computation to \pari{} whenever the newspace contains a Galois orbit of dimension at least 2/3 the dimension of the newspace, takes a total of \numprint{264726} seconds; this is more than 3 times faster than using \pari{} alone and more than 10 times faster than using \magma{} alone.

\begin{equation} \label{tab:hardspaces}\addtocounter{equation}{1} \notag
\begin{gathered}
{\renewcommand{\arraystretch}{1.2}
    \begin{tabular}{l|rc|rr}
    newspace & $[\Q(\chi):\Q]$ & decomposition & \magma{}(s) & \pari{}(s)\\\hline
    \href{http://www.lmfdb.org/ModularForm/GL2/Q/holomorphic/413/2/i/}{\textsf{413.2.i}} & 28 & $420+420$ & \numprint{68.91} & \numprint{23711.71}\\
    \href{http://www.lmfdb.org/ModularForm/GL2/Q/holomorphic/419/2/g/}{\textsf{419.2.g}} & 180 & $6120$ & \numprint{79654.08} & \numprint{5175.04}\\
    \href{http://www.lmfdb.org/ModularForm/GL2/Q/holomorphic/424/2/v/}{\textsf{424.2.v}} & 24 & $1248$ & \numprint{60111.92} & \numprint{150.62}\\
    \href{http://www.lmfdb.org/ModularForm/GL2/Q/holomorphic/431/2/g/}{\textsf{431.2.g}} & 168 & $5880$ & \numprint{82907.51} & \numprint{5333.59}\\
    \href{http://www.lmfdb.org/ModularForm/GL2/Q/holomorphic/435/2/bf/}{\textsf{435.2.bf}} & 12 & $240+240$ & \numprint{39.63} & \numprint{25272.26}\\
    \href{http://www.lmfdb.org/ModularForm/GL2/Q/holomorphic/443/2/g/}{\textsf{443.2.g}} & 192 & $6912$ & \numprint{180453.61} & \numprint{8134.21}\\
    \href{http://www.lmfdb.org/ModularForm/GL2/Q/holomorphic/454/2/c/}{\textsf{454.2.c}} & 112 & $1008+1120$ & \numprint{1197.47} & \numprint{44216.52}\\
    \href{http://www.lmfdb.org/ModularForm/GL2/Q/holomorphic/467/2/c/}{\textsf{467.2.c}} & 232 & $8816$ & \numprint{370791.77} & \numprint{22719.24}\\
    \href{http://www.lmfdb.org/ModularForm/GL2/Q/holomorphic/472/2/l/}{\textsf{472.2.l}} & 28 & $56+1568$ & \numprint{103117.42} & \numprint{562.40}\\
    \href{http://www.lmfdb.org/ModularForm/GL2/Q/holomorphic/478/2/g/}{\textsf{478.2.g}} & 96 & $960+960$ & \numprint{861.98} & \numprint{87147.90}\\
    \href{http://www.lmfdb.org/ModularForm/GL2/Q/holomorphic/479/2/c/}{\textsf{479.2.c}} & 238 & $9282$ & \numprint{363002.59} & \numprint{26148.67}\\
    \href{http://www.lmfdb.org/ModularForm/GL2/Q/holomorphic/486/2/i/}{\textsf{486.2.i}} & 54 & $702+756$ & \numprint{351.60} & \numprint{139762.27}\\
    \href{http://www.lmfdb.org/ModularForm/GL2/Q/holomorphic/487/2/k/}{\textsf{487.2.k}} & 162 & $6480$ & \numprint{110903.14} & \numprint{6766.85}\\
    \href{http://www.lmfdb.org/ModularForm/GL2/Q/holomorphic/489/2/q/}{\textsf{489.2.q}} & 54 & $702+756$ & \numprint{202.91} & \numprint{38345.59}\\
    \href{http://www.lmfdb.org/ModularForm/GL2/Q/holomorphic/491/2/k/}{\textsf{491.2.k}} & 168 & $6720$ & \numprint{121405.39} & \numprint{8558.33}\\
    \href{http://www.lmfdb.org/ModularForm/GL2/Q/holomorphic/497/2/v/}{\textsf{497.2.v}} & 24 & $408+456$ & \numprint{99.20} & \numprint{18438.91}\\
    \href{http://www.lmfdb.org/ModularForm/GL2/Q/holomorphic/498/2/f/}{\textsf{498.2.f}} & 40 & $560+560$ & \numprint{269.01} & \numprint{48844.21}\\
    \href{http://www.lmfdb.org/ModularForm/GL2/Q/holomorphic/499/2/g/}{\textsf{499.2.g}} & 164 & $6724$ & \numprint{119807.20} & \numprint{12148.53}\\
    \end{tabular}
} \\[6pt]
\text{Table \ref{tab:hardspaces}: Some computationally challenging newspaces}\\[4pt]
\end{gathered}
\end{equation}

\subsection{A trace formula implementation with complex coefficients} \label{sec:bober}

In this section, we describe an implementation of the trace formula using ball arithmetic over the complex numbers due to Bober \cite{bober}.  This implementation follows the description given in \secref{sec:algorithmTF}. The main focus here is to compute a moderate number of coefficients for all of the newforms in a given space $\Sknew{k}(N, \chi)$ as approximate complex numbers, which enables the computation of modular form $L$-functions at small height, for example. These computations use \arb{}~\cite{arb}, a C library which implements arbitrary precision ball arithmetic, so that we can ensure that all of our computations come with rigorous error bounds. There is also some facility for computing with coefficients in a finite field $\F_\ell$, where $\ell$ is some prime congruent to $1$ modulo the order of $\chi$, which is used in the computation of characteristic polynomials of Hecke operators, for example, and in other auxiliary pieces. The package also contains some limited functionality to compute information about spaces of \weightone{} modular forms, which we do not discuss here.

We describe briefly some details of how this implementation works in practice.

To start a computation we first choose a prime and determine a set of trace forms which will give a full rank basis of the space of newforms modulo this prime, avoiding any issues of computing the rank of a matrix with floating point entries. Specifically, we find some matrix of coefficients $(\Tr(T(m_i)T(n_j)\,|\,\Sknew{k}(N,\chi)))_{1\le i,j \le d}$ that has full rank, and we also choose our $m_i$ and $n_j$ so that $\gcd(m_i n_j, N) = 1$, which will make later computations easier.
Once we know which computation will give us a full rank matrix, we do the computation again over the complex numbers, computing additional coefficients so that we will be able to compute the action of Hecke operators. At this point we find a sum of Hecke operators $T = \sum_n c_n T_n$ such that the characteristic polynomial of $T$ is squarefree (keeping $c_n = 0$ when $\gcd(n, N) \ne 1$).

The diagonalization of $T$ would in general be a computation over the complex numbers, but because we have chosen to only use Hecke operators coprime to the level, we can use knowledge of the arguments of the eigenvalues to turn this into a problem of diagonalizing a real symmetric matrix.  This problem is solved by an implementation of Jacobi's algorithm, certifying the result as described in \secref{sec:certifyingeigenvalues}. Once we have diagonalized, we obtain a change of basis matrix that takes our trace form basis to the newform basis, and we can compute as many coefficients of newforms as we like by evaluation of the trace formula.

Once we have computed all the embeddings of all of our newforms, we may also wish to compute the decomposition of the space into Hecke-irreducible subspaces. To do this we will compute the characteristic polynomial of a linear combination $T$ of Hecke operators (it is sufficient to find one that is squarefree). If we have enough precision in the Hecke eigenvalues we have computed, we can do this simply by forming the product $\prod_\lambda (x - \lambda)$, where $\lambda$ ranges over the eigenvalues of $T$. In general we will find that we do not have enough accuracy to uniquely identify a polynomial with coefficients in $\Z[x]$, however, and we refine the computation by computing this polynomial modulo $\ell$ for enough small primes $\ell$ to obtain the polynomial exactly.

The factorization of this Hecke polynomial gives the decomposition of $\Sknew{k}(N, \chi)$ into Hecke irreducible subspaces. However, there is still one more problem we may be faced with: namely, identifying which embeddings correspond to which subspaces. In this problem we have a set of polynomials $f_1, f_2, \ldots, f_n$ and a set of approximations of complex numbers $r_1, r_2, \ldots, r_d$, and all we need to do is determine which complex number is a root of which polynomial. This may seem like a relatively trivial problem, but in fact these polynomials may be enormous and obtaining enough precision in the roots to solve this by simple evaluation may not be feasible. 

\begin{example}
To give a moderately-sized example, we can consider the space 
\href{http://www.lmfdb.org/ModularForm/GL2/Q/holomorphic/766/2/c/}{\textsf{766.2.c}}. This space is only 32-dimensional over the field of definition of $\chi$, but there are 190 Galois conjugate characters to consider, so the full degree is 6080. The characteristic polynomial of $T_3$ acting on the space $\Sknew{2}(766, [\chi])$ is squarefree and factors into 2 irreducible factors of degree 3040; each factor contains more than 1.5 million decimal digits.
\end{example}

To make this problem tractable, we again make use of the arguments of the eigenvalues. Let $f(x)$ be one of these irreducible factors. We know that each root of $f$ can be written as $\zeta t$ for some root of unity $\zeta$ and some real number $t$, and we find that $t$ is a root of the greatest common divisor of $f(\zeta x)$ and $f(\zeta^{-1}x)$ in $\mathbb{Q}(\zeta)[x]$. In fact, as we prefer to work with real numbers, we compute 
\[ \gcd(f(\zeta x) + f(\zeta^{-1}x),i(f(\zeta x) - f(\zeta^{-1}x))) \in \Q(\zeta + \zeta^{-1})[x]. \]
In principle, computing this greatest common divisor when we have only floating point approximations available could be troublesome, but it is possible because we know what its degree is.

\section{Issues: computational, theoretical, and practical} \label{sec:issues}

\subsection{Analytic conductor}

Earlier efforts to tabulate modular forms have tended to compute all newforms
in particular boxes, where the weight and level each vary in a specified range.
This approach is easy to describe, but the computational complexity of finding
newforms with simultaneously large weight and level ensures that some newforms
of interest will be missed (either large weight or large level).
Instead of working with boxes, we organized our computation around a single
invariant which scales with the complexity of the newform.

Introduced by Iwaniec--Sarnak \cite[Eq. (31)]{IS} (see also Iwaniec--Kowalski \cite[(5.7)]{IK}), the analytic conductor of a newform $f \in S_k^{\mathrm{new}}(N,\chi)$ is the positive real number 
\begin{equation}
A \colonequals N\left(\frac{\exp(\psi(k/2))}{2\pi}\right)^2,
\end{equation}
where $\psi(x):=\Gamma'(x)/\Gamma(x)$ is the logarithmic derivative of the Gamma function.  The analytic conductor includes a factor that can be thought of as measuring the complexity at infinity.  We have $A \sim \displaystyle{\frac{Nk^2}{16\pi^2}}$ as $k \to \infty$, so for simplicity we organized our computations by specifying bounds on the quantity $Nk^2$.

\subsection{Sturm bound} \label{sec:sturm}

In this section, we elaborate upon bounds for truncations of $q$-expansions of modular forms that determine them uniquely.  The most well-known of these bounds is due to Hecke (and more generally to Sturm \cite[Theorem 1]{Sturm}).

\begin{thm}[Hecke, Sturm] \label{thm:heckesturm}
Let $\Gamma \leq \SL_2(\Z)$ be a congruence subgroup and let $f$ be a modular form of weight $k$ for $\Gamma$.  Then $f=0$ if and only if $a_n(f)=0$ for all $0 \leq n \leq k[\SL_2(\Z):\Gamma]/12$.  
\end{thm}

In fact, for modular forms with character as in our setting, one can apply a sharper bound (as though it was without character) as follows.  

\begin{defn}
For $k,N \in \Z_{\geq 1}$, define the \defi{(Hecke-)Sturm bound} 
\[ \Sturm(k,N) \colonequals \frac{k}{12} [\SL_2(\Z):\Gamma_0(N)] = \frac{Nk}{12} \prod_{p \mid N} \left(1+\frac{1}{p}\right). \]
\end{defn}

\begin{prop}[Hecke, Sturm] \label{prop:heckesturm}
Let $N,k \geq 1$ and let $\chi$ be a character of modulus $N$.  Let $\TT \subseteq \End_\C(S_k(N,\chi))$ be the $\Z$-subalgebra generated by the Hecke operators $T_n$ for all $n \in \Z_{\geq 1}$, and let $\Z[\chi] \subseteq \C$ be the $\Z$-subalgebra generated by the values of $\chi$.  Then there is a natural inclusion $\Z[\chi] \hookrightarrow \TT$, and the following statements hold.
\begin{enumalph}
\item If $f \in S_k(N,\chi)$ has $a_n(f)=0$ for all $n \leq \Sturm(k,N)$, then $f=0$.   
\item $\TT$ is generated as a $\Z[\chi]$-module by $T_n$ for all $n \leq \Sturm(k,N)$.
\item $\TT$ is generated as a $\Z[\chi]$-algebra by $T_1$ and $T_p$ for all primes $p \leq \Sturm(k,N)$.
\end{enumalph}
\end{prop}

\begin{proof}
For the inclusion $\Z[\chi] \hookrightarrow \TT$, we argue as follows: from the Hecke recursion 
\begin{equation}  \label{eqn:Tp2}
T_{p^2} - T_p^2 + \chi(p)p^{k-1}=0 
\end{equation}
for $p \nmid N$, we see that $\chi(p)p^{k-1} \in \TT$; choosing two distinct primes $p$ congruent modulo $N$ and applying the CRT shows that $\Z[\chi] \subseteq \TT$.  Consequently, $\TT$ is a torsion free $\Z[\chi]$-module.  Since $\Z[\chi]$ is a Dedekind domain, $\TT$ is locally free.  

Abbreviate $S \colonequals S_k(N,\chi;\Z[\chi])$.  We claim that the pairing
\begin{equation} \label{eqn:TTSchi}
\begin{aligned} 
\TT \times S &\to \Z[\chi] \\
(T,f) &\mapsto a_1(Tf)
\end{aligned}
\end{equation}
is perfect, i.e., the map
\begin{equation}
\begin{aligned} 
\varphi \colon S &\to \Hom_{\Z[\chi]}(\TT,\Z[\chi]) \\
f &\mapsto (T \mapsto a_1(Tf))
\end{aligned}
\end{equation}
is an isomorphism.  When $\Z[\chi]=\Z$, this is an argument of Ribet \cite[Theorem (2.2)]{Ribet:modp}, and we only need to make a small modification.  The map $\varphi$ is injective with torsion-free cokernel because $a_1 \circ T_n = a_n$ and the map taking a form to its $q$-expansion is injective.  Since $S$ and $\TT$ are locally free $\Z[\chi]$-modules of finite rank, it suffices to show that the rank of $\TT$ is at most the rank of $S$ (localizing at primes $\frakl$ of $\Z[\chi]$).  To this end, consider the other map induced by the pairing, namely,
\begin{equation} \label{eqn:omegaTT}
\begin{aligned} 
\omega \colon \TT &\to \Hom_{\Z[\chi]}(S,\Z[\chi]) \\
T &\mapsto (f \mapsto a_1(Tf)).
\end{aligned}
\end{equation}
We claim that $\omega$ is injective.  Indeed, if $T \in \ker \omega$, then for all $f \in S$ and all $n \geq 1$ we have
\[ 0 = \omega(T)(T_n f) = a_1(TT_n f) = a_1(T_n T f) = a_n(Tf) \]
as $\TT$ is commutative.  Since the $q$-expansion map is injective, we conclude $Tf = 0$ for all $f$, so $T=0$ as an endomorphism, proving the claim.  Finally, localizing $\omega$ at $\frakl$, the injectivity of $\omega$ implies the desired rank bound.  

We now prove (a) following Buzzard, and suppose that $f \in S_k(N,\chi;\Z[\chi])$ has $a_n(f) = 0$ for all $n \leq \Sturm(k,N)$.  Let $d$ be the order of $\chi$, let $s \colonequals \Sturm(k,N)$, and consider $f^d \in S_{dk}(\Gamma_0(N);\Z[\chi])$.  Since $f = O(q^{s+1})$, we have $f^d = O(q^{d(s+1)})$.  Moreover, $\Sturm(kd,N)=ds$, so by the Sturm bound (Theorem \ref{thm:heckesturm}) applied to $S_{dk}(\Gamma_0(N);\Z[\chi])$ we conclude $f^d = 0$, which implies $f=0$.  

To prove (b), let $\TT_{\leq n} \subseteq \TT$ be the $\Z[\chi]$-submodule generated by $T_n$ with $n \leq \Sturm(k,N)$.  By the previous paragraph, the pairing \eqref{eqn:TTSchi} restricted to $\TT_{\leq n}$ is still perfect; indeed we can simply argue with $\TT_{\leq n}$ in the injectivity of $\omega$ in \eqref{eqn:omegaTT}.  We conclude that $\TT_{\leq n} = \TT$.

For part (c), we use multiplicativity to see that $T_n$ for $n$ composite is contained in the algebra generated by the prime power operators $T_{p^e}$, and then the Hecke recurrence and induction to see that $T_{p^e}$ is contained in the algebra generated by $T_1$ and $T_p$.
\end{proof}

\subsection{Atkin--Lehner operators and eigenvalues}\label{SS:AtkinLehner}
Let $\chi$ be a Dirichlet character modulo $N$. 
For $M\mid N$ with $\gcd(M, N/M)=1$, there are unique characters $\chi_M\pmod{M}$ and $\chi_{N/M}\pmod{N/M}$ such that $\chi=\chi_M \chi_{N/M}$. 
The Atkin--Lehner--Li operator $W_M$ \cite[\S1]{Atkin-Li} maps the space $S_k(N, \chi)$ to $S_k(N, \overline{\chi_M} \chi_{N/M})$. 
In general $\overline{\chi_M}\chi_{N/M}$ is different from $\chi$, so then this operator cannot be used for splitting up spaces. 
We have $\overline{\chi_M}\chi_{N/M}=\chi$ when the character $\chi_{M}$ is trivial or quadratic, and in these cases, $W_M$ is an involution on the space $S_k(N, \chi)$, which then splits as the direct sum of $\pm 1$-eigenspaces.   \magma{} only implements Atkin-Lehner operators on spaces with  trivial character, where they commute with all Hecke operators and hence map every newform $f$ to $\pm f$, and the sign~$\pm$ is the \emph{Atkin-Lehner eigenvalue} of~$f$ with respect to~$M$.  (By a common abuse of notation and terminology, when $M$ is the power of  a prime~$q$ not dividing $N/M$, the operator $W_M$ is often denoted $W_q$.)  In our computations we only compute Atkin-Lehner eigenvalues on newforms with trivial character.

In general, the image of a normalized newform~$f$ in $S_k(N, \chi_M\chi_{N/M})$ under $W_M$ is a multiple of a normalized newform in $S_k(N, \overline{\chi_M}\chi_{N/M})$, and the multiple, not necessarily~$\pm1$, is called the \emph{pseudo-eigenvalue} of~$W_M$ on~$f$.   Atkin--Li \cite{Atkin-Li} do not give a general formula for pseudo-eigenvalues, which are not always easy to compute in practice. See also Belabas--Cohen \cite[\S\S 5--6]{BelabasCohen}.

When $M=1$ the operator $W_M$ is trivial, while when $M=N$ it is called the \emph{Fricke involution}.   The Fricke involution is the product of all $W_q$ for primes $q\mid N$ (using the convention of the previous paragraph.)  For a newform of weight $k$ and trivial character, the Fricke eigenvalue $\epsilon$ is related to the sign $\varepsilon$ that appears in the functional equation \eqref{eqn:functionalequation} via $\epsilon=(-1)^{k/2}\varepsilon$, see Miyake \cite[Cor.~4.3.7]{Miyake}. Each $W_q$-eigenvalue is similarly related to the sign of a certain local functional equation.

\subsection{Self-duality}

The coefficient field of a newform $f \in S_k(N, \chi)$ is either totally real or CM \cite[Prop 3.2]{Ribet:galreps};
we say that $f$ is \defi{self-dual} in the totally real case.
Computing the coefficient field can be time consuming,
so we use the following easier criteria when applicable.

\begin{prop}[Ribet]
Let $f \in S_k(N, \chi)$ have Hecke orbit of dimension $d$ and trace form $\sum_{n=0}^\infty t_n q^n$.  Then the following statements hold.
\begin{enumalph}
    \item If $\chi$ is trivial or $d$ is odd, then $f$ is self-dual;
    \item If $\chi$ has order larger than 2, then $f$ is not self-dual;
    \item If there exists a prime $p$ so that $t_p \ne 0$ and $\chi(p) \ne 1$, then $f$ is not self-dual. 
\end{enumalph}
\end{prop}

\begin{proof}
See Ribet {\cite[Propositions 3.2 and 3.3]{Ribet:galreps}}.
\end{proof}

\subsection{Efficiently recognizing irreducibility} \label{sec:irred}

Level $N=2$ is by far the most time-consuming case for \magma{} (note that this allows for the largest range of $k$ with $N\ne 1$ for any given bound on the analytic conductor). For $k>400$ with $4 \mid k$, each space takes more than 12 hours of CPU time.  However, we observed behavior analogous to the Maeda conjecture in weight $1$ up to weight $k \leq 400$.  The Atkin--Lehner operator $W_2$ splits the space as evenly as possible, and the $W_2$-eigenspaces appear to always be irreducible.

\begin{conj}
For all $k \geq 2$, the space $\Sknew{k}(\Gamma_0(2))$ decomposes under the Atkin--Lehner operator~$W_2$ into Hecke irreducible subspaces of dimensions $\lfloor d/2\rfloor$\! and $\lceil d/2 \rceil$, where $d\! \colonequals\! \dim_\C \Sknew{k}(\Gamma_0(2))$.
\end{conj}

The dimensions in the corollary follow from work of Martin \cite[Thm.\ 2.2]{Martin}, which implies that for even weights $k>2$ we have
\begin{equation}\label{eqn:Sk2split}
\dim \Sknew{k}(\Gamma_0(2))^+ - \dim \Sknew{k}(\Gamma_0(2))^- = \begin{cases}
 0 & k = 4,6\bmod 8,\\
 (-1)^{k/2} & k \equiv 0,2\bmod 8,\\
 \end{cases}
\end{equation}
it is only the irreducibility of the eigenspaces that is conjectural.
The factor $(-1)^{k/2}$ in \eqref{eqn:Sk2split} appears as $1$ in \cite[Thm.\ 2.2]{Martin} because there the Atkin-Lehner operator follows the convention of Diamond--Shurman \cite[p.\ 209]{DiamondShurman}, which includes a factor of $(-1)^{k/2}$, while we are following the convention of Miyake \cite{Miyake}, which does not include this factor.

One can find similar formulas for $\dim \Sknew{k}(\Gamma_0(N))^+ - \dim \Sknew{k}(\Gamma_0(N))^-$ for any squarefree $N$ in Martin \cite{Martin}, in which case they are a linear function of the class number $h(-4N)$.  For general $N>4$ not of the form $M^2,2M^2,3M^2,4M^2$ with $M$ squarefree, we refer the reader to Helfgott \cite[pp.\ 266--267]{helfgott}.

\begin{ques}
Given an $n \times n$ matrix with entries in $\Z[\zeta_m]$ (typically sparse), is there a fast algorithm that with positive probability correctly determines when its characteristic polynomial is irreducible?
\end{ques}

In other words, if you expect that a polynomial is irreducible, can you verify this quickly without factoring the polynomial?  Under the expectation that the Galois group of the polynomial is transitive and therefore likely to be $S_d$, one could succeed in some cases by factoring the polynomial modulo primes.  This is different than the typical factorization problems solved in computer algebra systems, which compute a factorization $p$-adically and then reconstruct the factorization over $\Z$.  (See Table~\ref{tab:hardspaces} for some difficult spaces where this would help.)  A natural generalization of this question would be to efficiently determine the degrees of the irreducible factors of its characteristic polynomial (without explicitly computing it).

\subsection{Trace form} \label{sec:traceform}

As defined in \secref{sub:decomposition}, each newform $f\in \Sknew{k}(N,[\chi])$ has an associated trace form $\Tr(f)(q) = \sum_n t_n q^n \in S_k(N,[\chi];\Z)$ equal to the sum of the distinct $\Gal(\Qbar\,|\,\Q)$ conjugates of $f$ and thereby well-defined on its Galois orbit $[f]$.  More precisely, we have $\Tr(f) \in \Sknew{k}(\Gamma)$ where
\[ \Gamma \colonequals \{\left(\begin{smallmatrix}a&b\\c&d\end{smallmatrix}\right)\in \Gamma_0(N):\chi(a)=1\}\supseteq \Gamma_1(N), \]
(but in general $\Tr(f) \not\in S_k(N,\chi)$).  One can thus apply the Sturm bound (Theorem \ref{thm:heckesturm}) for $\Gamma$: trace forms of newforms in $S_k(N,\chi)$ with the same Fourier coefficients $a_n$ for $n\le k[\SL_2(\Z):\Gamma]/12$ must coincide, but note that this will typically be larger than the Sturm bound $\Sturm(k,N)$.  
As noted above, we always have $t_1=[K:\Q]$, where $K=\Q(f)$ is the coefficient field of $f$.

Trace forms can be efficiently computed using the trace formula.  In the common case where $S_k(N,\chi)$ is irreducible, this can be done via the \pari{} function \texttt{mftraceform} \cite{Pari}, which is dramatically faster than computing the coefficients of $f$ as elements of $K$ and taking traces (indeed, it is not even necessary to determine $K$).
More generally, if one knows the decomposition of $S_k(N,[\chi])$ into newform subspaces and has computed trace forms for all but one of them, the remaining trace form can be computed by subtracting corresponding coefficients of the known trace forms from the coefficients given by \texttt{mftraceform}, which computes absolute traces of the Hecke operators $T_n$ acting on the entire newspace $S_k(N,[\chi])$.  Alternatively, one can sum complex coefficients $a_n$ of the Galois conjugates of~$f$ computed to sufficient precision to allow the sum to be identified as a unique integer.

The coefficients $t_p$ of the trace form at primes $p$ are equal to the Dirichlet coefficients of the (typically imprimitive) $L$-function $L(s)=\sum b_n n^{-s}$ with integer Dirichlet coefficients $b_n$ obtained by taking the product of the $L$-functions of the Galois conjugates of $f$.  But for nonprime values of $n$ the integer coefficients $t_n$ do not match the integer coefficients $b_n$ unless the newspace has dimension one (in which case $L(s)$ is primitive).  Indeed $t_1=[K:\Q]$ cannot coincide with $b_1=1$, and in general the coefficients $t_n$ at nonprime values of $n$ encode different information. 

\subsection{Presenting coefficients using LLL-reduction}  \label{sec:LLLbasis}

One of the most dramatic improvements we saw, both in performance and in display, is in the choice of how to represent coefficient rings.  In this section and the next, we explain two such methods.  

As explained in section \ref{sub:decomposition}, one computes a minimal polynomial for the coefficient field by factoring the characteristic polynomial of a Hecke operator.  This polynomial may be unwieldy!  So we first apply the \pari{} function \texttt{polredbest} which finds an improved minimal polynomial representing the same field by computing an LLL-reduced basis for an order with respect to the Minkowski embedding (whose underlying quadratic form is given by the $T_2$-norm)---this runs in deterministic polynomial time in the size of the input.  When possible, we improve this to \texttt{polredabs}, which applies the same technique but to the maximal order (this may require factoring a discriminant, and we frequently encounter situations where this is a bottleneck).  

\begin{remark}
The function \texttt{polredabs} changed in \pari{} 2.9.5 (Fall 2017); we use the more recent version, described in \cite{CPS}.
\end{remark}

We can make significant further improvements by optimizing the $\Z$-basis we use to represent coefficients.  Let $f \in \Sknew{k}(N,\chi)$ be a newform.  By the Hecke--Sturm bound (Proposition \ref{prop:heckesturm}), the coefficient ring of $f$ is generated over $\Z[\chi]$ by the values $a_n(f)$ for $n \leq \Sturm(k,N)$, so by extension we obtain a set of $\Z$-module generators for the ring.  We reduce this to an LLL-reduced $\Z$-basis for the coefficient ring, and we rewrite the coefficients in terms of this basis.  In our computations, we always use complex precision that is at least as large as the discriminant of the coefficient ring.  

We observe the following.

\begin{lem}
Let $F$ be a number field and let $R \subseteq F$ be a $\Z$-order in $F$.  Then the shortest vectors in $R$ with respect to the $T_2$-norm are exactly the roots of unity in $R$.  
\end{lem}

\begin{proof}
The order $R$ contains $1$, and $T_2(1)=n \colonequals [F:\Q]$.  More generally, for any root of unity $\zeta \in R$, we have $T_2(\zeta)=n$.  Conversely, let $\alpha \in R$ have $T_2(\alpha)\le n$.  Then by the arithmetic-geometric mean inequality, we have
\[ 1 \geq \frac{T_2(\alpha)}{n} = \frac{1}{n} \sum_{i=1}^n \left|\alpha_i\right|^2 \geq \prod_{i=1}^{n} \left|\alpha_i\right|^{2/n} = \left|\Nm_{F|\Q} \alpha\right|^{2/n} \]
with equality if and only if $\left|\alpha_1\right|=\cdots=\left|\alpha_n\right|=1$.  But $\alpha \in R$ is integral, so $\left|\Nm_{F|\Q} \alpha \right| \geq 1$, so equality holds.  By Kronecker's theorem, we conclude that $\alpha$ is a root of unity.
\end{proof}

In spite of this lemma, because of nonuniqueness, we may not have $1$ as an element of an LLL-reduced basis as there may be more roots of unity than the degree, such as in a cyclotomic field.  However, using the above lemma we can recognize the roots of unity in the coefficient ring and thereby recognize when the coefficient ring is cyclotomic itself, where we may institute a canonical basis (see also the next section).  

The effect of such a representation is dramatic.  

\begin{example}
Consider the newform \newformlink{153.2.e.c}.  Its coefficient field is $\Q(\nu)$ where $\nu$ has minimal polynomial
\begin{align*}
&x^{20} - x^{19} + 3x^{18} + 2x^{17} + 13x^{16} - 12x^{15} + 54x^{14} + 27x^{13} + 93x^{12} - 54x^{11} + 693x^{10} - 162x^{9} \\
&\qquad + 837x^{8} + 729x^{7} + 4374x^{6} - 2916x^{5} + 9477x^{4} + 4374x^{3} + 19683x^{2} - 19683x + 59049.
\end{align*}
An integral basis written in terms of the powers of $\nu$ is too large to record here, and similarly the coefficients of $f$ written in a power basis are enormous!  

However, in terms of an LLL-reduced basis $\beta_{0}=1,\cdots,\beta_{19}$, we have coefficients
\begin{align*}
a_2(f) &= \beta_{16} \\
a_3(f) &= -\beta_{10} \\
a_4(f) &= -1-\beta_3-\beta_5 \\
\vdots \\
a_{57}(f) &= \beta_{2} - \beta_{3} - \beta_{4} + \beta_{5} - 3 \beta_{7} + 2 \beta_{9} - \beta_{10} \\
&\qquad + 2 \beta_{13} + \beta_{14} + \beta_{15} - \beta_{16} + 2 \beta_{17} - \beta_{18} - \beta_{19} \\
\vdots
\end{align*}
that are very small integer linear combinations of the basis elements.  Moreover, we have
\begin{align*}
1 &= \beta_0 \\
\nu &= \beta_1 \\
\nu^{2} &= \beta_{8} - \beta_{7} - \beta_{4}+ \beta_{2}  \\
\vdots \\
\nu^{19} &= 5114 \beta_{19} + 2632 \beta_{18} + \cdots - 1807\beta_1 - 6756
\end{align*}
and the powers of $\nu$ are reasonably sized $\Z$-linear combinations of our basis elements.  
\end{example}

We observe that the matrix that writes the powers of a primitive element in terms of the LLL-reduced basis is noticeably smaller than the other way around. Working with the coefficient ring itself rather than a maximal order containing it is not only more efficient (as it may be prohibitively expensive to compute such a maximal order), but it also seems to give better results.

The intuitive reason that this works is simple: by the Ramanujan--Petersson bounds, the coefficients of a newform are of small size in all complex embeddings, and so it can be expected that writing it in terms of a $\Z$-basis which is LLL-reduced with respect to size provides a small linear combination.

\begin{rmk}
In the above, we have been concentrating on the case where $f$ is a newform, representing a Galois orbit of newforms, and we write down its $q$-expansion in terms of a $\Z$-basis for the coefficient ring.  

As an alternative, we can consider the $\C$-vector space spanned by $f$ and its conjugates under $\Aut(\C)$, making a $\C$-vector space of dimension say $d$.  These conjugates will include conjugates that \emph{do not} preserve the character, so we would either be working implicitly in the direct sum of the spaces over the full Galois orbit of characters, or we need to restrict to quadratic characters, or we only consider conjugates under $\Aut(\C\,|\,\Q(\chi))$ and get a $\Q(\chi)$-vector space.  However, inside this space is a canonical $\Q$-subspace, namely, those forms whose $q$-expansion belongs to $\Q[[q]]$.  So we could instead represent the Galois orbit canonically by an echelonized basis of $d$ individual $q$-expansions with coefficients in $\Q$.  We could then write a representative newform as before as a linear combination of this basis over the coefficient field.

To go from the eigenform to the $\Q$-basis, we apply the operators $\Tr(\beta_i f)$ for $\beta_i$ any $\Q$-basis for the coefficient field.  (To go from the $\Q$-basis to an eigenform one needs to retain sufficiently many eigenvalues to do the linear algebra.  In other words, the eigenform contains more information than the $\Q$-basis.)  This generalizes the trace form, which is where we take $\beta_i=\beta_0=1$.  

We could also work integrally and take the $\Z$-module of forms whose $q$-expansions belong to $\Z$ and then take a LLL-reduced basis which minimizes a (weighted) sum of finitely many coefficients.  It is conceivable that in a world where linear algebra over $\Q$ is much faster than linear algebra over number fields that we could succeed in computing a $\Q$-basis in reasonable time but not succeed in computing an eigenform.  
\end{rmk}

\subsection{Presenting coefficients using a sparse cyclotomic representation}

When the coefficient ring of a newform is contained in a cyclotomic field $\Q(\zeta_m)$, writing coefficients in terms of an LLL-optimized basis as described in the previous section does not necessarily give the most compact representation, for two reasons.  First, when the coefficient ring is not the maximal order, it may be more compact to express coefficients as elements of the maximal order $\Z[\zeta_m]$.  Second, even when the coefficient ring is the maximal order, in which case the LLL-optimized basis will typically be the standard power basis $1,\zeta_m,\zeta_m^2,\ldots,\zeta_m^{\phi(m)-1}$, the eigenvalues $a_n$ can often be written more compactly by expressing them as sparse polynomials in $\zeta_m$ rather than integer linear combinations of the power basis.  Every integer linear combination of elements of the power basis is of course also a polynomial in $\zeta_m$; the question is whether to allow polynomials of higher degree whose terms involve powers of $\zeta_m$ that are not in the power basis (because $m>\phi(m)=[\Q(\zeta_m):\Q]$), which allows more flexibility and a potentially sparser choice of polynomial.

This added flexibility is particular relevant for \weightone{} newforms, whose coefficients always lie in a cyclotomic field.  The correspondence between \weightone{} newforms and (odd irreducible) 2-dimensional Artin representations \cite{DeligneSerre74} implies that for \weightone{} newforms the eigenvalues $a_p$ can always be written as a sum of at most two roots of unity.  For composite values of $n$ the eigenvalues $a_n$ will not be as sparse, but even if one na\"ively expands them as products of polynomials in $\zeta_m$ for each $a_{p^r}$, for most $a_n$ we obtain an expression with $O(2^{\log\log n})$ nonzero coefficients (a typical integer $n$ has $\log\log n$ distinct prime factors $p$ and $p$-adic valuation 1 at all but $O(1)$ of them), which is exponentially sparser than a generic element of $\Z[\zeta_m]$ written in the power basis.
For even values of $m$ we can improve on this na\"ive approach by using the identity $\zeta_m^{m/2}=-1$ to reduce to polynomials of degree at most $m/2-1$ in $\zeta_m$; this never increases the number of terms and may reduce it.

For example, the second Fourier coefficient of the newform \newformlink{3997.1.cz.a} is
\[
a_2 = -\zeta_{201}^{570} +\zeta_{570}^{244},
\]
but when written in terms of the standard power basis $1,\zeta_{570},\ldots,\zeta_{570}^{143}$ we instead have
\begin{align*}
a_2=1&+\zeta_{570}^{2}+\zeta_{570}^{5}+\zeta_{570}^{11}-\zeta_{570}^{12}-\zeta_{570}^{15}-\zeta_{570}^{17}+\zeta_{570}^{19}-\zeta_{570}^{20}+\zeta_{570}^{21}+\zeta_{570}^{24}+\zeta_{570}^{27}+\zeta_{570}^{30}-\zeta_{570}^{31}\\
&+\zeta_{570}^{32}-\zeta_{570}^{34}+\zeta_{570}^{35}-\zeta_{570}^{36}-\zeta_{570}^{39}-\zeta_{570}^{42}-\zeta_{570}^{45}+\zeta_{570}^{46}-\zeta_{570}^{47}+\zeta_{570}^{49}-\zeta_{570}^{50}+\zeta_{570}^{51}-\zeta_{570}^{59}\\
&+\zeta_{570}^{60}-\zeta_{570}^{61}-\zeta_{570}^{64}+\zeta_{570}^{65}-\zeta_{570}^{66}+\zeta_{570}^{74}-\zeta_{570}^{75}-\zeta_{570}^{78}+\zeta_{570}^{79}-\zeta_{570}^{80}+\zeta_{570}^{88}-\zeta_{570}^{89}+\zeta_{570}^{90}\\
&+\zeta_{570}^{93}-\zeta_{570}^{94}-\zeta_{570}^{97}-\zeta_{570}^{100}-\zeta_{570}^{103}+\zeta_{570}^{104}-\zeta_{570}^{105}-\zeta_{570}^{106}+\zeta_{570}^{107}-\zeta_{570}^{108}+\zeta_{570}^{109}+\zeta_{570}^{112}+\zeta_{570}^{115}\\
&+\zeta_{570}^{118}-\zeta_{570}^{119}+\zeta_{570}^{120}-\zeta_{570}^{122}+\zeta_{570}^{123}-\zeta_{570}^{124}-\zeta_{570}^{127}-\zeta_{570}^{130}+\zeta_{570}^{134}+\zeta_{570}^{137}+\zeta_{570}^{139}+\zeta_{570}^{142}.
\end{align*}
Among the 585 nonzero $a_n$ with $n\le 2000$, the average number of terms needed to express $a_n$ as a sparse polynomial in $\zeta_m$ is $4.1$; by contrast, when written in the power basis the average number of nonzero coefficients of $a_n$ is $42.8$.  This leads to a more than tenfold reduction in storage and a corresponding reduction in the time to transmit or render the $q$-expansion.

\begin{remark}
For modular forms of weight $k>1$ with cyclotomic coefficient fields $\Q(\zeta_m)$ there is no \emph{a priori} reason to expect the $a_p$ to be expressible as sparse polynomials in $\zeta_m$, and one can see in examples that this is often not the case.  One might instead try to apply a more general approach, which, given $\alpha\in \Z[\zeta_m]$ searches for a sparse polynomial $f(\zeta_m)$ of degree less than $m$ with small coefficients that is equivalent to $\alpha$.  We do not know an efficient solution to this problem, but we note that even if one exists, for generic values of~$\alpha$ it is unlikely to result in representations that are significantly more compact than using the power basis for purely information theoretic reasons: the number of $\alpha\in \Z[\zeta_m]$ that can be expressed as $r$-term polynomials in $\zeta_m$ using $b$ bits to represent the coefficients must be approximately the same as the number of integer vectors of length $\phi(m)$ that can be encoded in $b$ bits.  For this reason we use sparse cyclotomic coefficient representations only for $k=1$.
\end{remark}

\subsection{Hecke kernels} \label{sec:heckekernel}

Having determined the decomposition of a newspace $\Sknew{k}(N,[\chi])$ into Hecke orbits $V_f$ corresponding to newforms $f$, we can compute and store information that will allow us to reconstruct a single Hecke orbit, without having to decompose the entire newspace again.  This is particularly useful when the dimension of a particular newform $f$ of interest is much smaller than that of $\Sknew{k}(N,[\chi])$.  To achieve this we compute a list of pairs $(p,g_p(X))$, where $p$ is a prime and $g\in \Z[X]$ is the minimal polynomial of the Hecke operator $T_p$ acting on $V_f$ (viewed as a $\Q$-subspace of $\Sknew{k}(N,[\chi])$), such that $V_f$ is equal to the intersection of the kernels of the linear operators $g_p(T_p)$ acting on $\Sknew{k}(N,[\chi])$, in other words, the operators $g_p(T_p)$ generate the Hecke kernel of $V_f$.  Such a list of generators can be used to reconstruct the newform $f$ in \magma{} via the \texttt{Kernel} function.

It is computationally convenient to restrict to primes $p$ not dividing the level $N$, and to use the same list of primes $p$ for all the newforms in $\Sknew{k}(N,[\chi])$.  To this end, for a set of primes $\mathcal S$, not dividing $N$, and a newform $f$, we let $X_f(\mathcal S)$ denote the set of pairs $(p,g_p)$, where $g_p\in\Z[X]$ is the minimal polynomial of $T_p$ acting on $V_f$, and say that  $\mathcal S$ is a set of \defi{distinguishing primes} for the newspace $\Sknew{k}(N,[\chi])$ if the sets $X_f(\mathcal S)$ are distinct as $f$ varies over the newforms of $\Sknew{k}(N,[\chi])$.

We construct a set of distinguishing primes as follows.  We start by taking $\mathcal S$ to be the empty set.  If the newspace $\Sknew{k}(N,[\chi])$ consists of a single Hecke orbit, then $\mathcal S$ is a set of distinguishing primes, and otherwise we increase the size of $\mathcal S$ by adding the least prime $p\nmid N$ not contained in $\mathcal S$ for which
\[
\{X_f(\mathcal S):f\in \Sknew{k}(N,[\chi])\} \subsetneq \{X_f(\mathcal S\cup\{p\}):f\in \Sknew{k}(N,[\chi])\}.
\]
We observe that the cardinality of the set $\mathcal S$ constructed in this fashion is at most one less than the number of Hecke orbits in $\Sknew{k}(N,[\chi])$.
This greedy approach to constructing $\mathcal S$ does not necessarily minimize its cardinality, but it does minimize the largest $p$ that appears in $\mathcal S$, which may be viewed as an invariant of the newspace.
For example, we may distinguish the 8 Hecke orbits of the newspace \newformlink{2608.2.g}, where $2608 = 2^4 \cdot 163$, using $p = 3$ and $41$.
In this case $T_3$ distinguishes all the forms with the exception of the two CM forms, which both have vanishing $a_p$ for all $p$ split in $\Q (\sqrt{-163})$, hence the smallest prime $p$ such that $a_p$ could possibly distinguish them is $41$, and $41$ does in fact do so.

\begin{remark}
The largest prime $p$ that appears in $\mathcal S$ may occasionally exceed the Sturm bound, as in the case of the newforms \newformlink{66.2.b} and \newformlink{735.2.p}, for example.  This fact is relevant in the context of Theorem~\ref{thm:justcheckcoprime}, which we use to determine the group of inner twists of a newform in \secref{sec:twists}. This is one reason to compute $a_p(f)$ past the Sturm bound.
\end{remark}

\section{Computing \texorpdfstring{$L$}{L}-functions rigorously} \label{sec:Lfunctions}

In this section, we describe rigorous methods to compute $L$-functions of modular forms.

\subsection{Embedded modular forms}

To a newform $f \in S_k^\mathrm{new}(N,\chi)$, with $q$-expansion $\sum a_n q^n$, for each complex embedding of the coefficient field  $\iota: \Q(f) \hookrightarrow \C$
we may consider the embedded modular form 
\begin{equation}
\iota(f)\colonequals \sum \iota(a_n) q^n,
\end{equation}
which is a modular form over the complex numbers.

We label such forms by \textsf{N.k.s.x.c.j}, where \textsf{N.k.s.x} is the label of Hecke orbit, \textsf{N.c} is the Conrey label for the character corresponding to the embedding, and
$j$ is the index for the embedding within those with the same Dirichlet character; these embeddings are ordered by the vector $\iota(a_n)$, where we order the complex numbers first by their real part and then by their imaginary part.

To such an embedded modular form $\iota(f)$, we may associate a primitive $L$-function
of  degree~2
\begin{equation}
\begin{aligned}
L(\iota(f),s) &\colonequals \sum \iota(a_n) n^{-s} = \prod_p L_p(\iota(f), p^{-s})\\
& = \prod_{p\mid N} \left(1 - \iota(a_p) p^{-s}
\right)^{-1}\prod_{p\nmid N} \left(1 - \iota(a_p) p^{-s} + \chi(p)  p^{-2s}
\right)^{-1}.
\end{aligned}
\end{equation}

Let $\Lambda(\iota(f), s)  \colonequals N^{s/2} \Gamma_{\mathbb{C}}(s)  L(\iota(f), s)$, where $\Gamma_\C(s) \colonequals 2 (2 \pi)^{-s} \Gamma(s) $.
Then $\Lambda(\iota(f), s)$ continues to an entire function of order 1 and satisfies the functional equation
\begin{equation}
\Lambda(\iota(f), s) = \varepsilon \overline{\Lambda}(\iota(f), k-s),
\label{eqn:functionalequation}
\end{equation}
where $\varepsilon$ is the root number of $\Lambda(\iota(f), s)$, a root of unity.

The generalized Riemann hypothesis also predicts that any nontrivial zero of the $L$-function lies on the line of symmetry of its functional equation $\Re(s) = k/2$, known as the critical line. 
To study the behavior of $L(\iota(f), s)$ on the critical line, it is natural to introduce the associated $Z$-function, a smooth real-valued function of a real variable $t$ defined by
\begin{equation}
Z(\iota(f), t) \colonequals \overline{\varepsilon}^{1/2} \frac{\gamma(k/2+it)}{|\gamma(k/2+it)|} 
L(\iota(f), k/2+it),
\end{equation}
where $\gamma(s) \colonequals N^{s/2} \Gamma_{\mathbb{C}}(s)$ and the square root is chosen so that $Z(t)>0$ for sufficiently small $t>0$.
By construction, we have $|Z(\iota(f), t)|=|L(\iota(f), k/2+it)|$, the multiset of zeros of $Z(\iota(f), t)$ matches the multiset of zeros of $L(\iota(f), k/2+it)$, and $Z(\iota(f), t)$ changes sign at the zeros of $L(\iota(f), k/2+it)$ of odd multiplicity.

\subsection{Computations} \label{sec:Lfunctions-comp}
Given $\iota(f)$ we would like to compute certain invariants of $L(\iota(f), s)$.
For example, the root number $\varepsilon$, the imaginary part of the first few zeros on the critical line, an upper bound on the order of vanishing at $s =  k/2$, the leading Taylor coefficient at $s =  k/2$,  and the plot $Z(\iota(f), t)$ on some interval.
Given that a majority of these items cannot be represented exactly, we instead aim to determine a small interval in $\R$ or rectangle in $\C$.
Precisely, let $b$ denote the number of bits of target accuracy.  We would like to compute the following:
\begin{itemize}
    \item
    the root number: 
    $x_\varepsilon, y_\varepsilon \in \Z$ such that $2^{b+1}\Re(z) \in [x_\varepsilon-1, x_\varepsilon + 1]$ and $2^{b+1}\Im(z) \in [y-1, y + 1]$;
    \item the imaginary part of the first few zeros on the critical line: $t_1, \dots, t_n \in \Z$ such that $\bigcup_i [t_i -1, t_i + 1]2^{-b-1}$ covers the first $n$ zeros of $L(\iota(f), k/2+it)$;
    \item
    an upper bound on the order of vanishing at $s =  k/2$: 
    $r \colonequals \max_i \{i : |L^{(i)}(\iota(f), k/2) / i!| < 2^{-b-1} \}$;
    \item 
    the leading Taylor coefficient at $s =  k/2$: 
    $0 \neq s \in \Z$ such that $2^{p + 1}L^{(r)}(\iota(f),  k/2) / r! \in [s+1, s-1]$;
    \item  
    an approximation to the plot of $Z(\iota(f), t)$:
    approximations as doubles of $Z(\iota(f), i \delta)$ for some chosen $\delta$ and $i = 0, \dots, n$.
\end{itemize}

In order to rigorously compute the items above, we follow an approach that builds on several improvements and extensions of the algorithm from~\cite{booker} specialized to the motivic case, the details of which will appear in future work~\cite{motiviclfun}.
In practice, given the first $C_k \sqrt{N}$ embedded Dirichlet coefficients, with sufficient precision, and while carrying out all floating-point calculations using rigorous error bounds and interval arithmetic \cite{arb}, one may compute all the items above to the desired bit accuracy. 
A generic library to carry out such computations, due originally to Dave Platt \cite{lfunccode}, has been developed.

\begin{example}
For an explicit example, we encourage the reader to peruse the source file \href{https://github.com/edgarcosta/lfunctions/blob/master/examples/cmf_23.1.b.a.cpp}{\texttt{examples/cmf\_23.1.b.a.cpp}} in \cite{lfunccode}, where the authors show how to use the library to compute all of the items above for the modular form \newformlink{23.1.b.a}, which matches its unique embedding.
By running this example, one can compute that
$$\epsilon =  (1 \pm 10^{-117}) +  (0 \pm  4.7\times 10^{-59})\,i,$$
(since $f$ is self dual, we must actually have $\epsilon = 1$), and
$$L(f,1/2) = 0.174036326987934183499504592018 \pm 8.2317 \times 10^{-59},$$
as well as approximate values for the imaginary part of the first ten zeros.
Using the notation above, we can represent an approximation to the imaginary part of the first zero
$$5.11568332881511759855335642038 \pm 3.9443 \times 10^{-31}$$
by the interval $[t_1-1,t_1 + 1]2^{-101}$, where
$$t_1 = 12969798084700060914517716069360.$$
The imaginary part of the following nine zeros are approximately
$7.15926$,
$8.88140$,
$10.2820$,
$11.4300$,
$12.9344$,
$14.6625$,
$16.4982$,
$17.1013$, and 
$18.0807$.

We carried out this computation with 100 bits of target accuracy for the \numprint{14398359} embedded newforms in our database with $k \leq 200$.
In our computation we observed that it was sufficient to work with 200 bits of precision and $C_k \leq 0.08 k \log(k) + 24$.
While we did not keep track of CPU time used along the way, by rerunning some of the computations, we extrapolate that we spent at least 11 CPU years on these computations.
\end{example}

\subsection{Imprimitive \texorpdfstring{$L$}{L}-function}

Associated to a newform $f$ with coefficient field $\Q(f)$ of degree $d$, we may also consider the $L$-function of degree $2d$ associated to its Galois orbit:
\begin{equation}
\begin{aligned}
    L(f,s) &\colonequals \prod_{\iota : \Q(f) \hookrightarrow \C} L(\iota(f), s)
    &= \prod_p L_p(f, p^{-s}).
\end{aligned}
\end{equation}
This gives rise to a $\Q$-primitive $L$-function with $L_p(f,T) \in 1 + T\Z[T]$, which satisfies the functional equation
\begin{equation}
\Lambda(f, s) \colonequals N^{sd/2} \Gamma_{\mathbb{C}}(s)^d L(f, s) = \varepsilon
\overline{\Lambda}(f, k-s),
\end{equation}
where now we have $\varepsilon = \pm 1$.
Using the invariants for each $L(\iota(f), s)$ mentioned above, one can easily deduce the respective invariants for $L(f, s)$.

For these $L$-functions we would also like to compute the local factors for small $p$. 
This is straightforward if one has access to an exact representation of $a_p$ in $\Q(f)$.
Otherwise, we relied on Newton identities to compute $L_p(f, T) \in \Z[T]$ from the roots of $L_p(\iota(f), T) \in \C[T]$, while working with interval arithmetic \cite{arb}: see 
\href{http://www.lmfdb.org/L/ModularForm/GL2/Q/holomorphic/500/2/e/c/}{$L$(\textsf{500.2.e.c})} for an example.
In some cases, for example when $[\Q(f):\Q]$ or the weight is large, we were only able to compute the initial coefficients for some local factors---this occurred for \href{http://www.lmfdb.org/L/ModularForm/GL2/Q/holomorphic/20/10/e/b/}{$L$(\textsf{20.10.e.b})}, for example.

\subsection{Verifying the analytic rank} \label{sec:analyticrank}

In this section, we discuss methods for rigorously verifying the analytic rank of a modular form $L$-function.  Throughout, let $N$ and $k$ be positive integers and let $f \in S_k(\Gamma_1(N);\C)$ be a newform of weight $k$ and level $N$ (with coefficient field embedded in the complex numbers).  

\begin{defn}
Suppose $k$ is \emph{even}.  We define the \defi{analytic rank} of $f$ to be the order of vanishing of $L(f,s)$ at $k/2$.
\label{defn:analyticrank}
\end{defn}

When $L^{(n)}(f,k/2) \neq 0$ one can certify such a statement using ball arithmetic by working with enough precision.
However, if $L^{(n)}(f,k/2) = 0$, this approach does not work, as there is no known bound $\varepsilon$  such that
$|L^{(n)}(f,k/2)| < \varepsilon$ implies $L^{(n)}(f,k/2) = 0$.
Nonetheless, if the order of vanishing is small, then there are other methods to computationally verify the order of vanishing. 
Using these methods we were able to provably verify the analytic rank of all modular forms for which the $L$-functions were computed.
The way the analytic computations were verified is detailed below.
The strategy used depends on the order of vanishing, and whether the modular form is self-dual or not. The analytic rank zero case is skipped because this can just be done by computing $L(f,k/2)$ to enough precision using interval arithmetic until $0$ is no longer in the computed interval. 

\begin{equation} \label{tab:sometab}\addtocounter{equation}{1} \notag
\begin{gathered}
{\renewcommand{\arraystretch}{1.2}
\begin{tabular}{l|ccccc}
	& \multicolumn{5}{c}{Analytic rank} \\
	  & 0 & 1 & 2 & 3 & $\geq 4$ \\ \hline
	Self-dual & \numprint{83338} & \numprint{85254} & \numprint{2565} & 1 & 0 \\ 
	Not self-dual  & \numprint{63804} & \numprint{1798} & 1 & 0 & 0 \\ \hline
	Total & \numprint{147142} & \numprint{87052} & \numprint{2566} &  1&  0
\end{tabular}
} \\[4pt]
\text{Table \ref{tab:sometab}: Number of even weight newforms in the database by analytic rank}\\[4pt]
\end{gathered}
\end{equation}

\subsubsection*{Self-dual and analytic rank 1}
\label{sec:self_dual_1}

We begin by considering self-dual newforms $f$ whose analytic rank numerically appears to be $1$.  All such forms in the range of our computation have trivial character.  In this case the functional equation takes the form 
\begin{equation}
\Lambda(f, s) = \varepsilon' i^k \Lambda(f, k-s),
\end{equation}
where $\varepsilon'=\pm1$ is the eigenvalue of the Atkin Lehner involution $W_N$.  For such forms in the database, we verified that $\varepsilon' \iota^k = -1$, forcing $\Lambda(f, k/2)=0$, and the non-vanishing of $N^{s/2}\Gamma_\C(k/2)$ then implies that $L(f,k/2) = 0$. The upper bound of 1 on the analytic rank was obtained using interval arithmetic.

\subsubsection*{Non-self-dual and analytic rank 1}
\label{sec:non-self_dual_1}

Following Stein \cite[\S 8.5]{Stein}, we define a pairing between modular forms and modular symbols
$$ S_k(\Gamma_1(N)) \oplus \overline S_k(\Gamma_1(N)) \times \mathbb \ModSym_k(\Gamma_1(N)) \to \C$$
by defining
\[
\langle(f,g),P\{a,b\}\rangle \colonequals \int_a^b f(z)P(z,1)dz +  \int_a^b g(z)P(\overline z,1)d\overline z.
\]

This pairing allows one to determine the vanishing of $L$-functions, because for every integer $1 \leq j \leq k-1$ we have
\[
L(f,j) = \frac {(-2\pi i)^j}{(j-1)!}\langle(f,0), X^{j-1}Y^{k-2-(j-1)}\{0,\infty\} \rangle.
\]
The pairing is Hecke-equivariant, meaning that $\langle (T_nf,T_ng),x \rangle = \langle (f,g),T_n x \rangle$ for all integers $n$.

Let $f \in \Sknew{k}(\Gamma_1(N))$ be a newform and $V_f \subseteq \Sknew{k}(\Gamma_1(N))$ the subspace generated by its Galois conjugates.
Then by Atkin--Lehner--Li theory,  $V_f$ is a simple module over the Hecke algebra $\mathcal{T}$, and there exists a Hecke operator $t_f \in \mathcal{T}$ such that $t_f \colon  M_k(\Gamma_1(N)) \to M_k(\Gamma_1(N))$ is a projection onto $V_f$.
Because $t_f$ is a projection we have $\langle (f,0), x \rangle = \langle (t_ff,t_f0),x \rangle = \langle (f,0), t_fx \rangle$ for all $x \in \ModSym_k(\Gamma_1(N))$, and hence in particular this means that if 
\[ t_f (X^{j-1}Y^{k-2-(j-1)}\{0,\infty\}) = 0 \] 
then $L(f,j)=0$. 

A map $t_f'$ with the same kernel as $t_f$ can be obtained from $t_f(\ModSym_k(\Gamma_1(N),\Q))$ in \magma{} using the command \texttt{PeriodMapping}.
Furthermore, this \magma{} command only uses exact arithmetic over $\Q$.
For all non-self-dual modular forms whose analytic rank numerically seemed to be $1$, it was verified that indeed $t_f'(X^{k/2-1}Y^{k/2-1}\{0,\infty\})=0$, implying that $L(f,k/2)=0$.
The upper bound of 1 on the analytic rank was again obtained using interval arithmetic.

\subsubsection*{Self-dual and analytic rank 2}\label{sec:self_dual_2}
As in the preceding subsection, all newforms in the database whose analytic rank numerically seemed to be $2$ have trivial character.
This time it was verified that $\varepsilon' \iota^k = 1$ for all these modular forms.
In particular, the functional equation then forces all odd derivatives of $\Lambda(f,s)$ to vanish at $k/2$.
This forces the order of vanishing of $\Lambda(f,s)$ at $k/2$ to be even, and hence the analytic rank of $L(f,s)$ to be even as well.
The techniques of the preceding paragraph were used to prove that for all these modular forms one has that $L(f,k/2)=0$, which together with the parity argument gives a lower bound of 2 on the analytic rank.
The upper bound of $2$ was again obtained using interval arithmetic.

\subsubsection*{Non-self-dual analytic rank 2}\label{sec:non-self_dual_2}
There is exactly one Galois orbit of non-self-dual newforms in the database whose analytic rank numerically seems to be $2$.
Let $f$ denote the newform of weight $2$ and level 1154 with LMFDB label \newformlink{1154.2.e.a} with coefficient field $\Q(\zeta_3).$
This pair corresponds to an isogeny class of abelian surfaces, and our first goal is to find a representative of this isogeny class.
By searching for hyperelliptic curves over $\F_p$ that match the local factors of $L(f,s)$ for small $p$, and then by lifting their Weierstrass equations to $\Z$ we found the following genus 2 curve:
\begin{equation}
    C \colon y^2 = x^6 - 12x^5 + 34x^4 - 18x^3 - 11x^2 + 6x + 1\text.
\end{equation}
Letting $J$ denote its Jacobian, we find it is of conductor $1154^2$ as desired.
Our goal is first to show that $J$ really is in the isogeny class of abelian surfaces corresponding to the newform \textsf{1154.2.e.a}.
Using~\cite{rigendos} we were able to compute the endomorphism ring of $J$, and verify that $J$ is of $\GL_2$-type and hence is modular \cite{Ribet:modular, KW}.
Thus, $J$ is a good candidate to be a representative of the isogeny class of abelian surfaces corresponding to the newform \textsf{1154.2.e.a}.
Alternatively, one can also verify that $J$ is of $\GL_2$-type by noting that $C$ has an automorphism of order $3$ given by $x\mapsto 1-1/x$, $y \mapsto -y/x^3$ and thus showing that its Jacobian is of $\GL_2$-type.
Additionally, the Euler factor at $5$ of its $L$-function is 
$$1 + 6T + 17T^2 + 30T^3 + 25T^4$$
which is irreducible. Hence its Jacobian is simple, showing that its Jacobian corresponds to a pair of Galois conjugate newforms of level $1154$.
There is one other pair of Galois conjugate newforms whose coefficient field is $\Q(\zeta_3)$, namely that with LMFDB label \textsf{1154.2.c.a}.
So it remains to show that $J$ does not come from the newform with label \textsf{1154.2.c.a}.
However the Euler factor of the $L$-function at $5$ for that newform is $1 - 3 T + 4 T^{2} - 15 T^{3} + 25 T^{4}$ which does not match that of $J$.
This means that Jacobian of $C$ really is in the isogeny class of abelian surfaces corresponding to the newform \textsf{1154.2.e.a}.

Using the \magma{} function \texttt{RankBounds} one readily computes that $J$ has Mordell-Weil rank $4$.
In particular, it has rank 2 as a module over $\Z[\zeta_3]$.
The generalization by Kato of the work of Kolyvagin and Logachev on the Birch--Swinnerton-Dyer conjecture in the analytic rank 0 and 1 cases to all isogeny factors of $J_1(N)$ (see Kato \cite[Corollary 14.3]{Kato}) shows that the order of vanishing of $L(f,s)$ at 1 cannot be 0 or 1 since this would give $J$ rank 0 or 1 as a $\Z[\zeta_3]$-module.
So the order of vanishing is at least 2. An upper bound was again obtained using interval arithmetic.

\subsubsection*{Self-dual analytic rank 3}
The approach here is similar to that in \secref{sec:non-self_dual_2} and the result was already briefly mentioned in \cite[Section 3.4]{Cremona:database} where the analytic rank is determined for all elliptic curves of conductor $N<\numprint{130000}$. There is only one newform that numerically seems to be of analytic rank 3 in the database, namely \newformlink{5077.2.a.a} of weight $2$, level $5077$ and trivial character.  This modular form corresponds to the elliptic curve $y^2 + y = x^{3} - 7 x + 6$ which has rank $3$ and is the only one in its isogeny class.  The verification that its $L$-function has analytic rank $3$ is a famous calculation of Buhler--Gross--Zagier \cite{BGZ}, used by Gross--Zagier \cite{GZ} in their solution to the Gauss class number $1$ problem.  We confirm it quickly as follows: by known cases of the Birch--Swinnerton-Dyer conjecture, the analytic rank cannot be $0$ or $1$; by parity of the root number, the analytic rank cannot be $2$, so it must be at least $3$; and an upper bound on the analytic rank of 3 is obtained by interval arithmetic.  

\subsection{Chowla's conjecture} \label{sec:chowla}

The definition of analytic rank (Definition~\ref{defn:analyticrank}) as an order of vanishing also makes sense for $k$ odd, and by analogy one might also find it 
natural to study the central values of $L(f, s)$ at $k/2$ and their derivatives.
However, for $k$ odd the central value $s = k/2$ is not a \emph{special} value in the sense of Deligne~\cite{Deligne79} and thus there is no abelian group whose rank (as a module over an appropriate coefficient ring) is conjecturally related to its leading Taylor coefficient.  It would therefore be a stretch to call the order of vanishing at the central point an \emph{analytic rank}.
Moreover, one does not expect $L(f, k/2)$ to ever vanish, and this is a generalization of Chowla's conjecture for Dirichlet $L$-functions~\cite{Chowla}, as follows.

Let $\chi$ be a non-trivial Dirichlet character, then the functional equation associated to $L(\chi, s) \colonequals \sum \chi(n) n^{-s}$, similar to equation~\ref{eqn:functionalequation}, relates
$L(\chi, s)$ to $L(\overline{\chi}, 1- s)$.
The value of $L(\chi, 1)$ is quite well understood.
For example, the fact that $L(1, \chi) \neq 0$ gives us Dirichlet's theorem on arithmetic progressions, and for primitive real characters the value $L(\chi, 1)$ gives us Dirichlet's class number formula.
As mentioned above, inspired by Definition~\ref{defn:analyticrank}, one might also find it natural to study the order of vanishing of $L(\chi, s)$ at $s = 1/2$ and its derivatives.
However, it is believed that $L(\chi, 1/2) \neq 0$; this was first conjectured by Chowla~\cite{Chowla} for primitive real characters and later generalized to other characters.
One of the reasons behind such a belief is that for primitive real characters the root number of such $L$-functions is always $1$~\cite{fiorilli}, and thus there is no simple reason for $L(\chi, 1/2)$ to vanish.
Although Chowla's conjecture remains open, it has been numerically verified for all real characters $\chi$ of modulus less than $10^{10}$
\cite{Omar}, and substantial progress towards showing the non-vanishing of $L(\chi,  1/2)$ has also been made, see~\cite{fiorilli} for a short overview.

A generalization of Chowla's conjecture is that $L(f, k/2) \neq 0$ for $k$ odd.
As in the case of Dirichlet $L$-functions for primitive real characters, we also have that the root number of $L(f, k/2)$ can never be $-1$ when $f$ is self-dual.
This is in stark contrast to the case of self dual even weight modular forms, where the root numbers are split approximately 50-50 between $-1$ and $1$.
We verified this generalization of Chowla's conjecture, as we computed $L(f,k/2)$ for every newform in our database with $k \leq 200$, and found that this was nonzero for all the odd weight newforms.

\section{An overview of the computation}
\label{sec:overview}

In this section, we provide an overview of the computations we performed, the results of which are now available in the LMFDB \cite{LMFDB}. These were accomplished using a combination of \magma{}, \pari{}, and \sage{} scripts, as well as hand written C code for some of the more computationally-intensive tasks.  In aggregate, these computations consumed more than 100 years of CPU time.

\subsection{Data extent} \label{sec:dataextent}

Our database consists of four overlapping sets of newforms described in Table~\ref{table:extent}.  These datasets were chosen both for reasons of mathematical interest, and to ensure that the database included all modular forms contained in existing datasets such as the Stein tables of modular forms \cite{Stein:Tables}, the Buzzard-Lauder tables of \weightone{} newforms~\cite{BuzzardLauder}, and the previous database of modular forms contained in the LMFDB\@.
More detailed statistics on the newforms in the database can be found at the \href{http://www.lmfdb.org/ModularForm/GL2/Q/holomorphic/stats}{statistics page}.

\begin{equation} \label{table:extent}\addtocounter{equation}{1} \notag
\begin{gathered}
{\renewcommand{\arraystretch}{1.2}
\begin{tabular}{l|rrr}
         Constraints on $\Sknew{k}(N,\chi)$ & Newspaces & Newforms & Embeddings\\\hline
         (1) $Nk^2\le \numprint{4000}$ & $\numprint{30738}$ & $\numprint{67180}$ & $\numprint{9966498}$\\  
         (2) $Nk^2 \le \numprint{40000}$, $|\chi|=1$ & $\numprint{16277}$ & $\numprint{170611}$ & $\numprint{3092301}$\\
         (3) $Nk^2 \le \numprint{40000}$, $k>1$, $\dim \Sknew{k}(N,\chi)\le 100$ & $\numprint{30345}$ & $\numprint{131540}$ & $\numprint{1648617}$\\
         (4) $Nk^2\le \numprint{40000}, N\le 24$ or\\  \quad\ \  $Nk^2\le \numprint{100000}, N\le 10$ or\\ \quad\ \  $N\le 100, k\le 12$ & $\numprint{7627}$ & $\numprint{12237}$ & $\numprint{676574}$\\
         \hline
         Union of sets above & \numprint{62142} & \numprint{281219} & \numprint{14398359}
    \end{tabular}
} \\[4pt]
\text{Table \ref{table:extent}: Extent of the newform database (only nonzero newspaces are included)}\\[4pt]
\end{gathered}
\end{equation}

For the first dataset (1), we used three independent sources of newform data:
\begin{itemize}
    \item Complex eigenvalue data for each embedded newform of weight $k>1$ computed by the \texttt{mflib} software package \cite{bober}, which uses \arb{} \cite{arb} to rigorously implement the trace formula (as described in \cite{SV91}, for example) to obtain approximate complex values to a precision of 200 decimal digits.
    \item Exact algebraic eigenvalue data for each newform of weight $k>1$ and dimension $d\le 20$ computed using \magma{}'s \cite{Magma} modular symbols package (originally written by William Stein);
    \item Exact algebraic eigenvalue data for each newform of weight $k>1$ and dimension $d\le 20$ were computed using the modular forms implementation in \pari{} \cite{Pari} described in \cite{BelabasCohen}, which was also applied to all newforms of weight $k=1$.
\end{itemize}
For $k>1$ and $Nk^2\le 4000$ the decomposition of every newspace $\Sknew{k}(N,\chi)$ was computed in all three cases and compared for consistency.  Exact algebraic data was computed only for newforms of dimension $d\le 20$, except for $k=1$ where exact algebraic data was computed in every case.
For newforms of weight $k>1$ and dimension $d\le 20$, the algebraic data independently computed by \magma{} and \pari{} were checked for consistency (this was not a completely trivial task, as it generally required determining an appropriate automorphism of the coefficient field in order to compare sequences of Fourier coefficients).  We also compared the trace forms using all three methods and compared the results for consistency, and for newforms of weight $k=1$ and level $N\le 1500$ we compared the \pari{} computations with the Buzzard-Lauder database \cite{BuzzardLauder}.

Datasets (2) and (3) were computed entirely in \magma{}, as was dataset (4), except for~12 spaces of high dimension where complex analytic methods were used.  For the portions of these datasets that overlap with the Stein database of modular forms \cite{Stein:Tables}, we compared the results for consistency.

For newforms $f=\sum a_nq^n$ of level $N\le 1000$ we computed $1000$ coefficients $a_n$, while for newforms of level $1001\le N \le 4000$ we computed $2000$ coefficients, and for newforms of level $4001\le N\le \numprint{10000}$ we computed $3000$ coefficients.  This substantially exceeds the Sturm bound in every case, and also exceeds the bound $30\sqrt{N}$ required for the $L$-function calculations described in \secref{sec:Lfunctions}.  For every newform in the database we computed complex coefficients to a precision of at least 200 bits.  In cases where we compute algebraic coefficient data we computed an optimized representation using an LLL-basis as described in \secref{sec:LLLbasis}, along with a set of generators for the coefficient ring.

For each newform we determined any non-trivial self-twists admitted by the newform (CM, RM, or both), and for newforms with algebraic eigenvalue data available, we computed all inner twists as described in \secref{sec:twists}.  We also computed the analytic rank of every newform, as described in \secref{sec:analyticrank}, and for \weightone{} newforms we computed the image of the associated projective Artin representation and a defining polynomial for its kernel, as described in \secref{sec:weight1}.
These computations have now all been rigorously verified.

In addition to the newform database, we computed dimension tables for all newspaces in the range $Nk^2\le \numprint{40000}$ with $k>1$, and we computed trace forms for all newspaces of level $N\le 4000$ in this range using the \texttt{mftraceform} function in \pari{}.

\subsection{Statistics} \label{sec:statistics}

In addition to the ability to browse and to search for examples with specific
properties, the modular forms database allows for an investigation of
arithmetic statistics.  The LMFDB \cite{LMFDB} includes precomputed tables
displaying how various quantities vary across the database, some of which we
have duplicated here in Tables \ref{tab:anrank}, \ref{tab:projim},
\ref{tab:intwist}, and \ref{tab:selftwist}.

In addition to these static tables, we have added \emph{dynamic statistics}
\[
\text{\url{http://www.lmfdb.org/ModularForm/GL2/Q/holomorphic/dynamic\_stats}}
\]
which allow users to customize which variables to view and
any constraints to impose.  For example, a researcher might create a table
displaying how the weight and level vary among forms with complex multiplication.
We hope that this new feature will enable examination of large-scale patterns,
both in the modular form data and elsewhere in the LMFDB\@.

\begin{remark}
The statistics and examples presented in this article reflect the dataset defined in \secref{sec:dataextent}, which represents the state of the LMFDB as of January 2020.  As new data is added to the LMFDB these statistics may no longer match those displayed in the LMFDB, and the number of newforms returned by some of the example queries listed below may increase.
\end{remark}

\begin{equation} \label{tab:anrank}\addtocounter{equation}{1} \notag
\begin{gathered}
{\renewcommand{\arraystretch}{1.2}
\begin{tabular}{l|rrrr}
         analytic rank & 0 & 1 & 2 & 3 \\
         \hline
         count & \href{http://www.lmfdb.org/ModularForm/GL2/Q/holomorphic/?analytic_rank=0}{\numprint{191520}} & \href{http://www.lmfdb.org/ModularForm/GL2/Q/holomorphic/?analytic_rank=1}{\numprint{87052}} & \href{http://www.lmfdb.org/ModularForm/GL2/Q/holomorphic/?analytic_rank=2}{\numprint{2566}} & \href{http://www.lmfdb.org/ModularForm/GL2/Q/holomorphic/?analytic_rank=3}{1} \\
         proportion & 68.12\% & 30.96\% & 0.91\% & 0.00\% \\
         example & \newformlink{23.1.b.a} & \newformlink{37.2.a.a} & \newformlink{389.2.a.a} & \newformlink{5077.2.a.a}
    \end{tabular}
} \\[4pt]
\text{Table \ref{tab:anrank}: Distribution of analytic ranks}\\[4pt]
\end{gathered}
\end{equation}

\begin{equation} \label{tab:projim}\addtocounter{equation}{1} \notag
\begin{gathered}
{\renewcommand{\arraystretch}{1.2}
    \begin{tabular}{l|rrrrr}
    projective image & $A_4$ & $S_4$ & $A_5$ & $D_2$ & $D_n$ \\
    \hline
    count & \href{http://www.lmfdb.org/ModularForm/GL2/Q/holomorphic/?projective_image=A4}{\numprint{458}} & \href{http://www.lmfdb.org/ModularForm/GL2/Q/holomorphic/?projective_image=S4}{\numprint{1033}} & \href{http://www.lmfdb.org/ModularForm/GL2/Q/holomorphic/?projective_image=A5}{\numprint{202}} & \href{http://www.lmfdb.org/ModularForm/GL2/Q/holomorphic/?projective_image=D2}{\numprint{1311}} & \href{http://www.lmfdb.org/ModularForm/GL2/Q/holomorphic/?projective_image_type=Dn}{\numprint{17613}} \\
    proportion & 2.37\% & 5.35\% & 1.05\% & 6.79\% & 91.23\% \\
    example & \newformlink{124.1.i.a} & \newformlink{148.1.f.a} & \newformlink{1763.1.p.b} & \newformlink{3600.1.e.a} & \newformlink{3997.1.cz.a}
    \end{tabular}
} \\[4pt]
\text{Table \ref{tab:projim}: Distribution of projective images}\\[4pt]
\end{gathered}
\end{equation}

\begin{equation} \label{tab:intwist}\addtocounter{equation}{1} \notag
\begin{gathered}
{\renewcommand{\arraystretch}{1.2}
\begin{tabular}{l|rrrrrrrrrrrrrrr}
    Inner twists & Unknown & 1 & 2 & 4 & 6 & 8 & 10 & 12 \\
    \hline
    count & \href{http://www.lmfdb.org/ModularForm/GL2/Q/holomorphic/?inner_twist_count=-1}{\numprint{73993}} & \href{http://www.lmfdb.org/ModularForm/GL2/Q/holomorphic/?inner_twist_count=1}{\numprint{129197}} & \href{http://www.lmfdb.org/ModularForm/GL2/Q/holomorphic/?inner_twist_count=2}{\numprint{47492}} & \href{http://www.lmfdb.org/ModularForm/GL2/Q/holomorphic/?inner_twist_count=4}{\numprint{25803}} & \href{http://www.lmfdb.org/ModularForm/GL2/Q/holomorphic/?inner_twist_count=6}{24} & \href{http://www.lmfdb.org/ModularForm/GL2/Q/holomorphic/?inner_twist_count=8}{\numprint{4295}} & \href{http://www.lmfdb.org/ModularForm/GL2/Q/holomorphic/?inner_twist_count=10}{6} & \href{http://www.lmfdb.org/ModularForm/GL2/Q/holomorphic/?inner_twist_count=12}{51} \\
    proportion & 26.31\% & 45.94\% & 16.89\% & 9.18\% & 0.01\% & 1.53\% & 0.00\% & 0.02\% \\[0.15in]
    
    Inner twists & 16 & 20 & 24 & 32 & 40 & 44 & 56 \\
    \hline
    count & \href{http://www.lmfdb.org/ModularForm/GL2/Q/holomorphic/?inner_twist_count=16}{311} & \href{http://www.lmfdb.org/ModularForm/GL2/Q/holomorphic/?inner_twist_count=20}{3} & \href{http://www.lmfdb.org/ModularForm/GL2/Q/holomorphic/?inner_twist_count=24}{14} & \href{http://www.lmfdb.org/ModularForm/GL2/Q/holomorphic/?inner_twist_count=32}{20} & \href{http://www.lmfdb.org/ModularForm/GL2/Q/holomorphic/?inner_twist_count=40}{7} & \href{http://www.lmfdb.org/ModularForm/GL2/Q/holomorphic/?inner_twist_count=44}{1} & \href{http://www.lmfdb.org/ModularForm/GL2/Q/holomorphic/?inner_twist_count=56}{2} \\
    proportion & 0.11\% & 0.00\% & 0.00\% & 0.01\% & 0.00\% & 0.00\% & 0.00\%
    \end{tabular}
} \\[4pt]
\text{Table \ref{tab:intwist}: Distribution of inner twists}\\[4pt]
\end{gathered}
\end{equation}

\begin{equation} \label{tab:selftwist}\addtocounter{equation}{1} \notag
\begin{gathered}
{\renewcommand{\arraystretch}{1.2}
\begin{tabular}{l|rrrrrr}
    & \multicolumn{6}{c}{weight} \\
    & 1 & 2 & 3 & 4 & 5-316 & total \\
    \hline
    \multirow{2}{*}{neither} & \href{http://www.lmfdb.org/ModularForm/GL2/Q/holomorphic/?cm=no&rm=no&weight=1}{\numprint{1693}} & \href{http://www.lmfdb.org/ModularForm/GL2/Q/holomorphic/?cm=no&rm=no&weight=2}{\numprint{174853}} & \href{http://www.lmfdb.org/ModularForm/GL2/Q/holomorphic/?cm=no&rm=no&weight=3}{\numprint{11117}} & \href{http://www.lmfdb.org/ModularForm/GL2/Q/holomorphic/?cm=no&rm=no&weight=4}{\numprint{27877}} & \href{http://www.lmfdb.org/ModularForm/GL2/Q/holomorphic/?cm=no&rm=no&weight=5-}{\numprint{40278}} & \href{http://www.lmfdb.org/ModularForm/GL2/Q/holomorphic/?cm=no&rm=no}{\numprint{255818}} \\
    & 8.77\% & 98.27\% & 87.85\% & 98.02\% & 93.91\% & 90.97\% \\
    \multirow{2}{*}{CM only} & \href{http://www.lmfdb.org/ModularForm/GL2/Q/holomorphic/?cm=yes&rm=no&weight=1}{\numprint{15841}} & \href{http://www.lmfdb.org/ModularForm/GL2/Q/holomorphic/?cm=yes&rm=no&weight=2}{\numprint{3074}} & \href{http://www.lmfdb.org/ModularForm/GL2/Q/holomorphic/?cm=yes&rm=no&weight=3}{\numprint{1538}} & \href{http://www.lmfdb.org/ModularForm/GL2/Q/holomorphic/?cm=yes&rm=no&weight=4}{\numprint{563}} & \href{http://www.lmfdb.org/ModularForm/GL2/Q/holomorphic/?cm=yes&rm=no&weight=5-}{\numprint{2613}} & \href{http://www.lmfdb.org/ModularForm/GL2/Q/holomorphic/?cm=yes&rm=no}{\numprint{23629}} \\
    & 82.05\% & 1.73\% & 12.15\% & 1.98\% & 6.09\% & 8.40\% \\
    \multirow{2}{*}{RM only} & \href{http://www.lmfdb.org/ModularForm/GL2/Q/holomorphic/?cm=no&rm=yes&weight=1}{461} & & & & & \href{http://www.lmfdb.org/ModularForm/GL2/Q/holomorphic/?cm=no&rm=yes}{461} \\
    & 2.39\% & & & & & 0.16\% \\
    \multirow{2}{*}{both} & \href{http://www.lmfdb.org/ModularForm/GL2/Q/holomorphic/?cm=yes&rm=yes&weight=1}{\numprint{1311}} & & & & & \href{http://www.lmfdb.org/ModularForm/GL2/Q/holomorphic/?cm=yes&rm=yes}{\numprint{1311}} \\
    & 6.79\% & & & & & 0.47\% \\
    \end{tabular}
} \\[4pt]
\text{Table \ref{tab:selftwist}: Distribution of self twist types by weight}\\[4pt]
\end{gathered}
\end{equation}

\subsection{Data reliability} \label{sec:reliability}
All of our modular form data was computed or verified using rigorous algorithms that do not depend on any unproved assumptions or conjectures.

\begin{itemize}
\item Self-twists were either verified via Theorem \ref{thm:innertwist} and Proposition \ref{prop:intwistribet} using exact algebraic Fourier coefficients $a_n$ or ruled out using complex approximations of sufficient precision to rigorously distinguish zero and nonzero values of $a_n$ and checking for self-twists by all primitive quadratic characters $\psi$ of conductor dividing the level (a newform that admits a self-twist by $\psi$ must have $a_n=0$ whenever $\psi(a_n)\ne 1$).

\item We computed and verified inner twists for all newforms in our dataset that are either of \weightone{} or have dimension at most $20$ by computing sufficiently many algebraic Fourier coefficients and applying Theorem \ref{thm:innertwist} and Proposition \ref{prop:intwistribet}.

\item Analytic ranks were computed using complex approximations as described in \secref{sec:Lfunctions} and then rigorously verified using the symbolic methods described in \secref{sec:analyticrank}.

\item For \weightone{} newforms the  classification of projective images as $D_n$, $A_4$, $S_4$, $A_5$ was rigorously verified by explicitly computing the number field fixed by the kernel of the associated projective Galois representation.  As described in \secref{sec:weight1}, this was accomplished using a combination of the ray class field functionality provided by \pari{} and \magma{}, the rigorous tabulation of all $A_4$, $S_4$, and $A_5$ number fields with compatible ramification, and the explicit computation of quotients of ring class fields of orders in imaginary quadratic fields via the theory of complex multiplication.
\end{itemize}

In addition to using mathematically rigorous algorithms, we performed a variety of consistency checks intended to catch any errors in the software packages used to compute modular forms data, or any errors that might have been introduced during post-processing.  The following checks have been performed:

\begin{itemize}
\item All newforms of weight $k > 1$ and level $N$ satisfying $Nk^2 \le 2000$ have been independently computed using \magma{} and \pari{}.  By comparing the results of these computations we have verified that the decompositions of each newspace $\Sknew{k}(N,\chi)$ into  Galois orbits agree (with matching coefficient fields), that the first 1000 coefficients of the trace forms for each Galois orbit agree, and for newforms of dimension $d\le 20$, that there is an automorphism of the coefficient field that relates the sequences of algebraic eigenvalues $(a_1,\ldots,a_{1000})$ computed by \pari{} and \magma{}.

\item For all newforms of weight $k>1$ and level $N$ satisfying $Nk^2 \le 4000$ we have verified that the trace forms computed by \magma{} (using modular symbols) agree with the trace forms obtained from complex analytic data computed using the explicit trace formula.  This also verifies the dimensions of the coefficient fields.

\item For newforms of weight $k=1$ and level $N\le 1000$ we have matched the data computed using \pari{} with the tables computed by Buzzard and Lauder \cite{BuzzardLauder}.

\item For all dihedral newforms of weight $k=1$ and level $N\le 4000$ we have matched trace forms with data computed using the explicit trace formula in \pari{} with data independently computed using the ray class field functionality implemented in \pari{} and \magma{}.
\end{itemize}

As a consistency check for our $L$-function computations, after computing a provisional list of all non-trivial zeros on the critical line up to a chosen height bound $b$ we confirmed that no zeros are missing, in other words, that the Riemann Hypothesis holds for each $L$-function up to height $b$.
We use the method described in \cite{buthe2015} based on the Weil--Barner explicit formula.  If an $L$-function also arises from another object in the LMFDB for which we already had computed its $L$-function we verified that these computations match.

\subsection{Interesting, extreme behavior and examples from the literature} \label{sec:interestingandextreme}

When putting modular forms in  a database it is easy to view them as an aggregate, but of course each modular form is distinct and many have unique interesting properties.

We take this opportunity to recall the rich history and special properties of several forms in this database. We also provide links between these forms and the literature and note several forms that have naturally arisen in previous work.  We focus on weight $k \geq 2$ in this section; see \secref{sec:wt1cool} for interesting behavior in weight $k=1$. 
\begin{itemize} 
\item The most well known, and the prototypical, example of a modular form is the Ramanujan $\Delta$ function, of weight 12 and level 1; its label is \newformlink{1.12.a.a}.
This is the lowest weight in which a cusp form appears for the full modular group, so many properties of more general newforms were first noticed for $\Delta$. Similarly, $\Delta$ has served as a testing ground for techniques and results before they were known more generally.
For instance, the Ramanujan--Petersson conjecture was first made by Ramanujan for $\Delta$ but later extended to all newforms.
Additionally, computation of the $q$-expansion coefficients of $\Delta$, traditionally denoted by $\tau (n)$ and known as Ramanujan's $\tau$ function, is the    subject of the monograph~\cite{EdixhovenCouveignes}.

\item By the modularity theorem, newforms of weight 2 with rational coefficients correspond to isogeny classes of elliptic curves over $\Q$.  The smallest level in which a weight 2 form appears is $11$, corresponding to the smallest conductor of an elliptic curve over $\Q$.  Here we necessarily have trivial character and the label is \newformlink{11.2.a.a}; this form has $q$-expansion 
\[
q \prod_{k\ge 1} (1-q^k)^2(1-q^{11k})^2.
\]

\item The weight~2 newforms with CM by fields with the largest absolute discriminants in the database are \newformlink{2169.2.d.a} with CM by $\Q(\sqrt{-723})$, \newformlink{8388.2.e.c} and \newformlink{2097.2.d.a} with CM by $\Q(\sqrt{-699})$, \newformlink{2061.2.c.c} with CM by $\Q(\sqrt{-687})$, and \newformlink{7524.2.l.b} with CM by $\Q(\sqrt{-627})$---the last of these has 8 inner twists.

\item The weight~2 newform \newformlink{867.2.i.a} with CM by $\Q(\sqrt{-51})$ has 32 inner twists, and the weight~1 newform \newformlink{3481.1.d.a} with CM by $\Q(\sqrt{-59})$ has 56 inner twists.

\item The weight~3 newform \newformlink{7.3.b.a} has CM by $\Q(\sqrt{-7})$, making it the first (by analytic conductor) newform of weight $\ge 3$ with CM. 

\item Watkins \cite[\S9.1.3]{Watkins} discusses several examples  of modular forms of analytic rank 2.  The query \url{http://www.lmfdb.org/ModularForm/GL2/Q/holomorphic/?weight=4-&analytic_rank=2-} returns 130 forms of weight at least 4 and analytic rank at least 2, many of which are mentioned by Watkins, including 2 of weight 8.
\item Watkins also discusses modular forms of weight 2 with which are non-self-dual yet have positive analytic rank, particularly examples with quadratic character, such as \newformlink{122.2.b.a}. The query \url{http://www.lmfdb.org/ModularForm/GL2/Q/holomorphic/?weight=2&char_order=2&is_self_dual=no&analytic_rank=1-} produces 567 such examples.
In larger weight we have \newformlink{8.14.b.a} which is non-self-dual and has analytic rank 1, as does \newformlink{162.12.c.i}.

\item The index of the coefficient ring in the ring of integers of the coefficient field can get quite large, as in the case of the newform \newformlink{8.21.d.b} where the index is at least $ 2^{153}\cdot 3^{15}\cdot 5^{4}\cdot 7^{2} $.  In weight~2, the largest index we computed was $2^{26}\cdot 3^4$ for \newformlink{2016.2.k.b} and \newformlink{4032.2.k.h}.

\item Many newforms in our database have very large Hecke orbits.  For example, the newform \newformlink{983.2.c.a} has relative dimension 81 over its character field $\Q(\zeta_{491})$ and $\Q$-dimension \numprint{39690}.
\end{itemize}

\subsection{Pictures}

For every newform $f$, every nonempty newspace $S_k^{\mathrm{new}}(N,\chi)$, and $S_{k}^{\mathrm{new}}(\Gamma_1(N))$ for which we have all the newforms, we have created a portrait based on their trace forms, a total of \numprint{641562} portraits.
The picture is generated by plotting the absolute value of the trace form in the Poincar\'e disk, obtained as the image of $(1 - i z)/(z - i)$ in $\calH$, where the color hue represents the absolute value modulo $1$ (with blue being zero, and increasing through purple, red, orange, yellow, \dots).
For example, as the trace form is always zero at $\infty$, the top center is always blue, see Figure \ref{fig:23.1.b.a}.  
\begin{equation} \label{fig:23.1.b.a}\addtocounter{equation}{1} \notag
\begin{gathered}
\includegraphics[width=0.44\textwidth]{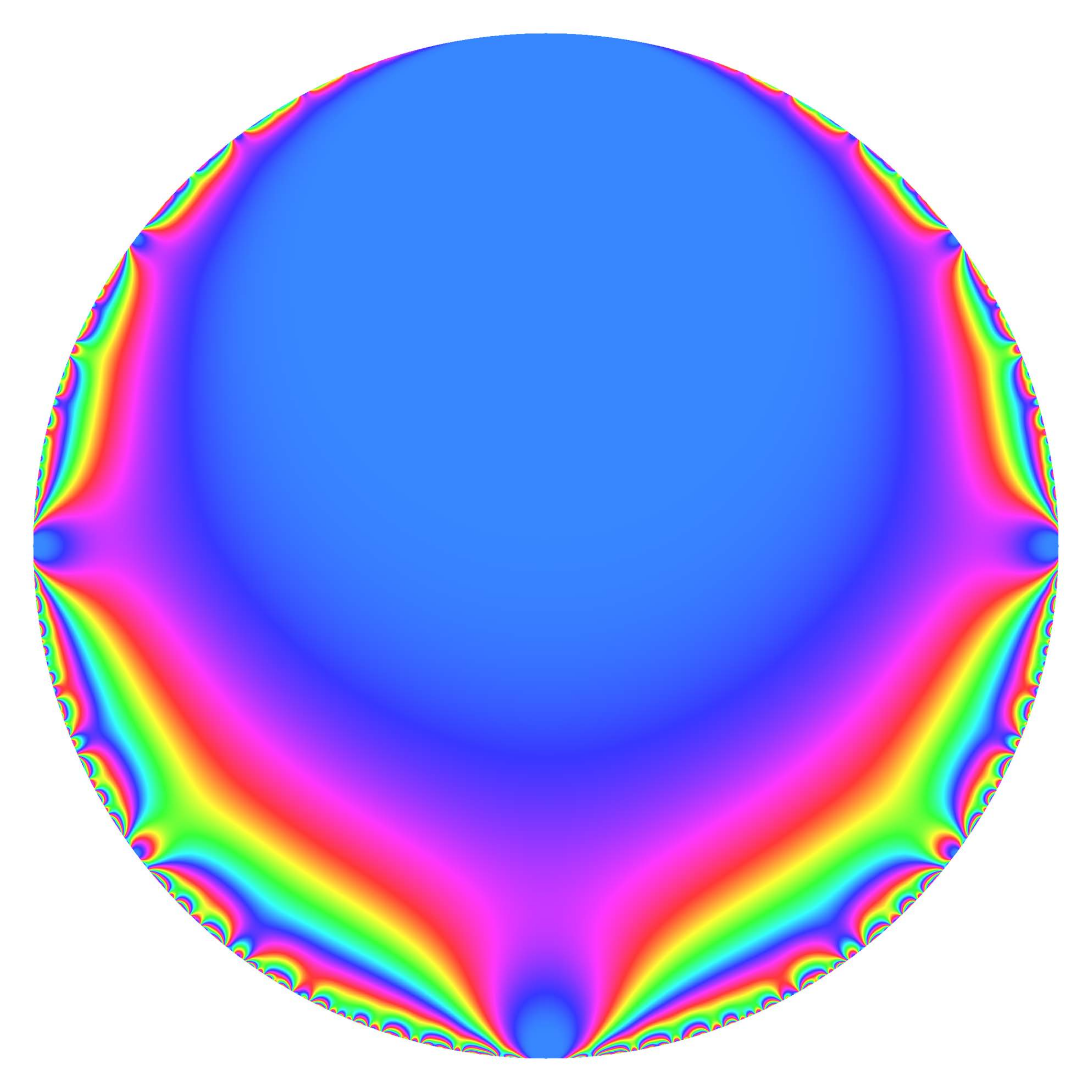}
 \\[4pt]
\text{Figure \ref{fig:23.1.b.a}: Portrait of \newformlink{23.1.b.a}}
\end{gathered}
\end{equation}

We deviated from the normal approach, used by \texttt{complex\_plot} in \sage{}, of representing magnitude by brightness (with zero being black and infinity being white) and the argument by hue, as this often leads to an overexposed or underexposed picture, see Figure~\ref{fig:complex_plot}.
\begin{equation} \label{fig:complex_plot}\addtocounter{equation}{1} \notag
\begin{gathered}
    \includegraphics[width=0.24\textwidth]{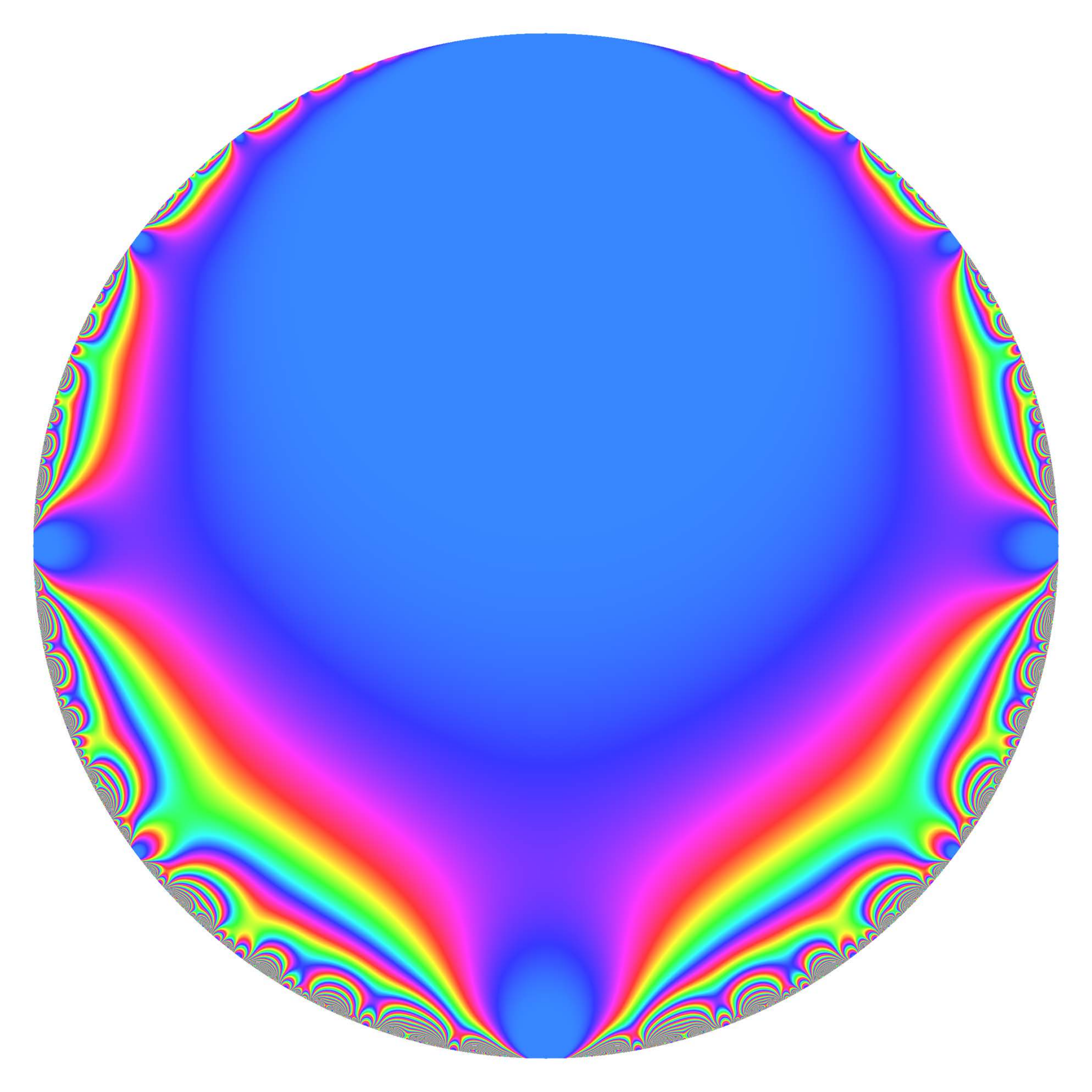}
    \includegraphics[width=0.24\textwidth]{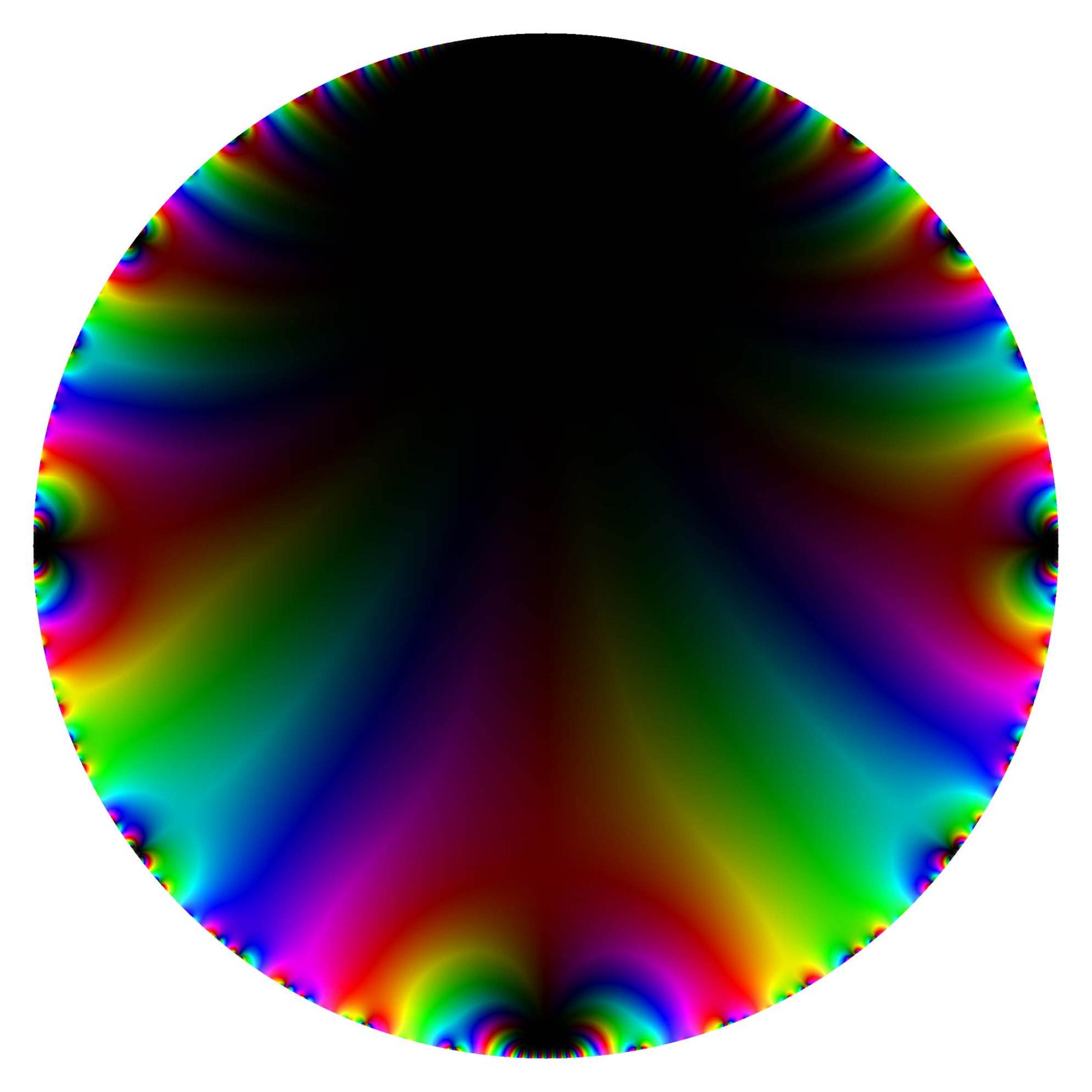}
    \includegraphics[width=0.24\textwidth]{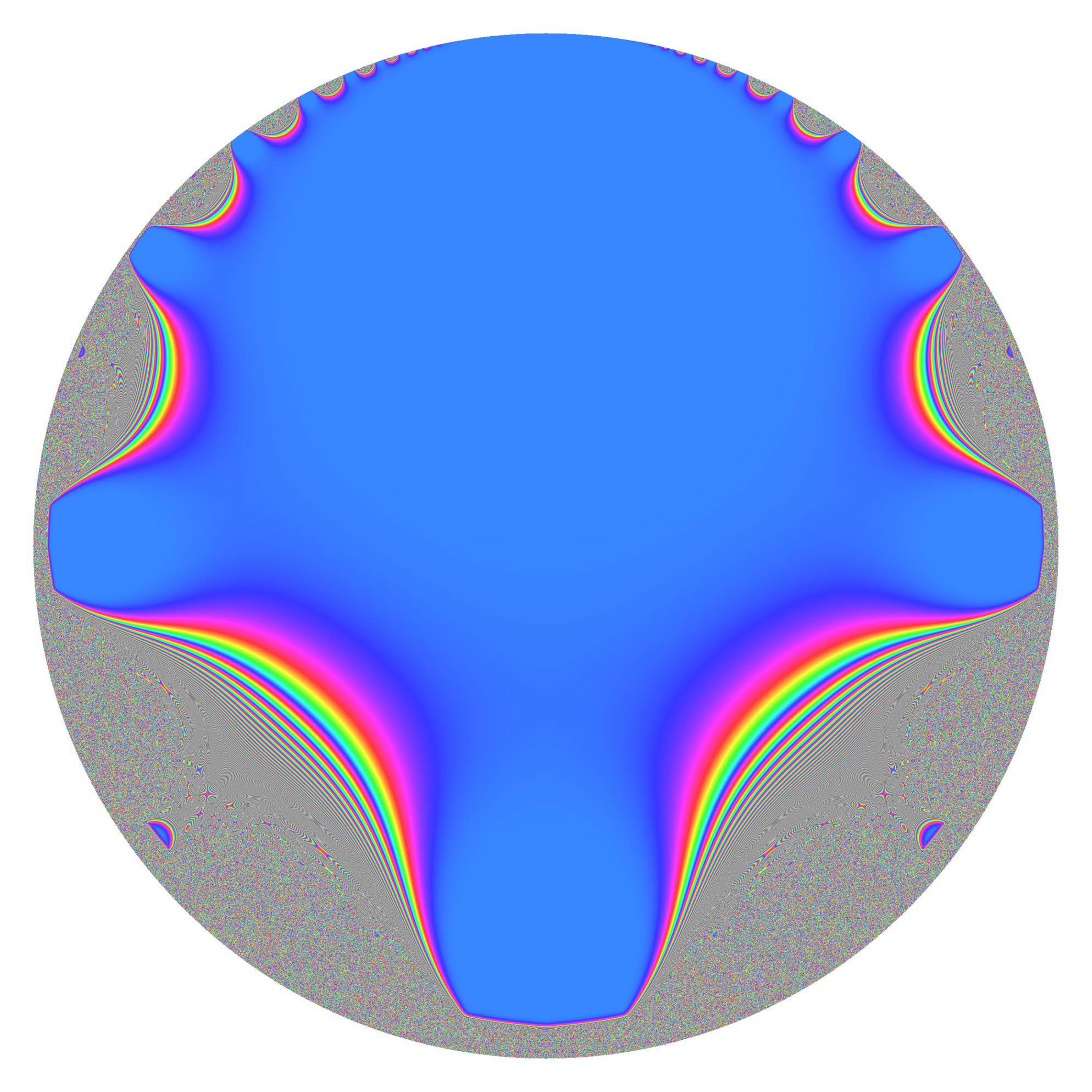}
    \includegraphics[width=0.24\textwidth]{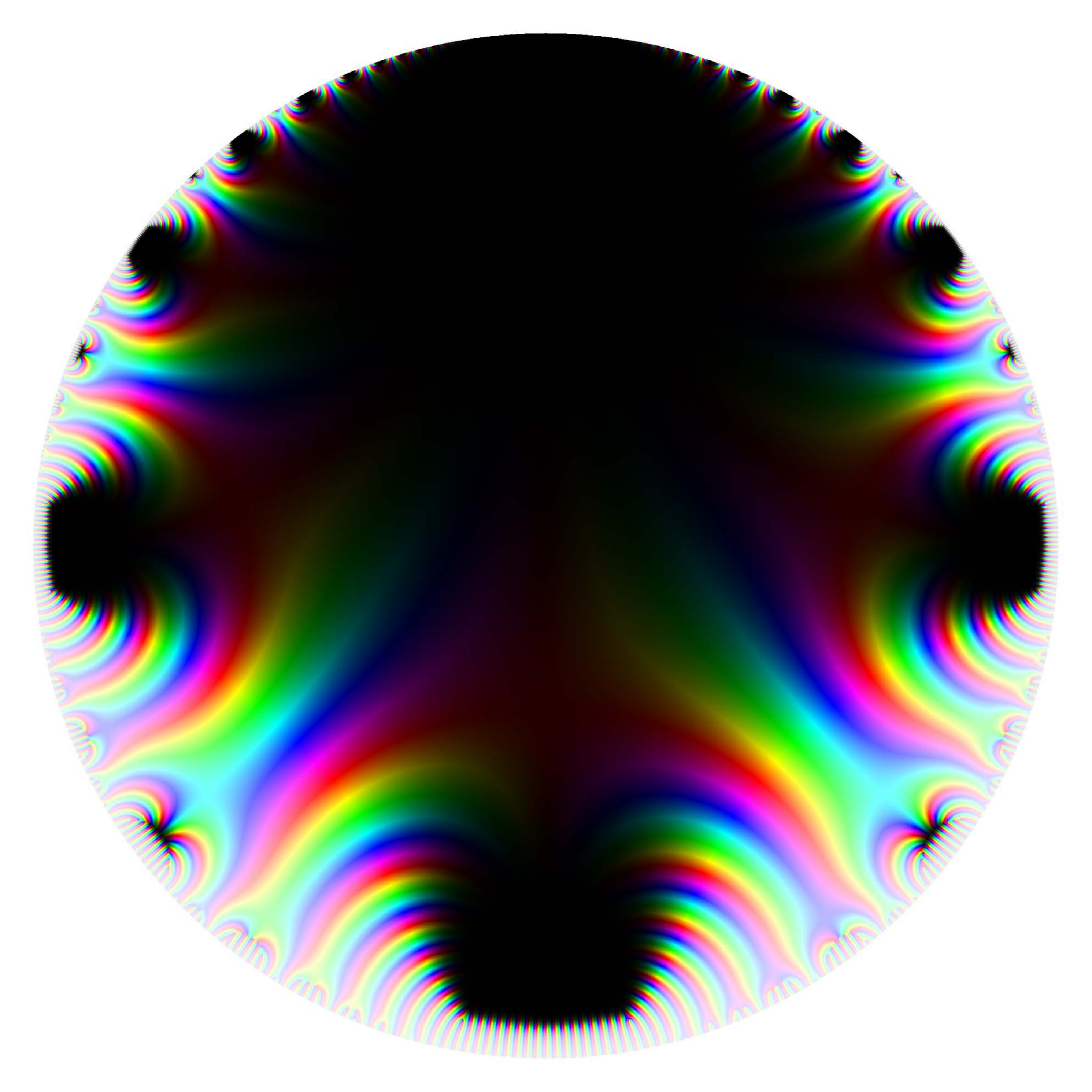}
\\[4pt]
\text{Figure \ref{fig:complex_plot}: Portraits of
    \newformlink{11.2.a.a} and 
    \newformlink{1.12.a.a}
    and their plots using \texttt{complex\_plot} in \sage{}}
\end{gathered}
\end{equation}

Given the number of portraits needed, we limited ourselves to the first 100 Dirichlet coefficients of the trace form, working with  200 bits of precision, evaluating it in a $300 \times 300$ grid in $[-1,1]^2$, and storing the picture as a $184 \times 184$ PNG.
Overall this consumed about 100 CPU days, and their disk footprint is 45 GB.
For aesthetic reasons, the portraits presented here were computed to a higher quality, which creates some discrepancies with the online version, especially in higher weight newforms.

Even though we opted for a plot with less information, it still captures some mathematically interesting features.
For example, the behavior on the edge of the disk is a good indicator for level and weight, see Figures~\ref{fig:level} and \ref{fig:weight}.
\begin{equation} \label{fig:level}\addtocounter{equation}{1} \notag
\begin{gathered}
    \includegraphics[width=0.24\textwidth]{pics/11_2_a_a.pdf}
    \includegraphics[width=0.24\textwidth]{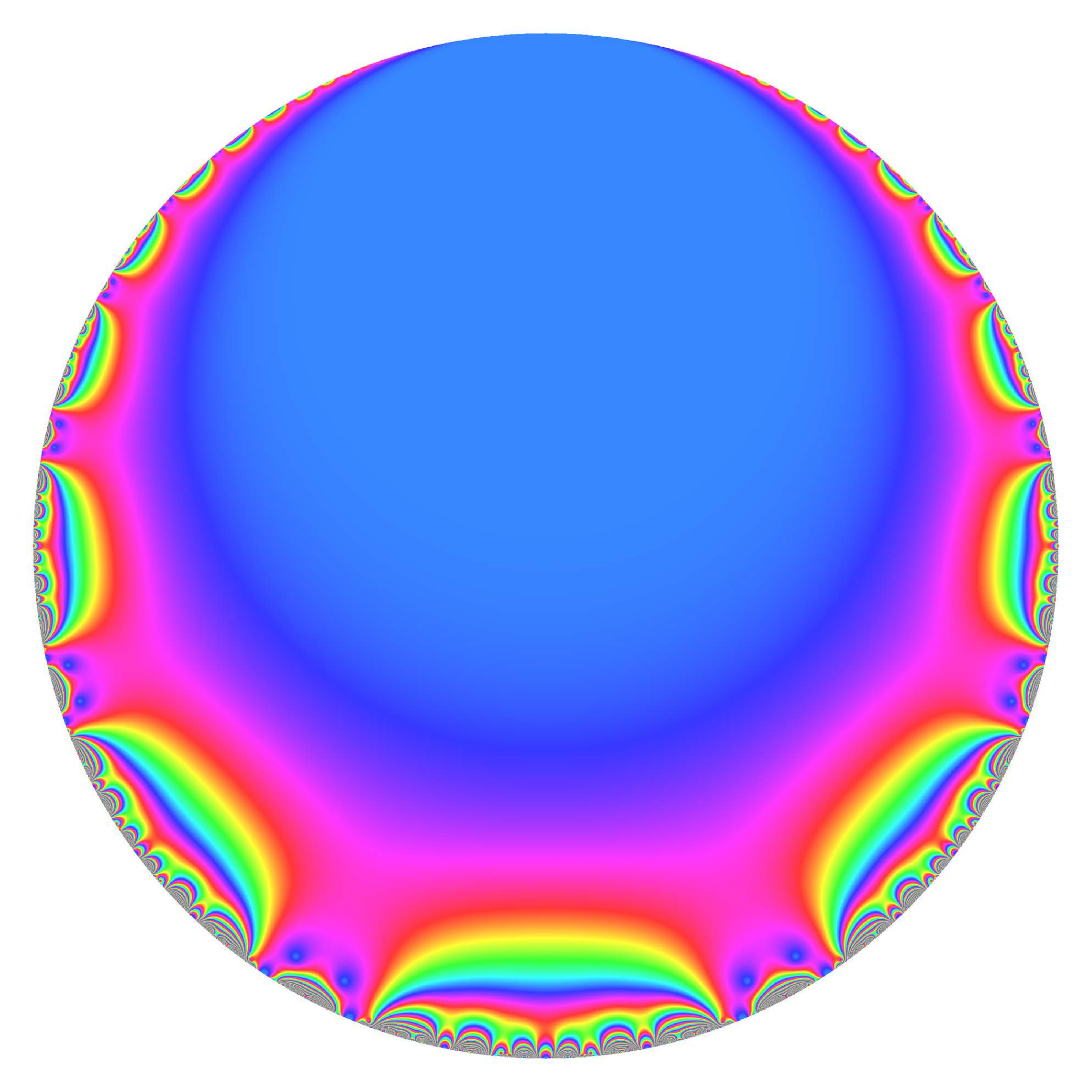}
    \includegraphics[width=0.24\textwidth]{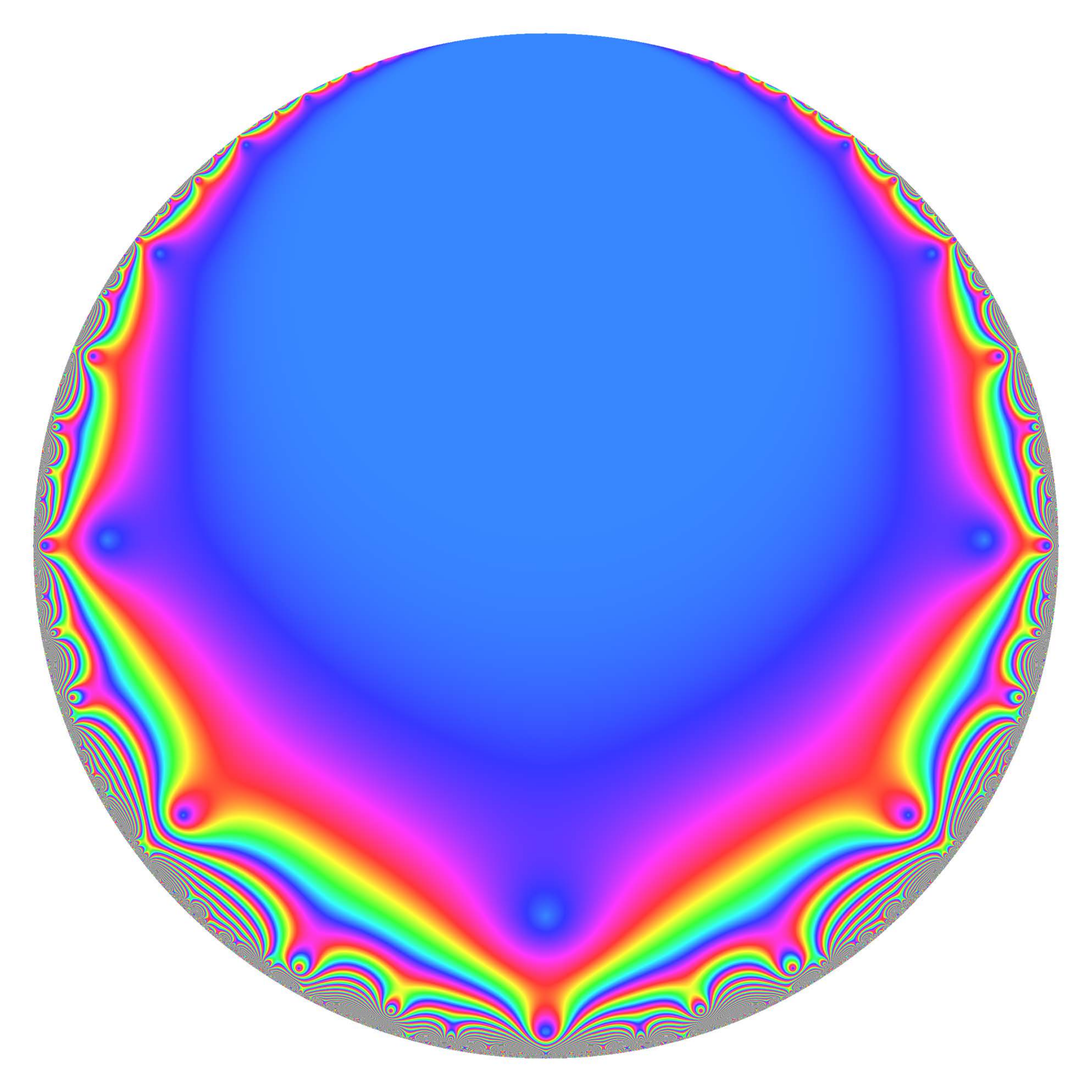}
    \includegraphics[width=0.24\textwidth]{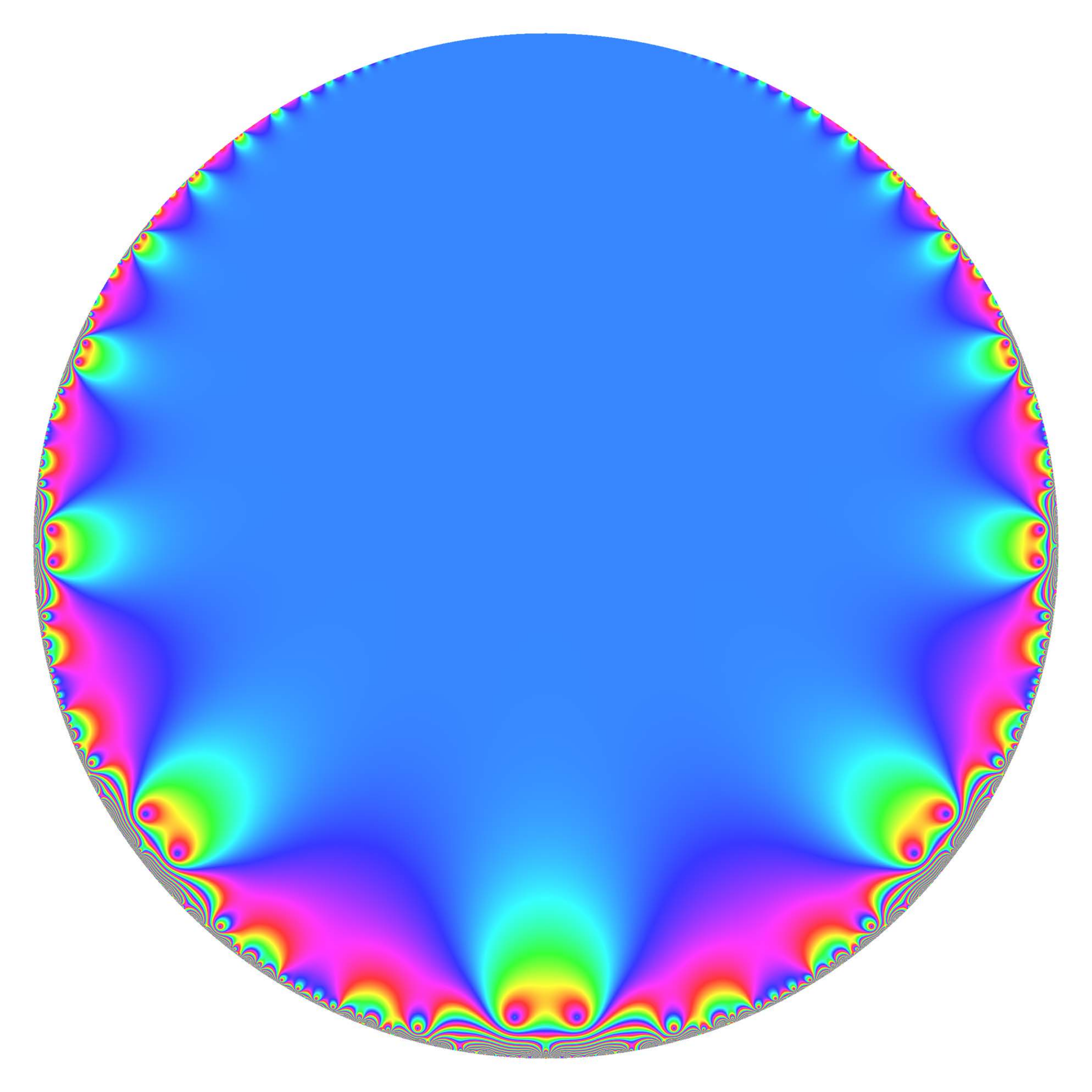}
\\[4pt]
\text{Figure \ref{fig:level}: The portraits for 
    \newformlink{11.2.a.a},
    \newformlink{100.2.a.a},
    \newformlink{1001.2.a.a}, and 
    \newformlink{9996.2.a.a}}
\end{gathered}
\end{equation}
\begin{equation} \label{fig:weight}\addtocounter{equation}{1} \notag
\begin{gathered}
    \includegraphics[width=0.24\textwidth]{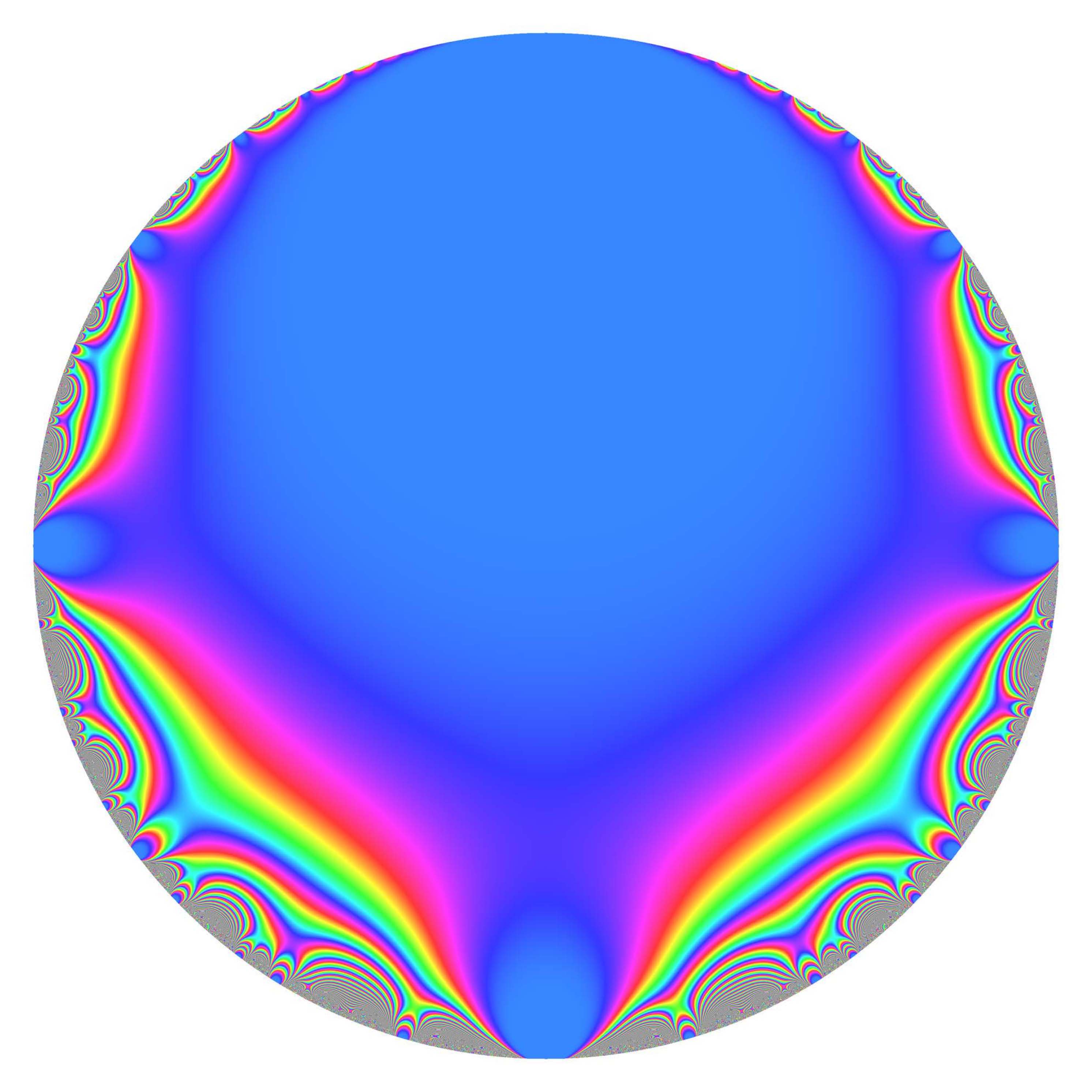}
    \includegraphics[width=0.24\textwidth]{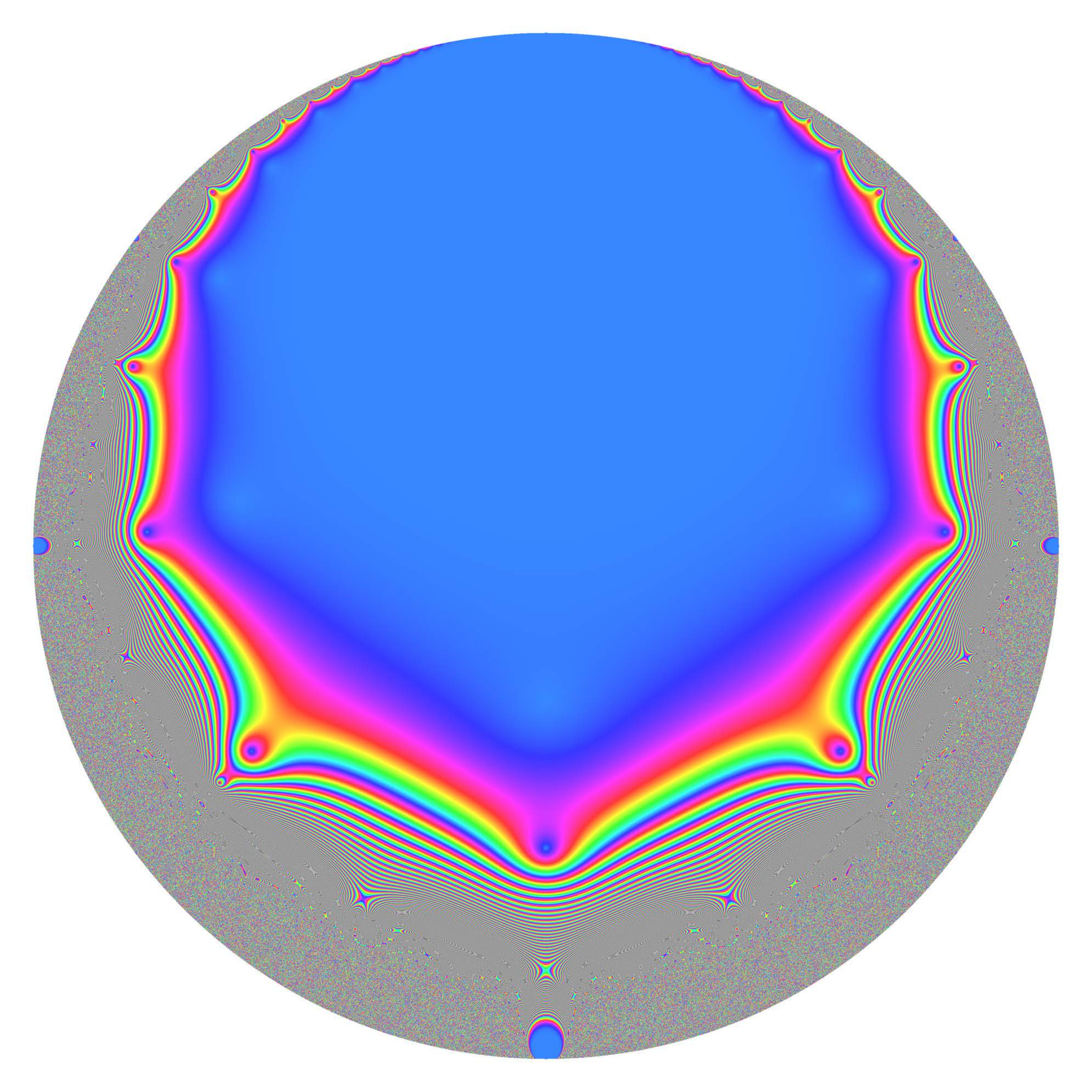}
    \includegraphics[width=0.24\textwidth]{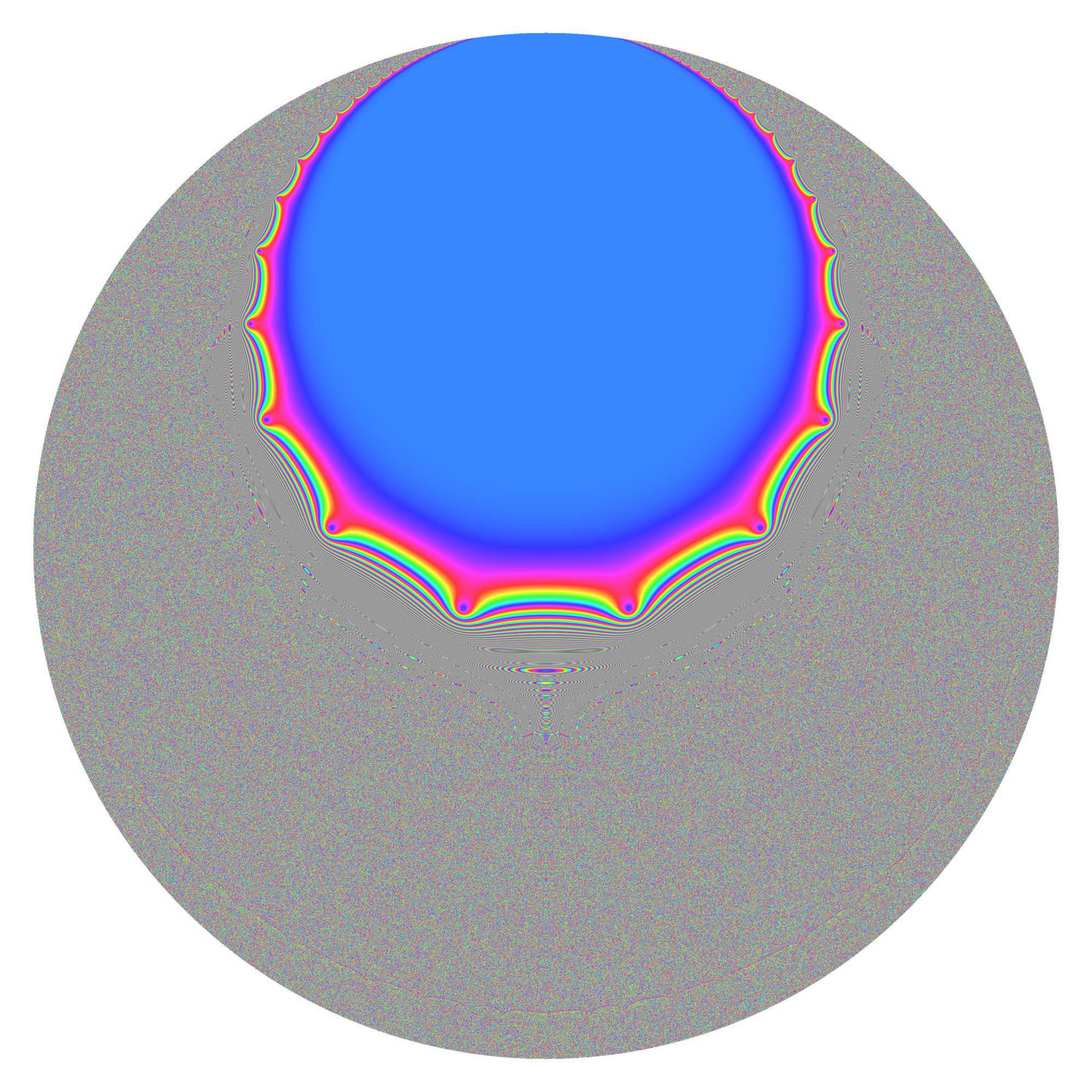}
    \includegraphics[width=0.24\textwidth]{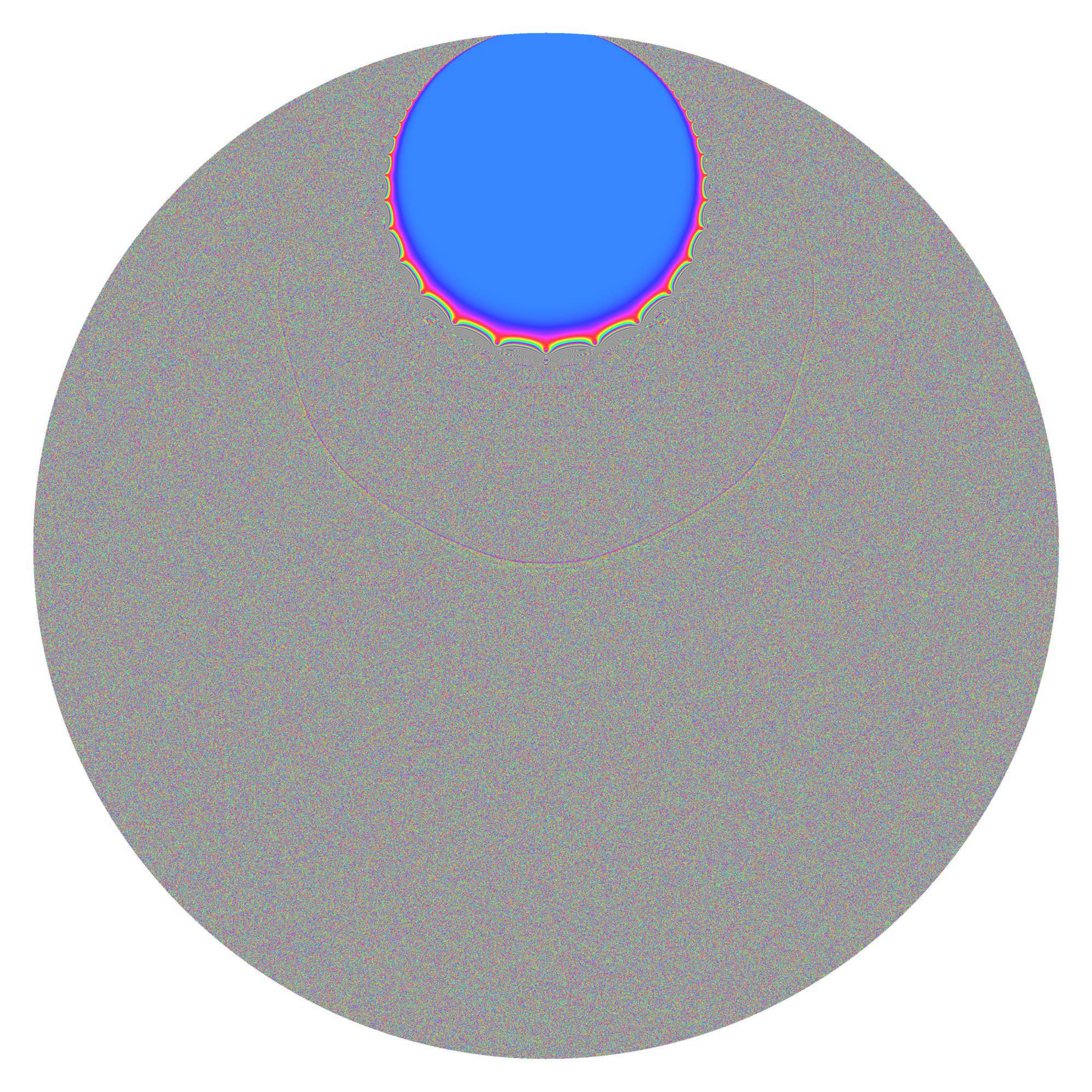}
\\[4pt]
\text{Figure \ref{fig:weight}: The portraits for 
    \newformlink{7.3.b.a},
    \newformlink{7.9.b.a},
    \newformlink{7.27.b.a}, and
    \newformlink{7.81.b.a}}
\end{gathered}
\end{equation}

The size of the blue spot on top center is inversely correlated with the growth of the trace form away from $\infty$, thus for fixed weight this is a good indicator for the dimension, see Figures~\ref{fig:dimension}: their dimensions are 1, 4, 33, and 120, respectively.
\begin{equation} \label{fig:dimension}\addtocounter{equation}{1} \notag
\begin{gathered}
    \includegraphics[width=0.24\textwidth]{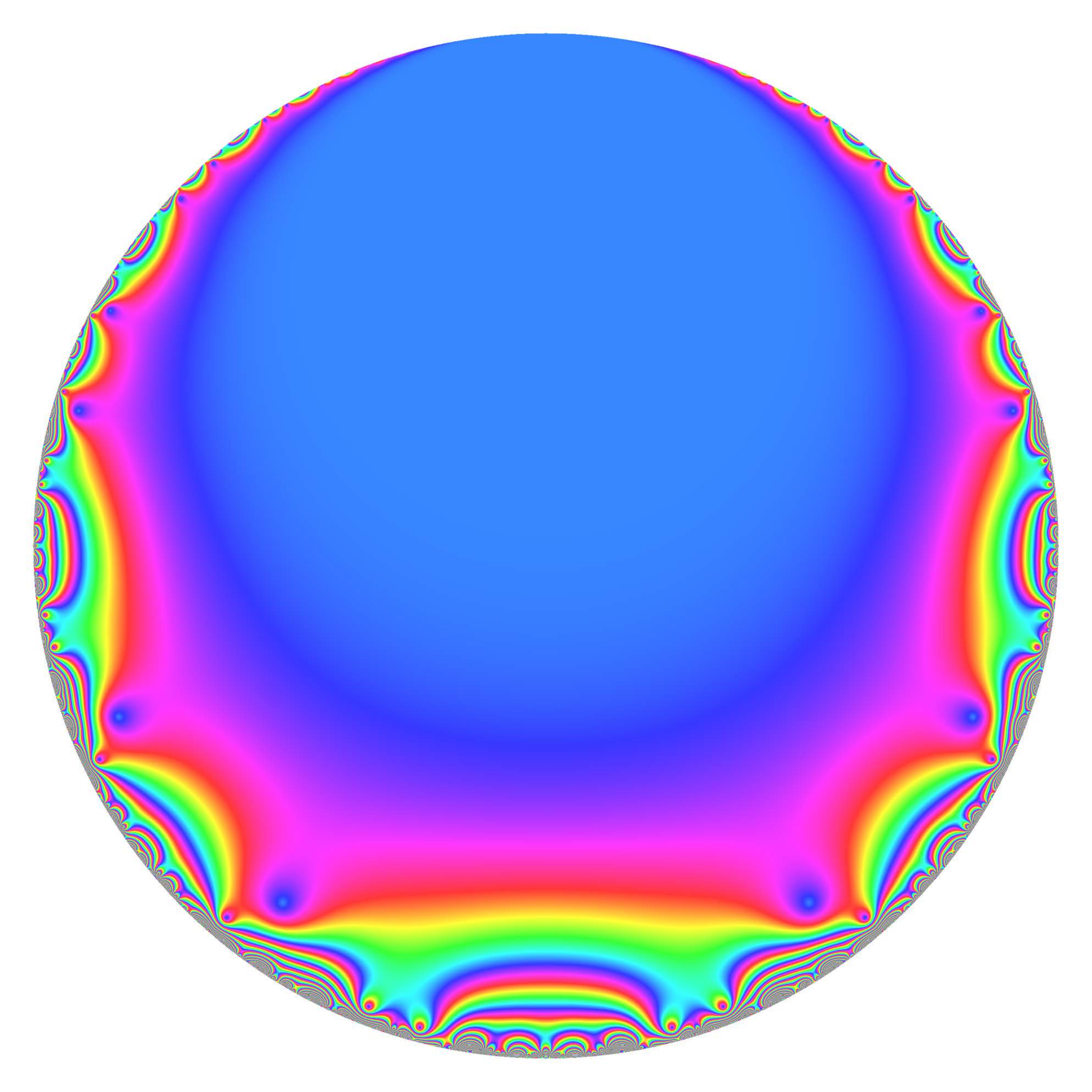}
    \includegraphics[width=0.24\textwidth]{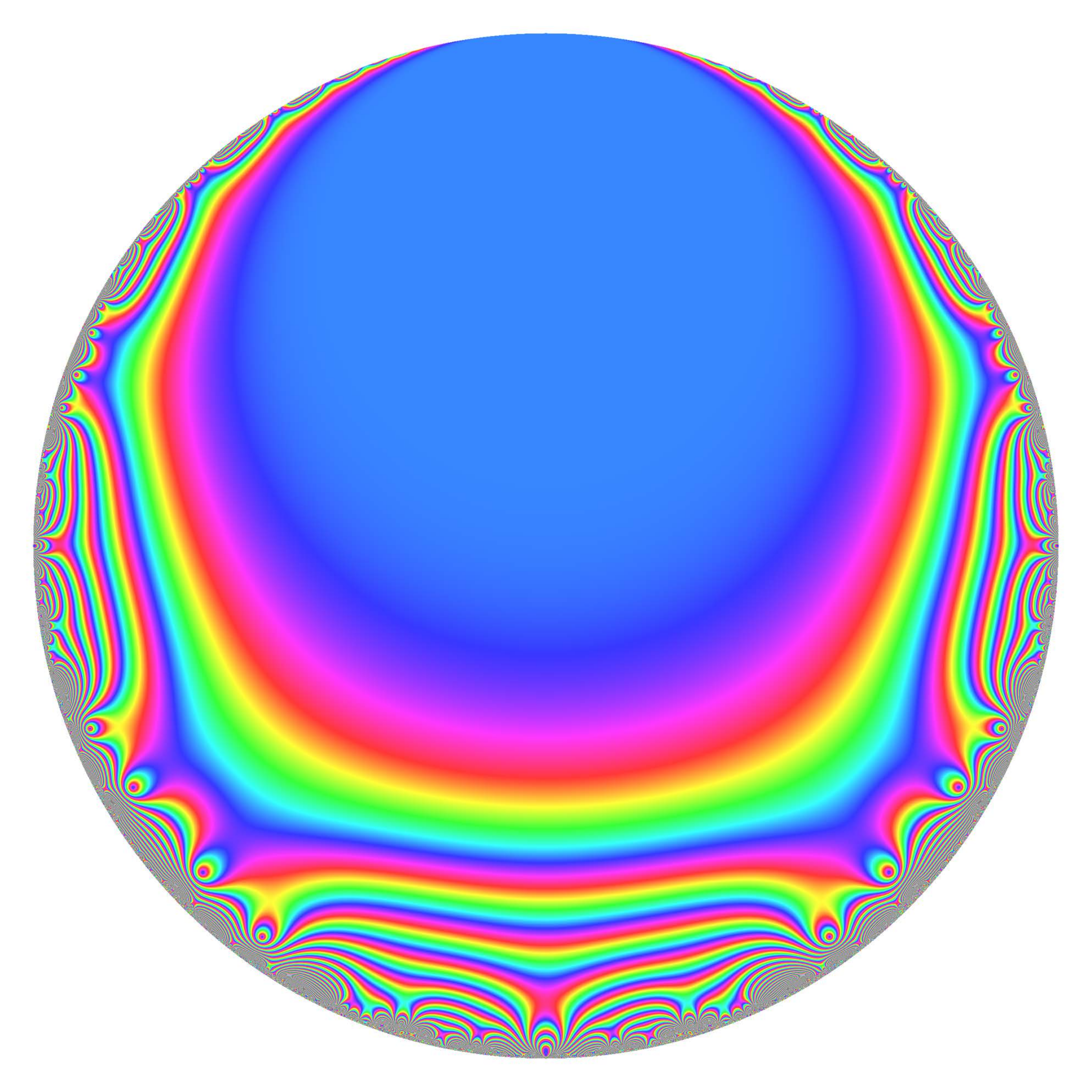}
    \includegraphics[width=0.24\textwidth]{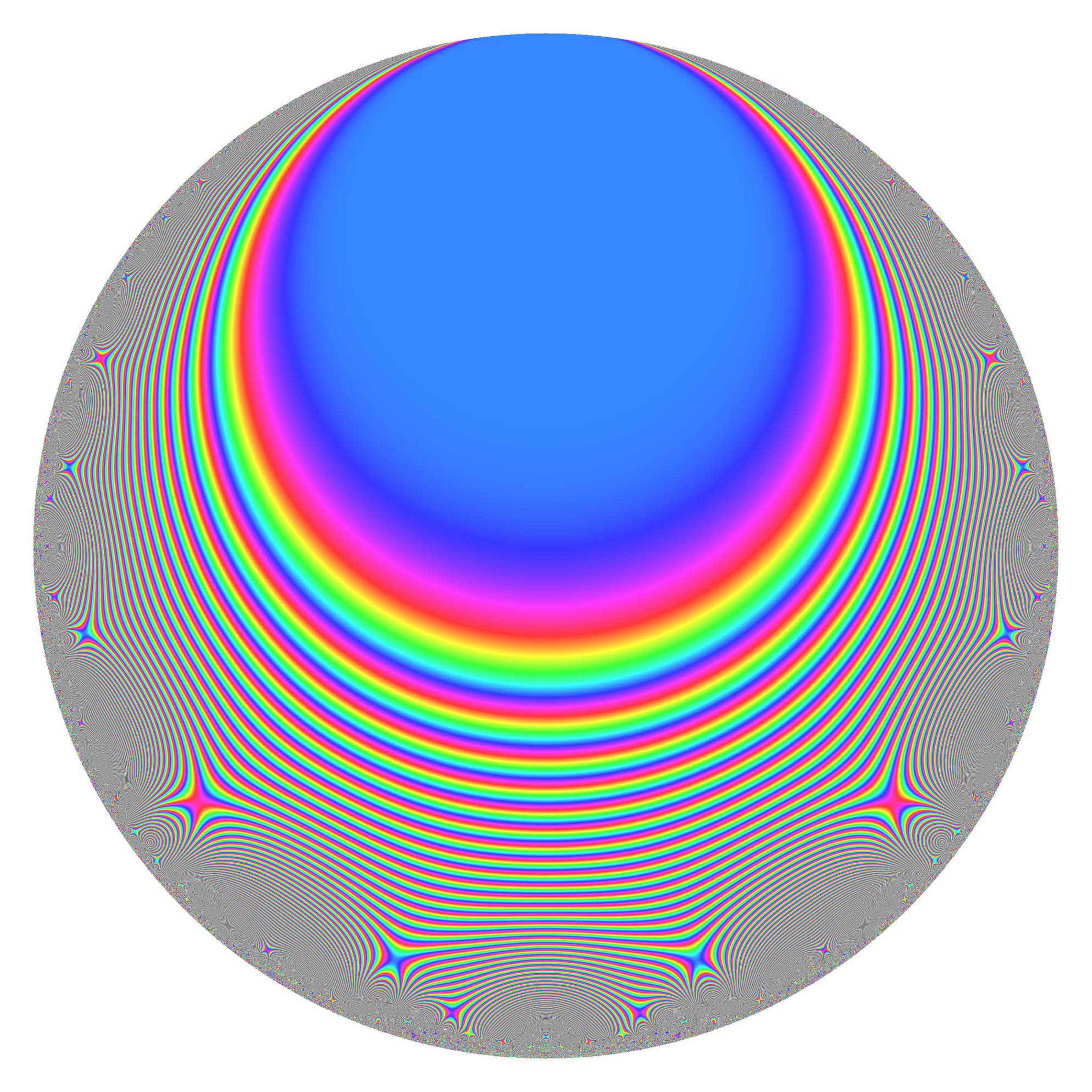}
    \includegraphics[width=0.24\textwidth]{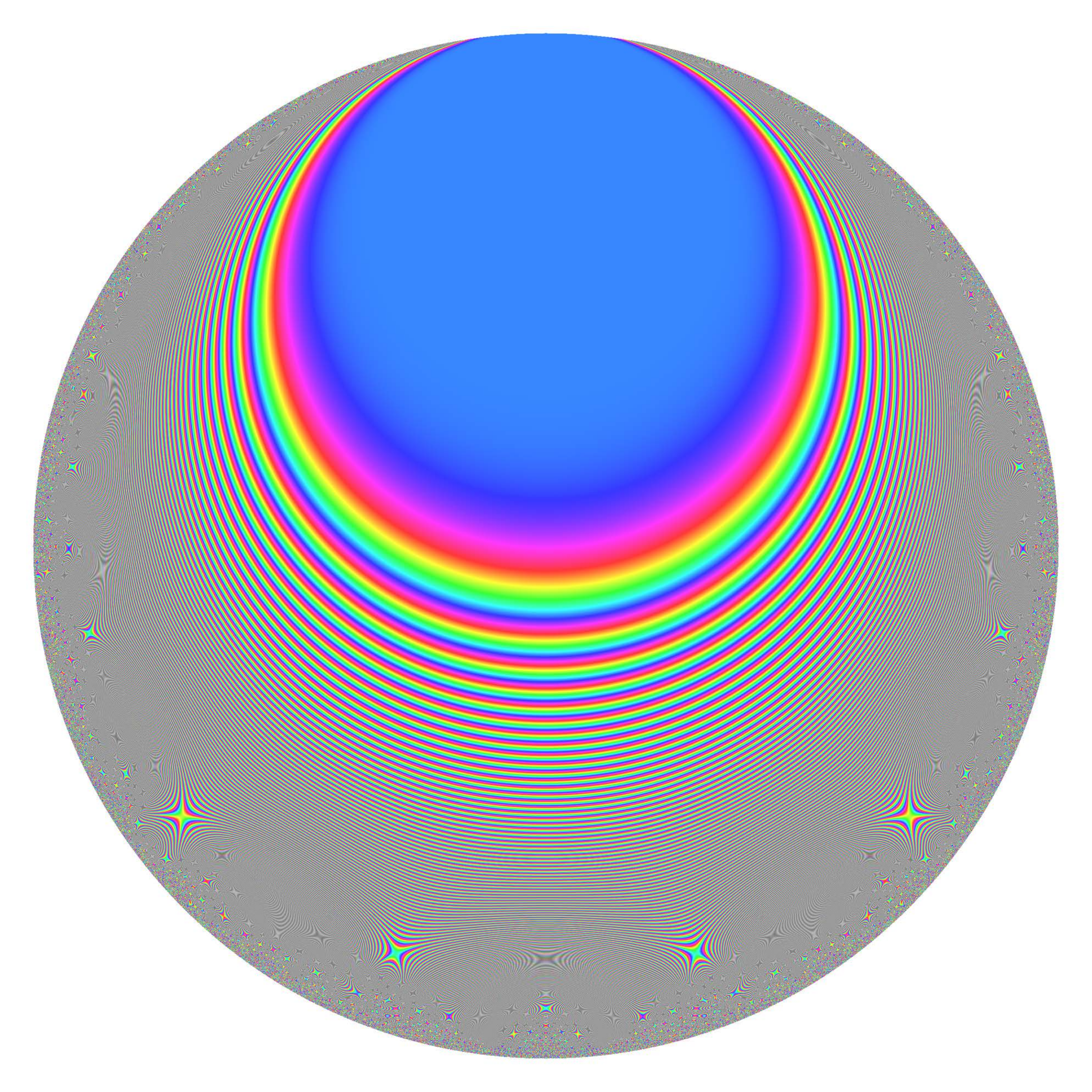}
\\[4pt]
\text{Figure \ref{fig:dimension}: The portraits for 
    \newformlink{9359.2.a.a},
    \newformlink{9359.2.a.e},
    \newformlink{9359.2.a.k}, and
    \newformlink{9359.2.a.r}}
\end{gathered}
\end{equation}

Finally, one could also be tempted to infer the self-twists of a newform by comparing it with other forms in $\Sknew{k} (\Gamma_1 (N))$, 
see Figures~\ref{fig:selt_twists} for $\Sknew{1} (\Gamma_1(164))$.
\begin{equation} \label{fig:selt_twists}\addtocounter{equation}{1} \notag
\begin{gathered}
    \includegraphics[width=0.24\textwidth]{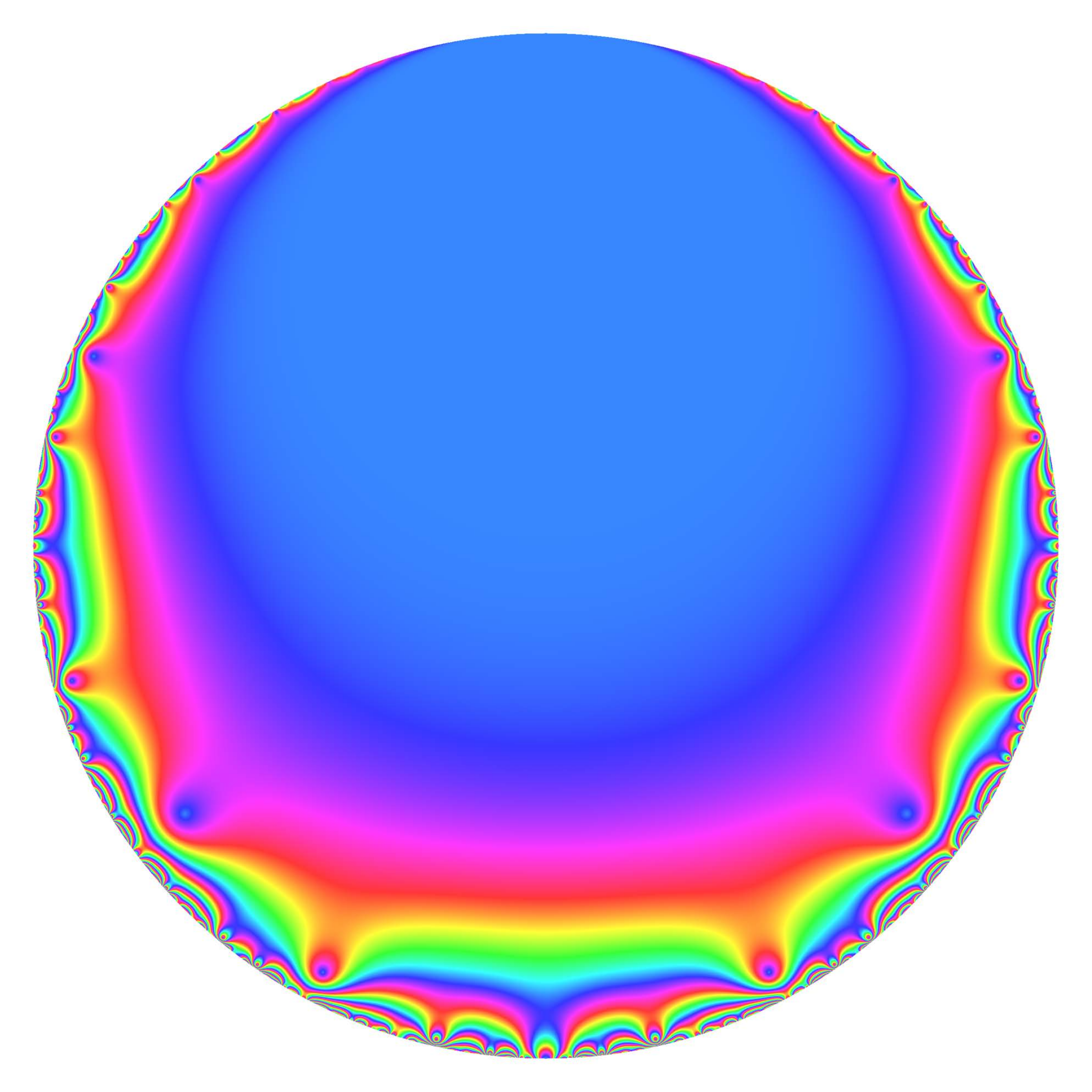}
    \includegraphics[width=0.24\textwidth]{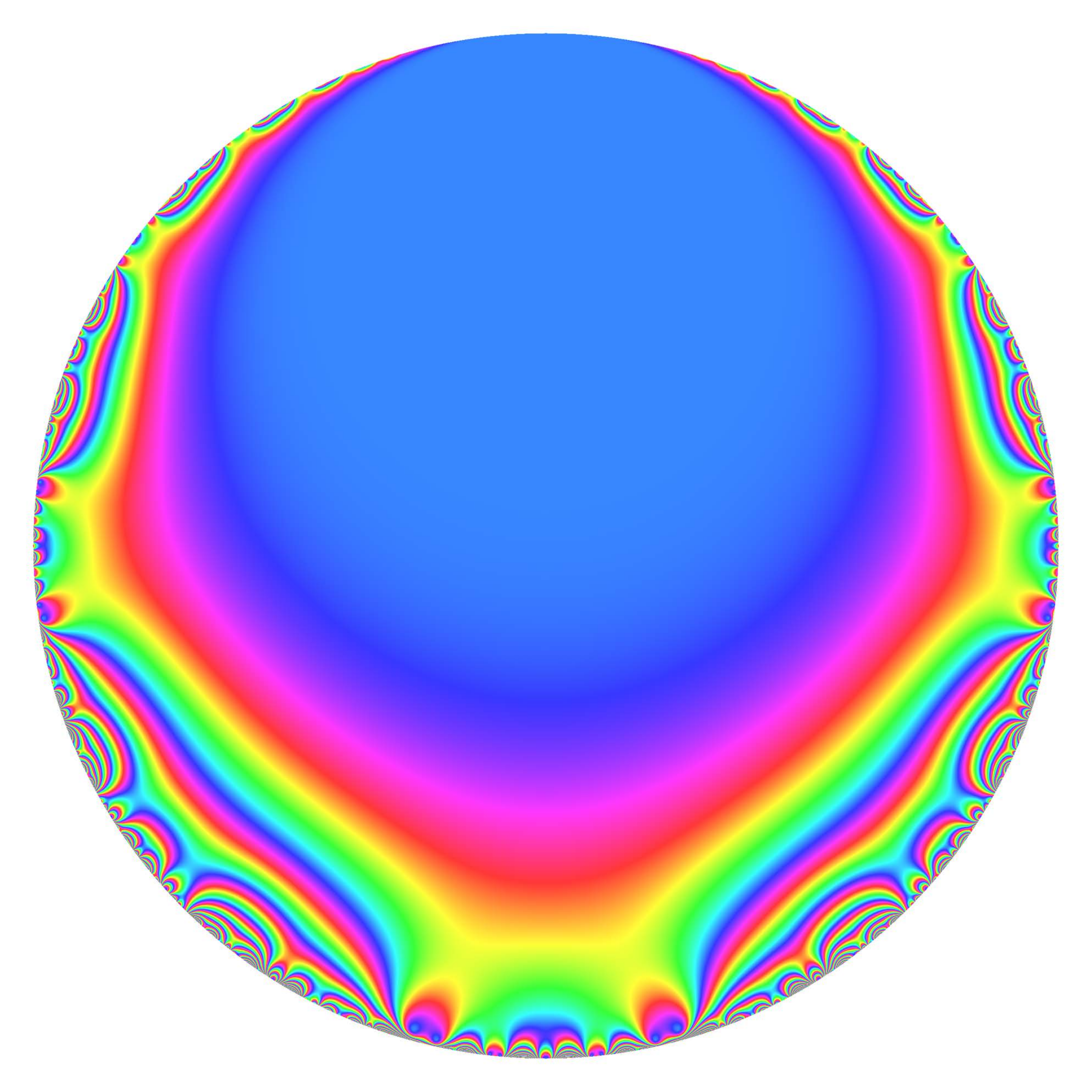}
    \includegraphics[width=0.24\textwidth]{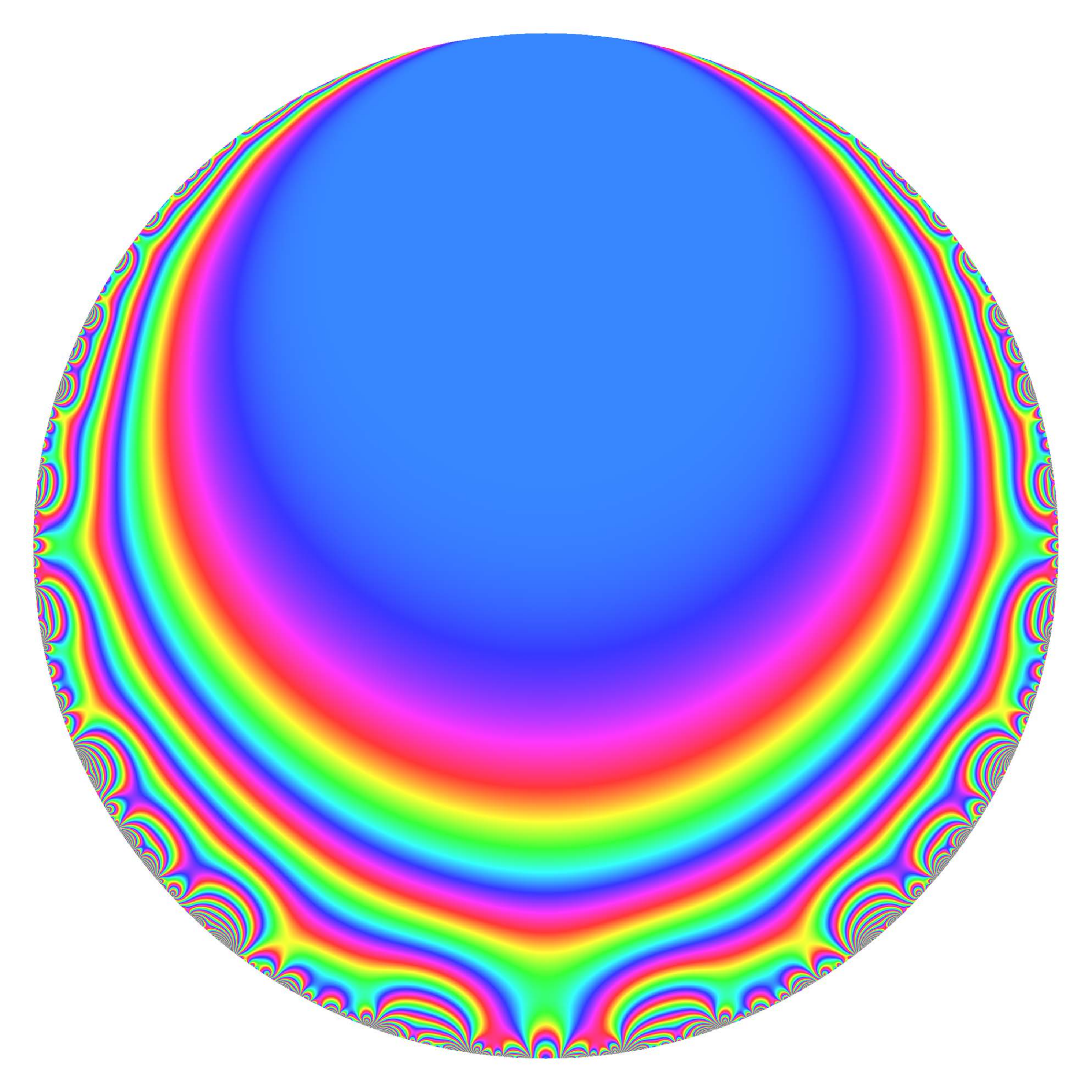}
    \includegraphics[width=0.24\textwidth]{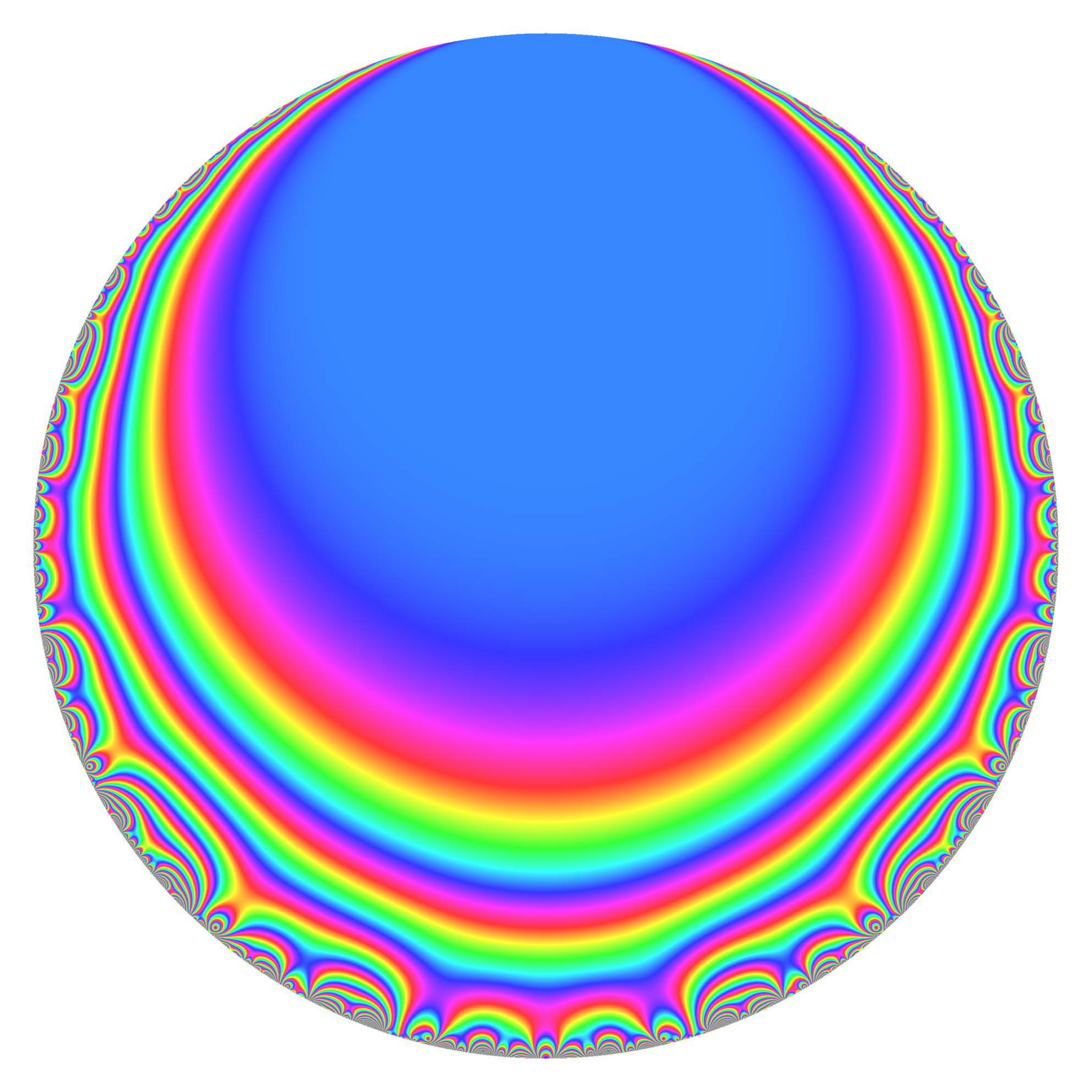}
\\[4pt]
\text{Figure \ref{fig:selt_twists}: The portraits for 
    \newformlink{164.1.d.a},
    \newformlink{164.1.d.b},
    \newformlink{164.1.j.a},
    and
    \newformlink{164.1.l.a}}
\end{gathered}
\end{equation}

\begin{remark}
These portraits differ from those now used in the LMFDB.
Between writing and publishing this article we chose to instead use the
pure phase portraits describe in §2.2.5 of \cite{Lowry-Duda}.
\end{remark}

\subsection{Features}

In parallel to carrying out the computations described elsewhere in this paper,
we rewrote the user interface to the database.  We highlight some of the
more prominent new features in this section, some of which are being extended
to other sections of the LMFDB\@.

The search interface includes multiple modes for viewing results.  After
entering constraints such as weight, level and dimension, there are
four different search buttons available.  In addition to the standard
list of results, a user can choose to go straight to a randomly chosen newform.
Alternatively, there are dimension tables available which display
the dimension of the spaces of newforms as a function of weight and level.
Finally, a table of traces allows for searching on specific Fourier coefficients,
including specifying a particular class modulo an arbitrary integer.  This
feature can be used to find modular forms matching geometric objects
via point-counting.

All of these search modes are also available for newspaces.  For newspaces,
the list mode shows the dimensions of the corresponding newforms as well
as the Atkin-Lehner dimensions in the case of trivial character.
For both newforms and newspaces, users can customize the order of the search
results.

The homepage for an individual newform has also been completely restructured.
Newforms can be downloaded and reconstructed in \magma{}, allowing for further
computations if desired.  
We include complex eigenvalues for embedded modular forms even when
exact Fourier coefficients are not feasible to compute.

One of the key motivations for our extensive computations of (exact or inexact)
Fourier coefficients of newforms is to allow their $L$-functions to be computed.
In addition to providing additional mathematical information about the newform,
such as its analytic rank and special values, this allows us to automatically
connect newforms to other objects in the LMFDB.  Examples include:

\begin{itemize}
\item The $L$-function \href{http://www.lmfdb.org/L/ModularForm/GL2/Q/holomorphic/256/2/a/e/}{$L$(\textsf{256.2.a.e})}
lists both the Bianchi modular form \href{https://www.lmfdb.org/ModularForm/GL2/ImaginaryQuadratic/2.0.4.1/4096.1/b/}{\textsf{2.0.4.1-4096.1-b}}
and the Hilbert modular form \href{https://www.lmfdb.org/ModularForm/GL2/TotallyReal/2.2.8.1/holomorphic/2.2.8.1-1024.1-m}{\textsf{2.2.8.1-1024.1-m}} as origins (both arise as base changes of \newformlink{256.2.a.e}), as well as the corresponding elliptic curve isogeny classes \href{https://www.lmfdb.org/EllipticCurve/2.0.4.1/4096.1/b/}{\textsf{2.0.4.1-4096.1-b}} over $\Q(i)$ and
\href{https://www.lmfdb.org/EllipticCurve/2.2.8.1/1024.1/m/}{\textsf{2.2.8.1-1024.1-m}} over $\Q(\sqrt{2})$.

\item The $L$-function \href{http://www.lmfdb.org/L/ModularForm/GL2/Q/holomorphic/72/2/d/a/}{$L$(\textsf{72.2.d.a})} has (at least) three additional origins: the Hilbert modular form \href{https://www.lmfdb.org/ModularForm/GL2/TotallyReal/2.2.8.1/holomorphic/2.2.8.1-81.1-b}{\texttt{2.2.8.1-81.1-b}}, the elliptic curve isogeny class \href{https://www.lmfdb.org/EllipticCurve/2.2.8.1/81.1/b/}{\texttt{2.2.8.1-81.1-b}}, and the isogeny class \href{https://www.lmfdb.org/Genus2Curve/Q/5184/a/}{\textsf{5184.a}} of the Jacobian of the genus 2 curve \href{https://www.lmfdb.org/Genus2Curve/Q/5184/a/46656/1}{\texttt{5184.a.46656.1}}.

\item The $L$-function \href{http://www.lmfdb.org/L/ModularForm/GL2/Q/holomorphic/1948/1/b/a/}{$L$(\textsf{1948.1.b.a})} also arises as the $L$-function of (the Galois orbit of) the icosahedral Artin representation 
\href{https://www.lmfdb.org/ArtinRepresentation/2.2e2_487.120.1}{\textsf{2.1948.24T576.1}}.  The $L$-functions home page also lists the four conjugate Artin representations (and four embedded \weightone{} newforms) whose $L$-functions are primitive factors of this imprimitive $L$-function of degree 8.
\end{itemize}

\section{Twisting} \label{sec:twists}

In this section, we discuss twists of modular forms and related computational issues.  For background and further reading, we refer the reader to the foundational articles by Ribet \cite{Ribet:galreps,Ribet:endos}.

\subsection{Definitions}

We begin with definitions, followed by some examples.  Throughout this section, let $f \in \Sknew{k}(N,\chi)$ be a newform of weight $k\in \Z_{\ge 1}$, level $N\in\Z_{\ge 1}$, and character $\chi$, and let $K_f \colonequals \Q(\{a_n(f)\}_n) \subseteq \C$ be its coefficient field.  Let $\psi$ be a Dirichlet character of conductor $\cond(\psi)$, and let $\psi_0$ be the primitive Dirichlet character inducing $\psi$ (with $\cond(\psi_0)=\cond(\psi)$).  Then there is a unique newform $g \colonequals f \otimes \psi$ characterized  by the property that 
\begin{equation} \label{eqn:angfpsi0}
a_n(g)=\psi_0(n)a_n(f) \quad \text{ for all $n$ coprime to $N \cond(\psi)$};
\end{equation}
we call $g$ the \defi{twist} of $f$ by $\psi$.  However, more is true: in fact, we have
\begin{equation} \label{eqn:angfpsi}
a_n(g)=\psi_0(n)a_n(f) \quad \text{ for all $n$ coprime to $ \cond(\psi)$}
\end{equation}
including those $n$ that are not necessarily coprime to $N\cond(\psi)$:  see Atkin--Li \cite[Theorem 3.2]{Atkin-Li}.
By the recurrence satisfied by the Hecke operators, \eqref{eqn:angfpsi} is equivalent to the condition 
\begin{equation}
a_p(g)=\psi(p)a_p(f) \quad \text{for all $p \nmid \cond(\psi)$}.
\end{equation}
The newform $g$ has character $\chi\psi^2$ (by \eqref{eq:twistchar} below) and level dividing $\lcm(N,\cond(\psi)\cond(\chi\psi))$ (by Lemma~\ref{lem:MNchipsi} below).
We call the newform $g$ \defi{the twist of $f$ by $\psi$} and say that $g$ is a \defi{twist} of $f$.

As above, the group $\Aut(\C)$ acts on the set of newforms in $\Sknew{k}(N,\chi)$, with $a_n(\sigma(f)) = \sigma(a_n(f))$ for all $n \geq 1$.  We have $\sigma(f) \in \Sknew{k}(N,\sigma(\chi))$, where $\sigma(\chi)(n)=\sigma(\chi(n))$ for all $n \geq 1$. If $g=f\otimes\psi$, then $\sigma(g)=\sigma(f)\otimes\sigma(\psi)$ for all $\sigma\in\Aut(\C)$.  Accordingly, the set
\begin{equation} [f] \otimes [\psi] \colonequals \{f' \otimes \psi' : f' \in [f], \psi' \in [\psi]\} 
\end{equation} 
has an action of $\Aut(\C)$ and so consists of finitely many $\Aut(\C)$-orbits (possibly more than one). Accordingly, we say that $[g]$ is \defi{a twist of $[f]$ by $[\psi]$} if there exist $f' \in [f]$, $\psi' \in [\psi]$, $g' \in [g]$ such that $g' = f' \otimes \psi'$, or equivalently, $[g]\subseteq [f]\otimes[\psi]$.

\begin{example}
The newform orbits \newformlink{3380.1.v.e} and \newformlink{3380.1.v.g} are both twists of \newformlink{3380.1.g.c} by \dircharlink{13.f} (and by \dircharlink{260.bc}). 
\end{example}

With this Galois digression out of the way, we return to the treatment of twists of (embedded) newforms.  

\begin{defn}
Let $\psi$ be a Dirichlet character and $\sigma \colon K_f \hookrightarrow \C$ be a field embedding.  We say that $f$ admits an \defi{inner twist} by the pair $(\psi,\sigma)$ if $f \otimes \psi = \sigma(f)$.  In the special case that $\sigma=\id|_{K_f}$, we say that $f$ admits a \defi{self-twist} by $\psi$.  
\end{defn}

Let $\InnTw(f)$ denote the set of inner twists of $f$ and $\SelfTw(f) \subseteq \InnTw(f)$ the subset of self-twists.  Then projection onto the first component identifies $\SelfTw(f)$ with a subgroup of Dirichlet characters.  By \eqref{eqn:angfpsi}, the form $f$ has an inner twist by $(\psi,\sigma)$ if and only if $\sigma(a_n)=\psi(n)a_n$ for almost all $n$.  The twist is said to be \defi{inner} because such twists stay ``within'' the Galois orbit of $f$ (a nontrivial inner twist is sometimes also referred to as an ``extra twist'').
Every newform has a trivial self-twist by $(\textsf{1.a},\id|_{K_f})$. 

\begin{prop}[{Ribet \cite{Ribet:endos}, Momose \cite{Momose}}] \label{prop:intwistribet}
The following statements hold.
\begin{enumalph}
\item If $(\psi,\sigma) \in \InnTw(f)$, then 
\[ \sigma(\chi)=\chi\psi^2; \]
so if $\psi \in \SelfTw(f)$ then $\psi$ is quadratic.  
\item If $(\psi,\sigma) \in \InnTw(f)$ then $\sigma \in \Aut(K_f)$.
\item $\InnTw(f)$ naturally forms a group under 
\[ (\psi,\sigma) \cdot (\psi',\sigma') \colonequals (\psi\,\sigma(\psi'), \sigma\sigma'). \]
\item There is an exact sequence of groups
\begin{align*} 
1 \to \SelfTw(f) \to \InnTw(f) &\xrightarrow{\pi} \Aut(K_f) \\
(\psi,\sigma) &\mapsto \sigma. 
\end{align*}
Let $A \colonequals \pi(\InnTw(f))$.  Then $\InnTw(f) \simeq \SelfTw(f) \times A$ is a \emph{direct} product.
\item The projection $(\psi,\sigma) \mapsto \psi$ from $\InnTw(f)$ to the set of Dirichlet characters is an injective map of sets.  
\item The group $A$ is abelian.
\item Suppose $\SelfTw(f)$ is trivial.  Then $\pi$ is an isomorphism and the assignment $\sigma \mapsto \psi_\sigma$ if and only if $(\psi_\sigma,\sigma) \in \InnTw(f)$ is a well-defined $1$-cocycle, i.e., 
\[ \psi_{\sigma\sigma'}=\psi_\sigma \sigma(\psi_{\sigma'}). \]
\end{enumalph}
\end{prop}

\begin{proof}
These results originate with Ribet \cite[\S 3]{Ribet:endos} and Momose \cite[Lemma (1.5)]{Momose}, but they work under the hypothesis that $f$ has no self-twists.  For clarity, we repeat these arguments to show this hypothesis is unnecessary.  Let $f(q)=\sum_n a_n q^n$.

Part (a) follows by looking at (Nebentypus) characters using the Hecke recurrence (or the determinant of the associated Galois representations).  Explicitly, on the one hand, the character of $\sigma(f)$ is $\sigma(\chi)$; on the other, if $\epsilon$ is the character of $f \otimes \psi$ then for all good primes $p$ the Hecke recurrence reads
\begin{equation}\label{eq:twistchar}
\begin{aligned}
\eps(p)p^{k-1} = a_{p}(f \otimes \psi)^2 - a_{p^2}(f \otimes \psi)^2 = \psi(p)^2(a_p(f)^2 - a_{p^2}(f)) = \psi(p)^2\chi(p)p^{k-1}  
\end{aligned}
\end{equation}
so $\eps=\chi\psi^2$.  Consequently, a self-twist by $\psi$ gives $\chi=\chi\psi^2$, so $\psi^2$ is the trivial character. 

For part (b), by (a) we have $\psi^2 = \sigma(\chi)\chi^{-1}$, and we claim $\psi$ takes values in $\Q(\chi)$: indeed, if $\chi(n)=\zeta$ is a primitive $d$th root of unity, then checking cases based on the parity of $d$ reveals that $\sigma(\zeta)/\zeta \in \langle \zeta^2 \rangle$.  Since $\Q(\psi) \subseteq K_f$, we conclude $\sigma(a_n)= \psi(n) a_n \in K_f$ for almost all $n$, so $\sigma(K_f) \subseteq K_f$ as desired.

For part (c), we start with $\sigma'(a_n) = \psi'(n) a_n$ and apply $\sigma$ to get 
\[ (\sigma\sigma')(a_n) = \sigma(\psi')(n) \sigma(a_n) = \sigma(\psi')(n)\psi(n)a_n \] 
for almost all $n$, so $(\psi\sigma(\psi'),\sigma\sigma') \in \InnTw(f)$.  This product is associative: the identity element in $\InnTw(f)$ is $(\textsf{1.a},\id|_{K_f})$, and inverses are given by $(\psi,\sigma)^{-1} = (\sigma^{-1}(\psi),\sigma^{-1})$.   

In part (d), the exact sequence is evident from (c).  The group $\InnTw(f)$ visibly has the structure of a semidirect product $\InnTw(f) \simeq \SelfTw(f) \rtimes A$ via $A \to \Aut(\SelfTw(f))$ by $\sigma \mapsto (\psi \mapsto \sigma(\psi))$.  However, by (b) $\SelfTw(f)$ consists only of quadratic characters, so $\sigma(\psi)=\psi$ for all $\sigma$ so the product is direct.

Part (e) follows from the fact that $\psi$ uniquely determines $\sigma$.  

Part (f) is claimed by Ribet \cite[Proposition (3.3)]{Ribet:endos}: we prove it as follows.  As in (a), let $\chi(n)=\zeta$ and $\sigma(\chi)(n)=\zeta^k$.  Then again $\psi(n)=\zeta^{(k-1)/2}$ (for some choice of square root of $\zeta$).  Write similarly $\sigma'(\chi)(n)=\zeta^{k'}$.  Then 
\[ \frac{\sigma'(\psi)}{\psi}(n) = \frac{\zeta^{k'(k-1)/2}}{\zeta^{(k-1)/2}} = \zeta^{(k-1)(k'-1)/2} \] 
is well-defined, and by symmetry this is equal to $(\sigma(\psi)/\psi)(n)$, giving $\psi\,\sigma(\psi')=\psi'\,\sigma'(\psi)$,  
and similarly $\sigma'(\chi)(n)=\zeta^{k'}$.  This calculation shows the projection of the products $(\psi,\sigma)(\psi',\sigma')=(\psi\sigma(\psi'),\sigma\sigma')$ and $(\psi',\sigma')(\psi,\sigma)=(\psi'\sigma'(\psi),\sigma\sigma')$ agree.  By part (e), it follows that $\sigma\sigma'=\sigma'\sigma$ and $A$ is abelian.

Finally, part (g) is immediate from (c).
\end{proof}

\begin{exm}
Consider the (embedded) newform \newformlink{180.1.m.a.107.2}; it represents the unique newform orbit in the space \newformlink{180.1.m} of weight $1$ and level $180$ with character orbit $\textsf{180.m}$, whose $q$-expansion begins
\[
f(q)=  q  - \zeta_{8}^{3} q^{2}   - \zeta_{8}^{2} q^{4}   + \zeta_{8}^{3} q^{5}   - \zeta_{8} q^{8}  +O(q^{10}),
\]
where $\zeta_8=\exp(2\pi i/8)=(1+i)/\sqrt{2}$ is the primitive eighth root of unity in the upper quadrant and $K_f=\Q(\zeta_8)$.  

The group $\SelfTw(f)$ of self-twists is of order $2$ with nontrivial character \href{http://www.lmfdb.org/Character/Dirichlet/4/3}{\textsf{4.3}}, the quadratic character of conductor $4$ associated to the field $\Q(\sqrt{-1})$.  The group of inner twists has order $\#\InnTw(f)=8$, and we compute $\InnTw(f) \simeq (\Z/2\Z)^3$, generated by the elements 
\[ (\href{http://www.lmfdb.org/Character/Dirichlet/4/3}{\textsf{4.3}}, \id), (\href{http://www.lmfdb.org/Character/Dirichlet/3/2}{\textsf{3.2}}, \zeta_8 \mapsto -\zeta_8), (\href{http://www.lmfdb.org/Character/Dirichlet/5/3}{\textsf{5.3}}, \zeta_8 \mapsto \zeta_8^3). \]
The character $\psi_5$ with label \href{http://www.lmfdb.org/Character/Dirichlet/5/3}{\textsf{5.3}} has order $4$, so letting $\sigma_3 \in \Aut(\Q(\zeta_8))$ by $\sigma_3(\zeta_8) = \zeta_8^3$, we have 
\[ (\psi_5,\sigma_3)^2 = (\psi_5\,\sigma_3(\psi_5), \sigma_3^2)= (\psi_5 \psi_5^{-1}, \id) = 1. \]

The projection of $\InnTw(f)$ onto the set of characters yields characters with conductors $1$, $3$, $4$, $5$, $12$, $15$, $20$, $60$.
\end{exm}

\begin{exm}
For $f$ with label \newformlink{361.2.e.d} and $K_f = \Q(\zeta_{18})$, we have no nontrivial self-twists and $\pi \colon \InnTw(f) \to \Aut(K_f)$ is an isomorphism onto its image.   In fact, we compute that $\pi$ is surjective, so $\InnTw(f) \simeq \Z/6\Z$.  More precisely, the elements of order $3$ in $\InnTw(f)$ correspond to the characters
\href{http://www.lmfdb.org/Character/Dirichlet/19/7}{\textsf{19.7}} and \href{http://www.lmfdb.org/Character/Dirichlet/19/11}{\textsf{19.11}} of order $3$, and in the character orbit \dircharlink{19.e} there are three characters whose elements match with automorphisms of order $2$ and two of order $6$.  
\end{exm}

\begin{exm}
Among the forms of weight $k=2$, trivial character, and dimension $2$, we can \href{http://www.lmfdb.org/ModularForm/GL2/Q/holomorphic/?weight=2&char_order=1&dim=2&has_inner_twist=yes&count=50&search_type=List}{search for forms with inner twist}, and we should see a table that matches Cremona \cite[Table 3]{Cremona:abextratwist} up to level $N \leq 300$.  The lists match with one exception: we found one form \newformlink{169.2.a.a} that was missed by Cremona.
\end{exm}

Newforms of weight $k\ge 2$ that admit nontrivial self-twists are commonly said to have \emph{complex multiplication}, for reasons we now explain.

\begin{prop}[Ribet] \label{prop:ribetcm}
The following statements hold.
\begin{enumalph}
\item If $k \geq 2$ and $f$ has nontrivial self-twist by $\psi$, then $\psi$ is associated to an imaginary quadratic field and is unique, i.e., $\SelfTw(f) \simeq \Z/2\Z$.  
\item If $k=1$, then $f$ has nontrivial self-twist by $\psi$ if and only if $f$ has dihedral projective image.  If so, then $\psi$ may be real or imaginary and $\SelfTw(f)$ is a subgroup of $(\Z/2\Z)^2$.
\end{enumalph}
\end{prop}

\begin{proof}
For part (a), see Ribet \cite[Theorem (4.5)]{Ribet:galreps}, a consequence of the theory of complex multiplication.

For part (b), we recall \secref{sec:weight1} and observe that $f$ has self-twist by $\psi$ if and only if $a_p(f)=0$ for all $p$ inert in $\Q(\psi)$ and by classification this happens if and only if the image of the projective Galois representation is dihedral.  In this case, let $L$ be the fixed field of the kernel of the projective Galois representation associated to $f$, so $\Gal(L\,|\,\Q) \simeq D_n$, the dihedral group of order $2n$.  Then for each quadratic subfield $F \subseteq L$, the form $f$ has self-twist by the character associated to $F$.  Accordingly, when $n>2$ the subfield $F$ and associated self-twist character are unique, and when $n=2$ (so $K$ is biquadratic) there are three distinct subfields and corresponding characters and there is a real quadratic subfield.  
\end{proof}

Example \ref{exm:D2} shows that forms in Proposition \ref{prop:ribetcm}(b) indeed occur.  In light of Proposition \ref{prop:ribetcm}, we make the following definition.

\begin{defn}
We say $f$ has \defi{real multiplication (RM)} if $f$ has self-twist by a character attached to a real quadratic field and \defi{complex multiplication (CM)} if $f$ has self-twist by a character attached to an imaginary quadratic field.
\end{defn}

\begin{rmk}
It is common in the literature to just replace the term \emph{self-twist} by \emph{complex multiplication}.  By Proposition \ref{prop:ribetcm}(a), there is no harm in this for weight $k \geq 2$, 
but for weight $k=1$ we think this is potentially confusing, and we want to avoid saying ``$f$ has complex multiplication by $\Q(\sqrt{5})$.''  
\end{rmk}

\begin{exm}
As in the proof of Proposition \ref{prop:ribetcm}(b), weight $1$ forms can have \href{http://www.lmfdb.org/ModularForm/GL2/Q/holomorphic/?hst=List&weight=1&cm=no&rm=yes&search_type=List}{RM} or \href{http://www.lmfdb.org/ModularForm/GL2/Q/holomorphic/?hst=List&weight=1&cm=yes&rm=no&search_type=List}{CM} or \href{http://www.lmfdb.org/ModularForm/GL2/Q/holomorphic/?hst=List&weight=1&cm=yes&rm=yes&search_type=List}{both}.  Forms with RM correspond precisely to ray class characters of real quadratic fields that are of mixed signature (i.e., even at one real place and odd at another).  
\end{exm}

\begin{exm}
CM modular forms may also have an inner twist that is not a self-twist: the smallest example by analytic conductor is \newformlink{52.1.j.a}, having CM by $\Q(\sqrt{-1})$ and two inner twists that are not self-twists.  This phenomenon is not restricted to weight $1$, for example the same is true of the form with label \newformlink{20.2.e.a}.
\end{exm}

Continuing with the theme of working with newforms that have not yet been embedded, we conclude this section by showing that the inner twist group is well-defined on the Galois orbit.   

\begin{lem}
For all $\tau \in \Aut(\C)$, we have an isomorphism of groups
\begin{equation}
\begin{aligned}
\InnTw(f) &\xrightarrow{\sim} \InnTw(\tau(f)) \\
(\psi,\sigma) &\mapsto (\tau\psi,\tau\sigma\tau^{-1}).
\end{aligned}
\end{equation}
\end{lem}

\begin{proof}
From $\sigma(a_n)=a_n \psi(n)$ for almost all $n$ we conclude
\[ (\tau\sigma\tau^{-1})(\tau(a_n)) = \tau(a_n) (\tau\psi)(n) \]
for almost all $n$, and conversely.
\end{proof}

\subsection{Detecting inner twists}

With definitions out of the way, we now drill down to precisely understand the level of twists.  We keep notation from the previous section, in particular $f(q)=\sum_n a_n(f) q^n \in \Sknew{k}(N,\chi)$ is a newform and $\psi$ is a Dirichlet character of conductor $\cond(\psi)$.

\begin{lem} \label{lem:MNchipsi}
Let $M$ be the level of $f \otimes \psi$, so $f \otimes \psi \in \Sknew{k}(M,\chi\psi^2)$.  Then the following statements hold:
\begin{enumalph}
\item For all primes $p$, we have the inequality
\[
\ord_p(M) \leq \max \bigl( \ord_p(N),\ord_p(\cond(\psi)\cond(\chi\psi))\bigr),
\]
with equality if 
$\ord_p(N)\ne \ord_p(\cond(\psi)\cond(\chi\psi))$.  In particular, the level $M$ divides $\lcm(N,\cond(\psi)\cond(\chi\psi))$.  
\item For all primes $p$ we have
\[
\ord_p(\cond(\psi))\le \ord_p(\cond(\psi)\cond(\chi\psi)) \le \max(\ord_p(N),\ord_p(M)).
\]
In particular, $\cond(\psi)\cond(\chi\psi)\mid \lcm(M,N)$, and if $M \mid N$, then $\cond(\psi)\cond(\chi\psi) \mid N$.
\end{enumalph}
\end{lem}

\begin{proof}
Statement (a) can be found in Booker--Lee--Str\"ombergsson \cite[Lemma 1.4]{bookerleestrombergsson}: this improves the upper bound of Shimura \cite[Proposition 3.64]{Shimura:intro} and Atkin--Li \cite[Proposition 3.1]{Atkin-Li} that 
\begin{equation} \label{eqn:atkinlibound}
M \mid \lcm(N,\cond(\psi)^2,\cond(\chi)\cond(\psi)),
\end{equation}
which can be proven directly.  

For statement (b), we prove the contrapositive.  Let $p \mid \cond(\psi)\cond(\chi\psi)$ and suppose that $\ord_p(\cond(\psi)\cond(\chi\psi))> \ord_p(N)$.  Then by (b) we have 
\[ \ord_p(M)=\ord_p(\cond(\psi)\cond(\chi\psi))> \ord_p(N). \qedhere \]
\end{proof}

\begin{lem}\label{aplemma}
If $a_p(f)\ne0$ for some prime number $p$, then $\ord_p(N)\in\{1,\ord_p\cond(\chi)\}$.
\end{lem}

\begin{proof}
If $\ord_p(N)=0$, then $\ord_p(\cond(\chi))=0$; if $\ord_p(N)=1$, also done (without using any hypothesis).  Finally, if $\ord_p \cond(\chi) \neq \ord_p(N)$, i.e., $\chi$ is a character modulo $N/p$, then $a_p(f) \neq 0$ implies $\ord_p(N)=1$ by a result of Li \cite[Theorem~3]{Li}.
\end{proof}

We recall by Proposition \ref{prop:intwistribet}(b) that if $(\psi,\sigma) \in \InnTw(f)$, then $\sigma \in \Aut(K_f)$.  But since we do not need this in the proof, we state the following theorem more generally.  

\begin{thm} \label{thm:innertwist}
Let $f(q) = \sum_n a_n(f) q^n \in \Sknew{k}(N,\chi)$, and let $\sigma\in\Gal(\widetilde{K}_f\,|\,\Q)$ where $\widetilde{K}_f \subseteq \C$ is the Galois closure of $K_f$.  
Let $\psi$ be a primitive Dirichlet character, and let
$\psi'$ be the primitive character that induces $\chi\psi$.
Then $f\otimes\psi=\sigma(f)$ if and only if all of the following conditions hold:
\begin{enumroman}
\item $\cond(\psi)\cond(\psi')\mid N$;
\item $\chi\psi^2=\sigma(\chi)$; and
\item $\sigma(a_p(f))\in\bigl\{a_p(f)\psi(p),\overline{a_p(f)}\psi'(p)\bigr\}$ 
for all primes $p\le \Sturm(k,N)$.
\end{enumroman}
\end{thm}
\begin{proof}
Let $\bar{f}\in \Sknew{k}(N,\overline{\chi})$
denote the dual of $f$, with
coefficients $a_n(\bar{f})=\overline{a_n(f)}$.
Thus $\bar{f}=f\otimes \overline{\chi}$ (cf.\ Atkin--Li \cite[Proposition~1.5]{Atkin-Li} or Ribet \cite[\S 1, p.\ 21]{Ribet:galreps}) and consequently $f\otimes\psi=\bar{f}\otimes\psi'$ as 
\[ a_n(\bar{f})\psi'(n) = a_n(f)\bar{\chi}(n)(\chi\psi)(n) = a_n(f) \psi(n) \]
whenever $\gcd(n,N)=1$.

First we prove $(\Rightarrow)$, and suppose that $f\otimes\psi=\sigma(f)$.  By Proposition \ref{prop:intwistribet} we have $\chi\psi^2=\sigma(\chi)$. Since $\cond(\sigma(f))=\cond(f)=N$,
we have $\cond(\psi)\cond(\psi')\mid N$ by Lemma~\ref{lem:MNchipsi}(c).
Let $D \colonequals \gcd(\cond(\psi),\cond(\psi'))$.  Then
\begin{equation}\label{eqn:condchi}
\cond(\chi) = \cond(\psi'\overline{\psi})\mid\lcm(\cond(\psi),\cond(\psi'))\\
= (\cond(\psi)\cond(\psi')/D) \mid (N/D).
\end{equation}

Let $p$ be prime.  If $p\nmid\cond(\psi)$ then
$\sigma(a_p(f))=a_p(f\otimes\psi)=a_p(f)\psi(p)$.
Similarly, if $p\nmid\cond(\psi')$ then
$\sigma(a_p(f))=a_p(\bar{f}\otimes\psi')=\overline{a_p(f)}\psi'(p)$.
Hence we may suppose that $p\mid D$, so by \eqref{eqn:condchi} we have $\ord_p(N)>\max\{1,\ord_p\cond(\chi)\}$.
By Lemma~\ref{aplemma}, it follows that $a_p(f)=0$, and thus
$\sigma(a_p(f))=a_p(f)\psi(p)$.

Now we prove the converse $(\Leftarrow)$, and suppose that conditions (i)--(iii) hold. 
Let
$M$ be the level of $f\otimes\psi$. Let $Q$ denote the product of
primes $p\mid N$ such that either 
\begin{itemize}
\item $p\nmid M$, or 
\item $a_p(f)=0$ and $a_p(f\otimes\psi)\ne0$. 
\end{itemize}
Let $\xi$ denote
the trivial character modulo $Q$, and define
\begin{equation} \label{eqn:gqxi}
g(q) \colonequals \sum_{n=1}^\infty a_n(f\otimes\psi)\xi(n)q^n.
\end{equation}
We claim that conditions (i)--(ii) imply that $g\in S_k(N,\chi\psi^2)$.
By Atkin--Li \cite[Proposition 3.1]{Atkin-Li} 
it suffices to show that
\begin{equation}\label{lcmcondition}
\lcm\bigl(M,\cond(\psi\psi')Q,Q^2\bigr)\mid N.
\end{equation}
By Lemma~\ref{lem:MNchipsi}(a) and the fact that
$\cond(\psi)\cond(\psi')\mid N$, we have
$$ 
\cond(\psi\psi')\mid M\mid\lcm\{N,\cond(\psi)\cond(\psi')\}=N,
$$
so to prove \eqref{lcmcondition} it suffices to show that
$\ord_p(N)\ge1+\max\{1,\ord_p\cond(\psi\psi')\}$
for all primes $p\mid Q$.  

Let $p$ be such a prime. Then either $p\nmid M$ or
$a_p(f)=\overline{a_p(f)}=0$ and
$a_p(f\otimes\psi)=a_p(\bar{f}\otimes\psi')\ne0$.
In either case we must have $p\mid \gcd(\cond(\psi),\cond(\psi'))$
and, by Lemma~\ref{aplemma},
$\ord_p(M)\in\{1,\ord_p\cond(\psi\psi')\}$. It follows that
$$
\max\{1,\ord_p\cond(\psi\psi'),\ord_p(M)\}
\le\max\{\ord_p\cond(\psi),\ord_p\cond(\psi')\}.
$$
Since $p\mid \gcd(\cond(\psi),\cond(\psi'))$, we have
$\min\{\ord_p\cond(\psi),\ord_p\cond(\psi')\}\ge1$, and hence
$$
\ord_p(\cond(\psi)\cond(\psi'))\ge
1+\max\{1,\ord_p\cond(\psi\psi'),\ord_p(M)\}.
$$
By Lemma~\ref{lem:MNchipsi}(b) we have
$$
\ord_p(N)=\ord_p(\cond(\psi)\cond(\psi'))
\ge1+\max\{1,\ord_p\cond(\psi\psi')\}.
$$
This concludes the proof that $g\in S_k(N,\chi\psi^2)$.

Next,  
we claim that $a_n(g)=\sigma(a_n(f))$ for all
$n\le \Sturm(k,N)$.
Since both sequences are multiplicative and $\chi\psi^2=\sigma(\chi)$,
it suffices to verify this
equality at primes, $p$. There are three cases to consider:
\begin{itemize}
\item If $p\nmid N$ then $a_p(f)\psi(p)=\overline{a_p(f)}\psi'(p)$,
so that $\sigma(a_p(f))=a_p(g)$.
\item If $p\mid N$ and $a_p(f)=0$ then $a_p(g)=0$ by construction, and
$\sigma(a_p(f))=0$.
\item If $p\mid N$ and $a_p(f)\ne0$ then
$0\ne\sigma(a_p(f))\in\{a_p(f)\psi(p),\overline{a_p(f)}\psi'(p)\}$.
\begin{itemize}
\item If $\sigma(a_p(f))=a_p(f)\psi(p)$ then $p\nmid\cond(\psi)$, so
$a_p(f)\psi(p)=a_p(f\otimes\psi)$.
\item If $\sigma(a_p(f))=\overline{a_p(f)}\psi'(p)$ then $p\nmid\cond(\psi')$, so
\[ \overline{a_p(f)}\psi'(p)=a_p(\bar{f}\otimes\psi')=a_p(f\otimes\psi). \]
\end{itemize}
In either case, we conclude that $\sigma(a_p(f))=a_p(f\otimes\psi)=a_p(g)$.
\end{itemize}

By the Hecke--Sturm bound (Proposition \ref{prop:heckesturm}), it follows that $g=\sigma(f)$. Finally, since $f$ is
a newform, $\sigma(f)$ is as well, and thus $\sigma(f)=g=f\otimes\psi$,
by strong multiplicity one.
\end{proof}

We conclude with a variant, similarly useful for algorithmic purposes.  We recall the notion of \emph{distinguishing primes} from \secref{sec:heckekernel}.

\begin{thm} \label{thm:justcheckcoprime}
With the same hypotheses as in Theorem \textup{\ref{thm:innertwist}}, we have $f \otimes \psi = \sigma(f)$ if and only if conditions hold:
\begin{enumroman}
\item $\cond(\psi)\cond(\chi\psi)\mid N$;
\item $\chi\psi^2=\sigma(\chi)$; 
\item $\sigma(a_p(f))=a_p(f)\psi(p)$ for all primes $p \leq \Sturm(k,N)$ with $p \nmid N$; and
\item $\sigma(a_p(f))=a_p(f)\psi(p)$ for $p$ in a set of distinguishing primes for $f$.  
\end{enumroman}
\end{thm}

\begin{proof}
The implication $(\Rightarrow)$ is clear, so we prove $(\Leftarrow)$.  

As in the proof of $(\Leftarrow)$ of Theorem \ref{thm:innertwist}, we again consider the form $g$ as in \eqref{eqn:gqxi} with $\xi$ the trivial character modulo $Q$.  Let $N_g$ be the level of $g$.  Then in the proof we showed that $N_g \mid N$ and $h \colonequals g - \sigma(f) \in S_k(N,\sigma(\chi))$.  By (iii) and Hecke recursion, we have $a_n(h)=0$ for all $n \leq \Sturm(k,N)$ coprime to $N$.  

If $N_g=N$, then by (iv), we have $\sigma(f)=f \otimes \psi$.  So we may assume that $N_g$ is a proper divisor of $N$.  We now employ degeneracy operators to upgrade (iii).  It is convenient to switch from lower-triangular to upper-triangular matrices.  Let
\[ \Gamma^1(N) \colonequals \left\{ \gamma \in \SL_2(\Z) : \gamma \equiv \begin{pmatrix} 1 & 0 \\ * & 1 \end{pmatrix} \psmod{N} \right\} \]
and similarly $\Gamma^0(N)$, and define spaces of modular forms on these groups similarly.  We refer to Diamond--Shurman \cite[\S 5.7]{DiamondShurman} for the results we need.  The groups $\Gamma_1(N)$ and $\Gamma^1(N)$ are conjugate by the matrix $\begin{pmatrix} N & 0 \\ 0 & 1 \end{pmatrix}$, giving an isomorphism $\iota_N \colonequals S_k(\Gamma_1(N)) \to S_k(\Gamma^1(N))$ whose effect on Fourier expansions is $\sum_n b_n q^n \mapsto \sum_n b_n q_N^n$ where $q_N \colonequals \exp(2\pi i z/N)$.  Moreover, this map preserves the Nebentypus character.  For any $d \mid N$, the trace operator defines a map 
\[ \pi_d \colon S_k(\Gamma^1(N)) \to S_k(\Gamma_d) \subseteq S_k(\Gamma^1(N)) \]
where $\Gamma_d \colonequals \Gamma_1(N) \cap \Gamma^0(N/d)$: its effect on Fourier expansions is
\[ \sum_{n=1}^{\infty} b_n q_N^n \mapsto \sum_{\substack{n=1 \\ d \mid n}}^\infty b_n q_N^n. \]  
The operator $\pi_d$ is a projection operator, and for $d,d' \mid N$ with $\gcd(d,d')=1$ we have $\pi_{d}\pi_{d'}=\pi_{d'}\pi_d$.  

Consider 
\[ h' \colonequals \prod_{p \mid N}(1-\pi_p)\iota_N(h) \in S_k(\Gamma^0(N),\chi). \] 
By construction, multiplicativity, and (iii), we have $a_n(h')=0$ for all $n \leq \Sturm(k,N)$.  Then by the Hecke--Sturm bound (Proposition \ref{prop:heckesturm}), we conclude $h'=0$.  Thus 
\begin{equation}
h(q) = \sum_{\substack{n=1 \\ \gcd(n,N) \neq 1}}^{\infty} a_n(h)q^n.
\end{equation}
We have realized $h$ as a sum of oldforms.  Turning this back to $\Gamma_1(N)$, we conclude that
\begin{equation} \label{eqn:supportatp}
h(q) = \sum_{p \mid N} h_p(q^p)
\end{equation}
with $h_p(q) \in S_k(\Gamma_p, \sigma(\chi)_p)$, as in the oldform theory of Atkin--Lehner \cite[Theorem 1]{AtkinLehner} and Li \cite[Corollary 1]{Li}; moreover, $h_p=0$ if and only if $h$ is new at $p$.

We now show that $h=0$.  Let $p \mid N$.  If $\chi$ is not a character modulo $N/p$, then $S_k(\Gamma_p,\sigma(\chi)_p)=0$ so $h_p=0$.  So suppose $\chi$ is a character modulo $N/p$. 
\begin{itemize}
\item Suppose that $a_p(f) \neq 0$.  Then by Lemma \ref{aplemma}, we have $p \parallel N$.  Thus $\ord_p(\cond(\chi))=0$, so by (i) we have $\ord_p(N) \geq 2\ord_p(\cond(\psi))$.  If $\ord_p(\cond(\psi))=0$, then we have twisted by a character trivial at $p$, so $\ord_p(M)=\ord_p(N)$ by Lemma \ref{lem:MNchipsi}(b). Therefore $f \otimes \psi$ is new at $p$, so $g$ is new at $p$ and $a_p(g)=a_p(f \otimes \psi)$ so $h_p=0$.  If instead $\ord_p(\cond(\psi)) \geq 1$, then $p^2 \mid N$, a contradiction.  
\item Suppose $a_p(f)=0$.  If $a_p(f \otimes \psi) \neq 0$, then by construction, $a_p(g)=0$ so by multiplicativity $a_{n}(f)=a_n(g)$ for all $p \mid n$; therefore $h_p=0$.
\end{itemize}

We have shown that $\sigma(f) = g$.  We then conclude as in the end of the proof of Theorem~\ref{thm:innertwist}.  
\end{proof}

\begin{example}\label{ex:sturmnotenough}
Consider the space \newformlink{24.2.f.a}.  There are two Galois-conjugate newforms with the same Nebentypus character. The Sturm bound is $8$, but the smallest $p\nmid N$ where the Fourier coefficients differ is $11$.  In particular, this shows that in the Hecke--Sturm bound (Proposition \ref{prop:heckesturm}) we cannot ignore the primes $p \mid N$.  
\end{example}

The virtue of Theorems \ref{thm:innertwist} and \ref{thm:justcheckcoprime} is that they give explicit criteria to certify inner twists, with care taken concerning primes dividing the level.  

\subsection{Computing inner twists}

We used Theorem~\ref{thm:justcheckcoprime} to compute the complete group of inner twists for all the modular forms in our dataset.
Specifically, we enumerate the finite set $X$ of Dirichlet characters $\psi$ satisfying condition (i) of Theorem~\ref{thm:justcheckcoprime} for which $\chi\psi^2$ is conjugate to $\chi$.  Note that the set $X$ does not depend on $f$ or its coefficient field, only the character $\chi$ and level $N$.  We then determine the subset of $X$ that satisfy conditions (iii) and (iv) for some $\sigma\in \Gal(\widetilde{K}_f)$ as follows:

\begin{enumerate}
\item  We first remove from $X$ all characters $\psi$ for which there is a prime $p\le \Sturm(k,N)$ not dividing $N$ such that $a_p(f)\psi(p)$ is not conjugate to $a_p(f)$; this is accomplished by comparing the minimal polynomials of $a_p(f)\psi(p)$ and $a_p(f)$.
\item For all remaining $\psi\in X$, set $T\colonequals\Gal(\widetilde{K}_f)$ and for successive primes $p\le \Sturm(k,N)$ with $p\nmid N$, replace $T$ with $\{\sigma\in T:\sigma(a_p(f)) = a_p(f)\psi(p)\}$, stopping if $T$ becomes empty.  This yields a list of candidate inner twists $(\psi,\sigma)$ containing $\InnTw(f)$.
\item Finally, for each candidate $(\psi,\sigma)$ we check whether (iv) holds; if so then Theorem~\ref{thm:justcheckcoprime} implies that $(\psi,\sigma)$ is an inner twist of $f$.
\end{enumerate}

As shown by Example~\ref{ex:sturmnotenough}, the third step above is potentially necessary, but in our computation we never encountered a case where a candidate inner twist that survived step (2) was discarded in step (3).

\begin{rmk}
The \magma{} function \texttt{InnerTwists} implements a weaker form of Theorem \ref{thm:innertwist}.  It requires checking eigenvalues up to the Sturm bound for level $\lcm(N,\cond(\psi)^2,\cond(\psi)\cond(\chi))$, and it performs eigenvalue comparisons using complex approximations that do not guarantee a rigorous result.
Indeed, even when the optional parameter \texttt{Proof} is set to \texttt{True}, \magma{} version 2.24-7 displays the following message:
\bigskip

\parbox{457pt}{
\texttt{WARNING: Even if Proof is True, the program does not prove that every 
twist return\-ed is in fact an inner twist (though they are up to 
precision 0.00001).}
}

\end{rmk}

\section{Weight one} \label{sec:weight1}

Modular forms of \weightone{} are of particular interest due to the connection with Artin representations, provided by a theorem of Deligne and Serre \cite{DeligneSerre74}: one can associate to each \weightone{} newform $f$ an odd irreducible 2-dimensional Galois representations $\rho_f\colon G_\Q\to \GL_2(\C)$ for which $L(f,s)=L(\rho_f,s)$ (recall that a Galois representation is \defi{odd} if complex conjugation has determinant $-1$).  Following the proof of Serre's conjecture by Khare and Wintenberger \cite{KW}, we now know that the map $f\mapsto \rho_f$ is in fact a bijection.
This connection allows one to attach several additional arithmetic invariants to \weightone{} newforms that we would like to compute, including:
\begin{itemize}
\item The \defi{projective image} of $\rho_f$ in $\PGL_2(\C)$, which by Klein's classification is isomorphic to either $D_n$ (dihedral of order $2n$, including $D_2\colonequals\Z/2\Z\times\Z/2\Z$), or one of the exceptional groups $A_4$ (tetrahedral), $S_4$ (octahedral), or $A_5$ (icosahedral).
\item The \defi{projective field} of $\rho_f$: the fixed field of the kernel of $G_\Q\overset{\rho_f}{\longrightarrow}\GL_2(\C)\twoheadrightarrow\PGL_2(\C)$.
\item The \defi{Artin image} of $\rho_f$: the finite group $\rho_f(G_\Q)\leq \GL_2(\C)$.
\item The \defi{Artin field} of $\rho_f$: the fixed field of $\ker\rho_f$, with Galois group isomorphic to $\rho_f(G_\Q)$.
\end{itemize}

One can also consider the projective representation $\bar\rho_f\colon G_\Q\to \PGL_2(\C)$ induced by $\rho_f$ as an invariant in its own right: it uniquely determines the \defi{twist class} of $f$.  Two newforms $f$ and $g$ are said to be \defi{twist equivalent} if $g=f\otimes \psi$ for some Dirichlet character $\psi$, and in weight $1$ this occurs if and only if $\bar\rho_f=\bar\rho_g$.

\subsection{Computational observations}\label{subsection:wt1dihedral}

The Deligne--Serre theorem also has important computational implications.  In the typical case where $\rho_f$ is a dihedral representation (meaning that its projective image is dihedral), the Artin $L$-function $L(\rho_f,s)$ is also the Weber $L$-function $L(\omega,s)$ of a ray class character $\omega$ of the quadratic field $K$ fixed by the preimage of $C_n\subseteq D_n\simeq \bar\rho_f(G_\Q)$.  (For $n=2$ there are three choices for $C_2\subseteq D_2$; we can use any one of the three.)
The quadratic field $K$ and the ray class character $\omega$ necessarily satisfy
\begin{equation}
\left|d_K\right|\Nm(\cond(\omega))=\cond(\rho_f)=N,
\end{equation}
where $d_K$ is the discriminant of $K$ and $N$ is the level of $f$.  In order to obtain an odd representation $\rho_f$ we also require that if $K$ is a real quadratic field then the modulus for $\omega$ should include exactly one of the real places of $K$.

For any given level $N$, it is straightforward to enumerate all quadratic fields $K$ of discriminant $d_K \mid N$, all $\mathcal O_K$-ideals of absolute norm dividing $N/\left|d_K\right|$, and all ray class characters $\omega$ of $K$ for the modulus with finite part $I$ and infinite part compatible with an odd representation.  This makes it feasible to explicitly compute Fourier expansions of all dihedral newforms of level $N$ to any desired precision; to compute $a_p(f)$ for $p\nmid N$ this simply amounts to evaluating the corresponding ray class character $\omega$ at the prime ideals of $\mathcal O_K$ above $p$.

\pari{} contains extensive support for computing with ray class characters that are particularly efficient in the case of quadratic fields.
We used this to compute all dihedral newforms of level $N\le \numprint{40000}$ with Fourier coefficients $a_n(f)$ computed for $n\le 6000$ (well past the Sturm bound).
This yielded a total of \numprint{572462} dihedral newforms, corresponding to \numprint{14634052} embedded newforms.  The largest dimension we found was 2818, which arises for a dihedral newform of level 39473, and the largest projective image we found was $D_{2846}$ for a newform of level \numprint{39851}.

These computations go far beyond the extent of our database described in \S\ref{sec:dataextent}, which only covers levels $N\le 4000$ in \weightone{}.  For comparison, the largest dimension arising for $N\le 4000$ is 232 and the largest projective image is $D_{285}$.
The reason for this discrepancy is that while it is computationally very easy to compute dihedral newforms, to obtain a complete enumeration of all the newforms in a given \weightone{} newspace, one must also enumerate the tetrahedral, octahedral, and icosahedral newforms, which is more difficult---particularly in the icosahedral case.  Interestingly, the main difficulty often lies not in enumerating these exceptional newforms, but in verifying that one has actually found them all.  In contrast to the case $k>1$ where there are well known dimension formulas, while there are computational tricks that work well in special cases, to our knowledge no efficient method for computing $\dim \Sknew{1}(N)$ for general $N$ is currently known.

\subsection{Classifying the projective image}

The \pari{} function \verb|mfgaloistype| can be used to classify the projective image, but given that we actually computed the projective field in every case (which of course determines the projective image), we did not exploit this feature.

\begin{remark}
Buzzard--Lauder \cite{BuzzardLauder} describe an approach to classifying the projective image by computing projective orders of elements that they applied to all \weightone{} newforms of level up to 1500. They note in their paper that their approach relies on the convenient fact that there are no \weightone{} newforms of level $N\le 1500$ with projective image $A_4$ whose coefficient field contains $\Q(\sqrt{5})$.  Five such examples arise in our dataset, the first of which is \newformlink{2299.1.w.a}.
\end{remark}

\subsection{Computing the projective field}
Our strategy for computing the projective field is to exhaustively compute a complete set of candidates and then rule out all but one.  As noted in \secref{subsection:wt1dihedral}, we can effectively determine all the dihedral forms at each level, so we know in advance exactly which forms are dihedral (and the exact order of the projective image in each of these cases).  In cases where a dihedral image has moderate degree---less than 100, say---it is feasible to use the ray class field functionality in \pari{} to compute the projective field.  This notably includes all of the dihedral projective fields whose distinguished quadratic subfield is real: the largest such example in our database is \newformlink{2605.1.bd.a} with projective image $D_{40}$.

The dihedral fields in which the distinguished subfield is imaginary quadratic can be much larger: the largest example  \newformlink{3997.1.cz.a} has projective image $D_{285}$.  In these cases, we exploit the fact that every dihedral field whose distinguished quadratic subfield is imaginary can be realized as a subfield of a ring class field that can be explicitly computed using the theory of complex multiplication.  There is a well-developed theory for efficiently computing these ring class fields, even in cases where the degree may be in the millions, motivated by applications to cryptography and elliptic curve primality proving (the CM method for constructing elliptic curves over finite fields).

Given a dihedral \weightone{} newform $f\in \Sknew{1}(N,\chi)$ with dihedral image $D_n$ and distinguished imaginary quadratic field $K$, there is a finite set of possible suborders $\mathcal O$ of $\mathcal O_K$ and conductors $c$ such that the projective field of $f$ arises as a cyclic degree-$n$ extension of $K$ of conductor $c$ contained in the ring class field $K$ of $\mathcal O$.  The enumeration of these dihedral fields was achieved using an algorithm based on the techniques developed by Enge-Sutherland \cite{EngeSutherland} and Sutherland  \cite{Sutherland:HCPbyCRT,Sutherland:AcceleratingCM} that will be described in a forthcoming paper.

Having enumerated a complete list of candidate fields $L\colonequals \Q[x]/(g_L(x))$, for successive primes $p\nmid N$ we can compute the order of $\rho_f(\Frob_p)$ in $\PGL_s(\C)$ by determining the positive integer $n$ for which $a_p(f)^2/\chi(p) = \zeta_n + \zeta_n^{-1}+2$ and compare this to the inertia degree of the primes above $p$ in $\mathcal O_L$.  This will eventually eliminate all but one candidate field, since the sequence of inertia degrees uniquely determines a Galois number field, and in practice this happens very quickly.
To accelerate the computation we precompute defining polynomials for the real cyclotomic fields we may encounter and use $p$ coprime to the discriminants of the defining polynomials $g_L$ so that we can compute the inertia degree as the degree of the irreducible factors of $g_L(x)$ in $\F_p[x]$.

For the non-dihedral projective images we used the methods of Cohen--Diaz y Diaz--Olivier \cite{CDO1,CDO2} to enumerate all $A_4$ and $S_4$ fields unramified outside a given set of primes, and for the $A_5$ fields we used existing tables of fields in the Jones--Roberts database and the LMFDB combined with a targeted Hunter search for some missing cases, as described by Jones--Roberts \cite{targetedhunter}.  This allowed us to construct complete lists of candidate fields for each non-dihedral \weightone{} form from which we then ruled out all but one candidate by comparing orders of Frobenius elements with inertia degrees as described above.

\subsection{Computing the Artin image, the Artin field, and the associated Artin representation}
As of January 2020 the LMFDB contained 5116 odd 2-dimensional Artin representations of conductor $N\le 4000$, all of which we were able to uniquely match to a corresponding  newform of \weightone{}.  For each of these Artin representations the LMFDB provides the Artin image, the Artin field, and a complete description of the Artin representation given values on each conjugacy class of Frobenius elements.
We were also able to compute the Artin image and Artin field for 833 additional \weightone{} newforms that are twists of a \weightone{} newform for which we know the corresponding Artin representation by taking the compositum of the known Artin field with an appropriate cyclotomic field.

There is work in progress to add as many of the Artin representations corresponding to the remaining \numprint{14190} \weightone{} newforms as possible; these will be linked to the corresponding weight~1 newforms as they become available.

\subsection{Interesting and extreme behavior} \label{sec:wt1cool}

Weight one modular forms behave rather differently than those of higher weight.  As seen in \secref{sec:weight1}, one important invariant of \weightone{} forms is the projective image of the associated Galois representation.
We will discuss some forms with dihedral projective image first.

Hecke also constructed \weightone{} modular forms starting from imaginary quadratic fields with odd class number at least 3. The first examples of such fields come from $\Q(\sqrt{-23})$, $\Q(\sqrt{-31})$, $\Q(\sqrt{-39})$, and the corresponding modular forms are the three smallest level \weightone{} newforms; these have labels   \newformlink{23.1.b.a}, \newformlink{31.1.b.a} and \newformlink{39.1.d.a}, respectively. \cite{hecke} 
\begin{exm} \label{exm:D2}
The last of these, \newformlink{39.1.d.a}, is the $D_2$ form of lowest level and has CM by both $\Q(\sqrt{-3})$ and $\Q(\sqrt{-39})$, and RM by $\Q(\sqrt{13})$. 
This form appears in work of Darmon--Lauder--Rotger \cite[Example 2.5]{DLR:def}.
\end{exm}

The first examples of newforms with RM but no CM occur in level 145 with \newformlink{145.1.f.a} (RM by $\Q(\sqrt 5)$, \cite[Example 3.3]{DLR:def}, \cite[Example 4.1]{DLR:stark}) and \newformlink{145.1.h.a} (RM by $\Q(\sqrt{29})$, \cite[Example 1.2]{DLR:over}).  

The problem of constructing \weightone{} forms whose projective image is not dihedral was considered by Tate and Serre in the 1970s. These forms are sometimes called \emph{non-banal} or \emph{exotic}. Such forms divide up into 3 cases based on their projective image, which can be one of $A_4, S_4, A_5$: the forms are then known as tetrahedral, octahedral and icosahedral, respectively.

Tate together with his students, Flath, Kottwitz, Tunnell, and Weisinger, and additionally Atkin, exhibited a form of level $133$, with projective image $A_4$ 
described in a letter to Atkin \cite[p.~713]{Tate:letter}; this form is \newformlink{133.1.m.a} in our database. The smallest level example is actually in level 124, given by \newformlink{124.1.i.a}.

In the octahedral case, the smallest level example is in level $4\cdot 37 = 148$ with label \newformlink{148.1.f.a}; this newform is discussed by Buzzard \cite[\S 2.3]{BuzzardWtOne} and Darmon--Lauder--Rotger \cite[Example 5.6]{DLR:stark}.

Many modular forms previously considered in the literature with interesting Galois representations can now be found in our database.
Ogasawara \cite{Ogasawara} takes the mod-3 Galois representations attached to certain elliptic curves and constructs a $\GL_2(\F_3)$ Artin representation: for example, the elliptic curve of conductor 11 with label 
\href{https://www.lmfdb.org/EllipticCurve/Q/11/a/3}{\textsf{11.a3}} is used and the corresponding octahedral modular form of \weightone{} over $\Q(\sqrt {-2})$ is constructed.  Using the $q$-expansion coefficients given there, we can use the trace search functionality to locate a (unique) matching form in our database: \newformlink{3267.1.b.d}. We then verify that it has the right Artin field: a degree 8 extension over which \href{https://www.lmfdb.org/EllipticCurve/Q/11/a/3}{\textsf{11.a3}} gains 3-torsion.

Buhler \cite{BuhlerIcos,BuhlerMono} constructs the icosahedral Galois representation of level 800, labeled \newformlink{800.1.bh.a}. Kiming--Wang \cite{KimingWang} gave several more instances of icosahedral newforms of \weightone{} with characters of order 2, showing their existence in order to verify the Artin conjecture in these cases.  The new database now contains all but one of these: \newformlink{2083.1.b.b}, \newformlink{1948.1.b.a}, \newformlink{3004.1.b.a}, \newformlink{3548.1.d.a}, \newformlink{3676.1.c.a}, \newformlink{2336.1.c} (two newforms).
The only newspace discussed in loc. cit. with level outside our range would have label \newformlink{6176.1.b}.
The database also contains the icosahedral newforms \newformlink{1376.1.r.a}, \newformlink{2416.1.p.a}, \newformlink{3184.1.t.a}, \newformlink{3556.1.ba.a} and \newformlink{3756.1.q.b} which were all shown to satisfy Artin's conjecture by Buzzard--Stein \cite{BuzzardStein}.   The proof of Serre's conjecture \cite{KW} established Artin's conjecture for all odd irreducible 2-dimensional representations, including all of the icosahedral cases.
The smallest level example of an icosahedral newform is \newformlink{633.1.m.b}.

Constructing exotic forms of prime level with specific projective image is also a much studied problem.  Such forms do not exist in the tetrahedral case \cite[Thm.~7, p.~245]{Serre-wt1},
leaving only octahedral and icosahedral forms with the possibility of prime level.

In the octahedral case the smallest prime level is 229, and the space of newforms \newformlink{229.1.d} splits into two Galois orbits, (see Serre  \cite[p.~265]{Serre-wt1}). The second smallest level is 283, where we have the newform \newformlink{283.1.b.b} that appears also in work of Serre \cite{Serre:jordan}.

In the icosahedral case, we have seen above the first example of such a form: the one with level $2083$ of Kiming--Wang.
In fact the query for forms with projective image $A_5$ shows that there are 4 such forms with prime level $\le 4000$:  \newformlink{2083.1.b.b}, \newformlink{2707.1.b.b}, \newformlink{3203.1.b.a}, \newformlink{3547.1.b.c}.  It is conjectured that these forms are rare.  

\begin{conj} For any \(\epsilon>0\), the number of exotic newforms of prime level \(N\) is \(O_{\epsilon}\left(N^{\epsilon}\right)\).
\end{conj}

Bhargava--Ghate \cite{BhargavaGhate} have shown an averaged version of this conjecture in the octahedral case.

\end{document}